\documentclass{article}
\usepackage{amsmath,amsfonts,amsthm,amscd,amssymb,pstricks}
\newtheorem{thm}{Theorem}[section]
\newtheorem{pro}[thm]{Proposition}
\newtheorem{cor}[thm]{Corollary}
\newtheorem{lem}[thm]{Lemma}
\newtheorem{conjecture}[thm]{Conjecture}

\newtheorem{defn}[thm]{Definition}
\newcommand{\dime}{\operatorname{dim}}

\newcommand{\degr}{\operatorname{deg}}

\newcommand{\trace}{\operatorname{Trace}}
\newcommand{\rad}{\operatorname{rad}}
\newcommand{\Hom}{\operatorname{Hom}}
\newcommand{\Endo}{\operatorname{End}}

\begin{document}

\author{Jonathan Pakianathan and Nicholas F. Rogers}
\title{Lie algebras and Higher Torsion in p-groups
}
\maketitle

\begin{abstract}
The primary aim of this paper is to study exceptional torsion in the integral
cohomology of a family of $p$-groups associated to $p$-adic Lie algebras.
A spectral sequence $E_r^{*,*}[\mathfrak{g}]$ is defined for any
Lie algebra $\mathfrak{g}$ which models the Bockstein spectral sequence of the corresponding
group in characteristic $p$.

This spectral sequence is then studied for complex semisimple Lie algebras like $\mathfrak{sl}_n(\mathbb{C})$
and the results there are transferred to the corresponding $p$-group via the intermediary arithmetic
Lie algebra defined over $\mathbb{Z}$. The results obtained this way for a fixed Lie algebra scheme like
$\mathfrak{sl}_n(-)$ hold in a range in the corresponding Bockstein spectral sequence for all but finitely
many primes depending on the chosen range.

Over $\mathbb{C}$, it is shown that $E_1^{*,*}[\mathfrak{g}]=H^*(\mathfrak{g},U(\mathfrak{g})^*)=H^*(\Lambda BG)$
where $U(\mathfrak{g})^*$ is the (filtered) dual of the universal enveloping algebra of $\mathfrak{g}$ equipped with
the dual adjoint action and $\Lambda BG$ is the free loop space of the classifying space of an associated
compact, connected real form Lie group $G$ to $\mathfrak{g}$.

When passing to characteristic $p$, in the corresponding Bockstein spectral sequence,
a char $0$ to char $p$ phase transition is observed.
For example, it is shown that the algebra $E_1^{*,*}[\mathfrak{sl}_2[\mathbb{F}_p]]$
requires at least 17 generators unlike its characteristic
zero counterpart which only requires two.


\end{abstract}

\tableofcontents

\section{Introduction}

Let $\mathbf{k}$ be a PID throughout this paper
(in a lot of cases we'll be using
a field but sometimes we need $k=\mathbb{Z}$ or other decent rings and still use the term algebra through abuse of notation).
By a $\mathbf{k}$-Lie
algebra $\mathfrak{g}$ we mean a free $\mathbf{k}$-module
equipped with a bilinear bracket
$[-,-]: \mathfrak{g} \otimes_{\bf{k}} \mathfrak{g} \to \mathfrak{g}$
satisfying the Jacobi identity:
$$
[[x,y],z] + [[y,z],x] + [[z,x],y]=0
$$
for all $x, y ,z \in \mathfrak{g}$. We refer to the dimension of
$\mathfrak{g}$ as the rank of the underlying free $\mathbf{k}$-module, and
will usually assume this is finite unless otherwise specified.

Lie algebras arise in a variety of contexts: over $\mathbb{R}$ as the tangent
space at the identity of a Lie group or as the collection of smooth
vector fields on a smooth manifold,  over $\mathbb{Q}$ as the rational
homotopy Lie algebra of a space, over $\mathbb{Z}$ as the Lie algebra
associated to the descending central series of a residually
nilpotent group and over the $p$-adic integers $\mathbb{Z}_p$ or finite
fields in the theory of pro-$p$ groups (see \cite{DS}) 
and $p$-groups respectively. For
basic facts on Lie algebras and their cohomology used in this paper
see \cite{Bo} or \cite{GG}. For the basic facts on the cohomology 
of groups used in this paper see \cite{Benson} or \cite{Brown}.

Due to our primary motivation let us give a few more details on the
correspondence between $p$-groups and Lie algebras. 
Consider the formal power series for $e^x$ in
$\mathbb{Q}[[x]]$, i.e., $e^x = \sum_{n=0}^{\infty} \frac{x^n}{n!}$,
it is a trivial formality to check $e^0=1$ and $e^{-x} e^x = 1$
and that if $x$ and $y$ commute then $e^{x+y}=e^x e^y$.
However working in the (completed) free associative algebra over
$\mathbb{Q}$ on two
variables one finds that if $x$ and $y$ don't commute, then one has
the fundamental Baker-Campbell-Hausdorff identity:
$$
e^{x}e^{y}=e^{x+y+\frac{1}{2}[x,y] + \frac{1}{12}[x,[x,y]] - \frac{1}{12}[y,[x,y]] + \dots } = e^{x+y+I}
$$
where $I$ is an infinite sum of iterated brackets of $x$ and $y$
with rational coefficients and where $[a,b]=ab-ba$. Due to this,
whenever $\mathbb{Q} \subseteq \mathbf{k}$, we can associate a formal
group $e^{\mathfrak{g}}$ to every $\mathbf{k}$-Lie algebra $\mathfrak{g}$.
When $k=\mathbb{C}$ and the usual Euclidean metric is used, the defining
series converge and one recovers the classical ``Exponential-Log''
correspondence between Lie algebras and Lie groups.

Over the $p$-adic integers $\mathbb{Z}_p$ with the $p$-adic metric, and
$p$ an odd prime, $e^{px}$ can be shown to converge for every $x \in \mathbb{Z}_p$.
Thus if we let $\mathfrak{g}=\mathfrak{gl}_n(\mathbb{Z}_p)$ be the Lie algebra
of $n \times n$, $p$-adic matrices with the bracket $[\mathbb{A},\mathbb{B}]
=\mathbb{A}\mathbb{B}-\mathbb{B}\mathbb{A}$, one has that
$e^{p\mathfrak{g}}$ is a group. In fact since $e^{p\mathbb{A}} \equiv \mathbb{I} \text{ mod } p$, it
is not hard to show that $e^{p\mathfrak{g}}$ is the group of
$p$-adic matrices which are congruent to the identity matrix mod $p$,
which is called the $p$-congruence subgroup $\Gamma_{\mathfrak{gl}_n}$ of $GL_n(\mathbb{Z}_p)$.
(See exponential-log correspondence in \cite{Ro}).
From this formal group, one can form a tower of finite $p$-groups
$\Gamma_{\mathfrak{g},k}=e^{p\mathfrak{g}}$ mod
$p^{k+1}$ for $k=1,2,3,\dots$ for which
$e^{p\mathfrak{g}}$ is the inverse limit. It is not hard to see
that $\Gamma_{\mathfrak{g},1}$ is elementary abelian and can be identified
naturally with the residue Lie algebra of the $p$-adic Lie algebra
$\mathfrak{g}$ i.e.,
$\mathfrak{g}/p\mathfrak{g}=\mathfrak{g} \otimes \mathbb{F}_p$.
In fact $e^{p\mathbb{A}}=\mathbb{I} + p\mathbb{A} \text{ mod } p^2$
so $e^{p\mathbb{A}} \text{ mod } p^2$ corresponds to $\mathbb{A} \text{ mod }
p$.
A little more analysis shows that we have a central short exact
sequence
$$
0 \to \mathfrak{g} \otimes \mathbb{F}_p \to
\Gamma_{\mathfrak{g},2} \to \mathfrak{g} \otimes \mathbb{F}_p \to 0.
$$
and so $\Gamma_{\mathfrak{g},2}=\Gamma_{\mathfrak{g}} \text{ mod } p^3$ is a $p$-group
given by a central extension of an elementary abelian $p$-group by itself.

More generally, for $p$ odd, 
one can show given any $\mathbb{F}_p$-Lie algebra $\mathfrak{L}$
(whether it lifts to the $p$-adics or not)
there exists a unique $p$-power exact sequence
$$
0 \to \mathfrak{L} \to G(\mathfrak{L}) \to \mathfrak{L} \to 0
$$
where $G(\mathfrak{L})$ is a $p$-group of order $p^{2\dime (\mathfrak{L})}$.
In fact the construction $\mathfrak{L} \to G(\mathfrak{L})$
is part of a covariant functor from the category of
$\mathbb{F}_p$-Lie algebras
to the category of $p$-groups. (See \cite{BrP}). For $p=2$ certain phenomena involving quadratic forms arise 
and the correspondence needs to be modified (see \cite{PY1} and \cite{PY2}). Due to this, throughout this paper, 
$p$ denotes an odd prime.

In \cite{BrP} it was shown that if $n$ is the dimension of $\mathfrak{L}$, 
then
$$
H^*(G(\mathfrak{L}),\mathbb{F}_p) \cong \Lambda^*(\mathfrak{L}^*)
\otimes Poly(\mathfrak{L}^*) \cong \Lambda(x_1,\dots,x_n) \otimes
\mathbb{F}_p[y_1,\dots,y_n]
$$
where $\mathfrak{L}^*$ is the dual
of $\mathfrak{L}$. Here $\Lambda^*(V)$ denotes the exterior algebra
on the vector space $V$ where $V$ is given grading 1
while $Poly(V)=\mathbb{F}_p[V]$ denotes the polynomial
algebra on $V$ where $V$ is given grading 2.

Furthermore the differential provided by
the Bockstein $\beta$ was shown in \cite{BrP} to be the same as the differential
that makes this algebra into the Koszul resolution computing
$H^*(\mathfrak{L}, Poly(ad^*))$
where $ad^*$ is the dual adjoint representation
of $\mathfrak{L}$ on $\mathfrak{L}^*$ and the action of $\mathfrak{L}$
is extended to $Poly(ad^*)$ by declaring it to act via derivations.

Thus the $\beta$-cohomology of these $p$-groups was shown to be
$H^*(\mathfrak{L}, Poly(ad^*))$ and this is interesting as this cohomology
gives the 2nd page of the Bockstein spectral sequence used
to analyze the integral cohomology $H^*(G(\mathfrak{L});\mathbb{Z})$.
Explicitly if we decompose the polynomial algebra $Poly(ad^*)$ into
its homogeneous components (Hodge decomposition) $Poly(ad^*)=\oplus_{n=0}^{\infty} S^n$ we have
$$
B^*_2 = \bigoplus_{n=0}^{\infty} H^*(\mathfrak{L}, S^n)
$$
and so the higher torsion in the integral cohomology of the $p$-group
$G(\mathfrak{L})$ is reflected in these Lie algebra cohomology groups.

In fact an analysis of the integral cohomology of
$G(\mathfrak{sl}_2(\mathbb{F}_p))$ using these techniques was used to provide
a counterexample in \cite{Pk} to a
conjecture of Adem (see \cite{Adem}) at odd primes.

In this paper, we study a fundamental differential graded algebra (dga)
associated to any Lie algebra $\mathfrak{L}$ given by the
Koszul resolution whose underlying algebra is $\Lambda^*(\mathfrak{L}^*)
\otimes Poly(\mathfrak{L}^*)$ with a differential such that the
cohomology calculates $H^*(\mathfrak{L}, Poly(ad^*))$.
While our motivation is to calculate higher torsion in the integral
cohomology of $p$-groups and hence primarily concerns Lie algebras over
$\mathbb{F}_p$, our basic technique is to relate these Lie algebras to Lie algebras defined over $\mathbb{Z}$ and to then compare these to the corresponding
Lie algebras defined over $\mathbb{Q}$ or $\mathbb{C}$ where classical
results can be appealed to. This translates then to results
which hold for ``all but finitely many primes'' when studying the
corresponding $p$-groups. This philosophy was motivated by previous work with A. Adem
in \cite{AP}, though the implementation is quite different here.

The dga we construct $E[\mathfrak{L}] = \Lambda^*(\mathfrak{L}^*)
\otimes Poly(\mathfrak{L}^*)$ comes equipped with a spectral sequence
that is constructed in the appendix and functions in some sense as
an algebraic classifying complex for the co-Lie algebra $\mathfrak{L}^*$
much like $EG$ does for a Lie group $G$.

Interestingly enough, this seemingly specialized dga carries a lot of
structure. Over fields of characteristic zero and for nondegenerate
Lie algebras (Killing form nondegenerate), one has an isomorphism
of $\mathfrak{L}$-modules between $ad$ and $ad^*$. Furthermore using
the Poincare-Birkoff-Witt theorem one can identify $Poly(ad)$ with
$U(\mathfrak{L})$ as filtered modules
and $Poly(ad^*)$ with the filtered dual $U(\mathfrak{L})^*$ as graded modules where
$U(\mathfrak{L})$ is the universal enveloping algebra of $\mathfrak{L}$
equipped with the {\bf adjoint} action. (Throughout this paper $U(\mathfrak{L})$
and its filtered dual $U(\mathfrak{L})^*$ will always be equipped with the adjoint action and not the left
or right translation action! The filtered dual consists of functionals on the filtered universal enveloping 
algebra with support in one of the filtration levels, i.e., functionals that vanish on elements of high enough "degree". This is analogous to 
the graded dual construction. Note it is a proper subspace of the vector space dual as it does not allow for functionals which do not 
vanish on elements of arbitrarily high "degree".)
Thus $E[\mathfrak{L}]$ as a dga is nothing
other than the canonical Koszul complex computing
$H^*(\mathfrak{L}, U(\mathfrak{L})^*)$.
In this context it is
important to point out that over fields of characteristic zero,
$H^*(\mathfrak{L}, U(\mathfrak{L})^*)$
can be identified with the cohomology of the free loop space of $BG$ in the case that $G$ is a compact, connected
Lie group with Lie algebra a real form for
$\mathfrak{L}$ as we will see in this paper in the case that $G$ is semisimple as a
byproduct of our analysis of this spectral sequence.

The cohomology of the free loop space $\Lambda M$ on a manifold $M$ has been an object of intense study in the
last decade due to the existence of a ``string multiplication'' introduced by Chas and Sullivan
(See \cite{CS}) and the structure of a Batalin-Vilkovisky algebra with string theoretic
interpretations, (see \cite{GW}, \cite{Ma}).

Furthermore $H^0(\mathfrak{L},U(\mathfrak{L}))$ consists
of the central elements of $U(\mathfrak{L})$, the so called ``Casimir
ring'' which is very important in the representation theory of
$\mathfrak{L}$. We will see that their dual elements, the ``dual Casimirs'' in
$H^0(\mathfrak{L}, U(\mathfrak{L})^*)$ play an important role in the higher
torsion of $p$-groups.

%The Casimir elements themselves play an important role in the physical
%applications of the representations of Lie algebras to spin systems ($\mathfrak{sl}_2$)
%and elementary particle theory ($\mathfrak{sl}_3$ in the ``eight-fold path''). Some comments
%about these pictures will be made at the relevant points of the paper.
%
%Finally one would be remiss if one
%did not point out the similarity to the Eichler-Shimura theorem
%which relates certain modular forms to the cohomology of $SL_2(\mathbb{Z})$
%with coefficients in the polynomial ring on the canonical complex
%2-dimensional representation of $SL_2(\mathbb{Z})$. (See \cite{S}, \cite{FTY}.)

 As one of the main tools of this paper,
we construct a spectral sequence whose $E_0$ term
is this dga and whose $E_{\infty}$ term is the cohomology of a point
which gives a lot of structure to the underlying dga. Though this
spectral sequence is defined over any coefficient $\mathbf{k}$, over
$\mathbb{F}_p$ it models the associated Bockstein Spectral sequence
of the corresponding $p$-group mentioned above. More precisely
we show that the $(E_0,d_0)$ term of this spectral sequence is
identical to the $(B_1,\beta)$ term of the Bockstein spectral sequence
of the $p$-group $G(\mathfrak{L})$. Though both spectral sequences converge
to the cohomology of a point, it is unknown if the higher pages coincide.
Nevertheless,
using the structure theorems available for both spectral sequences, and
that the first one can be used to compare the situation for the
$\mathbb{F}_p$-Lie algebra with that of corresponding integral and
complex Lie algebras, one can obtain results. Throughout this paper
$(E_r, d_r)$ will always refer to the algebraic spectral sequence
constructed in the appendix and $(B_r, \beta_r)$ will denote
the Bockstein spectral sequence of the group $G(\mathfrak{L})$
in the case that $\mathfrak{L}$ is a $\mathbb{F}_p$-Lie algebra.

The spectral sequence $E_r$ carries a lot of information, indeed even over
$\mathbb{C}$ it can be used to recover Harish-Chandra's calculation
of the Casimir ring of $sl_2(\mathbb{C})$ among many other results.
(We discuss the physical relevance of this within the paper.)
In this introduction we quote just this classical corollary for simplicity
to show the main frame of the arguments. In some of these statements
the ``Hodge decomposition'' $Poly(ad^*) = \oplus_{n=0}^{\infty} S^{n}$
of the polynomial algebra into its homogeneous components is used and $S^0$
always denotes the base ring $\mathbf{k}$ with trivial $\mathfrak{L}$
action.

\begin{thm}[Harish-Chandra calculation for $\mathfrak{sl}_2(\mathbb{C})$]
$$
H^*(\mathfrak{sl}_2(\mathbb{C}), Poly(ad^*)) \cong
H^*(\mathfrak{sl}_2(\mathbb{C}), U(\mathfrak{sl}_2(\mathbb{C}))^*)
\cong \Lambda^*(u) \otimes \mathbb{C}[\kappa]
$$
where $\kappa = H^2+EF \in H^0(\mathfrak{sl}_2(\mathbb{C}),S^2)$
is (a nonzero scalar multiple) of the Killing form
(an example of a ``dual Casimir'' element)
and $u \in H^3(\mathfrak{sl}_2(\mathbb{C}),S^0)$ is the volume form. $\kappa$ is dual to
the central Casimir element $H^2 + 2EF + 2FE$ in $H^0(\mathfrak{sl}_2,U(\mathfrak{sl_2}))$ which
corresponds to ``total angular momentum squared'' in spin systems. 
\end{thm}

Recall that a Casimir element is an element in the center of $U(\mathfrak{L})$ or equivalently an 
ad-invariant element of $U(\mathfrak{L})$. A dual Casimir element is an ad-invariant element of the 
filtered dual $U(\mathfrak{L})^*$.

As mentioned before, this calculation when used in conjunction with the spectral sequence
we construct can be used to derive results for the corresponding
Lie algebras $\mathfrak{sl}_2(\mathbb{F}_p)$ after passing through $\mathfrak{sl}_2(\mathbb{Z})$.
To state these results let us note
that for $\mathbb{F}_p$-Lie algebras $\mathfrak{L}$ we have
$(E_0,d_0)=(B_1,\beta)=\Lambda^*(\mathfrak{L}^*) \otimes
Poly(\mathfrak{L}^*)$ where $(B_r,\beta_r)$ is the Bockstein spectral
sequence of the $p$-group $G(\mathfrak{L})$. We will call the polynomial
degree of a homogeneous element of this algebra its ``Hodge degree''.

This leads to the following:

\begin{thm}[Higher torsion in $G(\mathfrak{sl}_2(\mathbb{F}_p))$]
Let $G(\mathfrak{sl}_2(\mathbb{F}_p))$ be the kernel of the reduction homomorphism
$SL_2(\mathbb{Z}/p^3\mathbb{Z}) \to SL_2(\mathbb{F}_p)$. Then if $B^*$ denotes
the Bockstein spectral sequence for $G(\mathfrak{sl}_2(\mathbb{F}_p))$ we have:

$$
E_1^{*,*}=B_2^*=H^*(\mathfrak{sl}_2(\mathbb{F}_p),Poly(ad^*))
\sim
\Lambda^*(u) \otimes \mathbb{F}_p[\kappa]
$$
where $\sim$ denotes isomorphism in the range of Hodge degree $\leq N$
for all but finitely many primes $p$ (depending on $N$).

Furthermore $B_3^*$ is {\bf finite dimensional} and we have
$B_3^* \sim \Lambda^*(u) \otimes \mathbb{F}_p[\kappa]$. In particular though
$B_3^*$ is {\bf always} finite dimensional, there is {\bf no fixed bound on its dimension
that holds for all primes $p$}.
\end{thm}

Let us try to explain this in a less technical manner. For a finite $p$-group
$G$ let $exp(G)$ be the exponent of $G$, i.e., the smallest positive
integer $n$ such that $g^n=e$ for all $g \in G$. Let $\bar{H}^*(G,\mathbb{Z})$ 
denote the reduced integral cohomology of $G$ and $e(G)$ denote 
its exponent. Let $e_{\infty}(G)$ denote the asymptotic exponent of $G$, i.e.,
the smallest positive integer such that
$e_{\infty}(G)\bar{H}^*(G, \mathbb{Z})$
is finite. It is known that
$$
exp(G) \text{ } \Big{|} \text{ } e_{\infty}(G) \text{ } 
\Big{|} \text{ } e(G) \text{ } \Big{|} \text{ } |G|
$$
and there exist $G$ such that $e_{\infty}(G) \neq e(G)$.
(See \cite{Pk}).
When $e_{\infty}(G) \neq e(G)$, we can define the highest dimension
of an element in the finite graded group
$e_{\infty}(G)\bar{H}^*(G,\mathbb{Z})$ as the
``exceptional dimension'' of $G$. Thus all torsion elements of exceptionally
high order lie at or below the exceptional dimension of $G$.

\begin{cor}
For all odd primes $p$, $e_{\infty}(G(\mathfrak{sl}_2(\mathbb{F}_p)))=p^2$
while $e(G(\mathfrak{sl}_2(\mathbb{F}_p)))=p^3$. Moreover by the calculations above,
for any $N$, for all but a finite number of primes $p$, the exceptional
dimension of $G(\mathfrak{sl}_2(\mathbb{F}_p))$ is bigger than $N$.
\end{cor}

Thus for every odd prime there are exceptional elements of order
$p^3$ in the reduced integral cohomology of $G(\mathfrak{sl}_2(\mathbb{F}_p))$, 
but there is no bound on the exceptional dimension that holds for all primes.
Thus by suitable choice of primes $p$, one can find elements of order
$p^3$ in as high a dimension as one likes. However, for any fixed prime $p$,
asymptotically the exponent of the integral cohomology of
$G(\mathfrak{sl}_2(\mathbb{F}_p))$ is always $p^2$.

For details on these calculations and their implications a more thorough and
leisurely development can be found in the paper itself. The basic idea
is as follows: \\
(1) For every simple Lie algebra over $\mathbb{C}$, a Cartan-Serre basis
can be taken to extract a corresponding integral Lie algebra. The
algebraic spectral sequences for these can be compared and, 
through characteristic zero techniques, computations of the required
dual Casimirs can be done. \\
(2) Since each piece in the Hodge decomposition of the dga of the
corresponding integral Lie algebra is of finite type, when looking at
a chunk corresponding to terms with Hodge degree $N$ or less, one can use
the universal coefficient theorems to say that for all but a finite number
of primes, the corresponding dga over $\mathbb{F}_p$ will look similar
to the one over $\mathbb{C}$. \\
(3) In every case though there is a breakdown in the dga over
$\mathbb{F}_p$ when the Hodge degree becomes close to $p-1$ and there
is generally a phase transition between characteristic zero behaviour and
characteristic $p$ behaviour. \\
(4) In many cases we find that the ``exceptional torsion'' is created
by the part which corresponds to the characteristic zero case and
the transition to characteristic $p$ kills this exceptional torsion and
is signaled by a divided power
issue in the dga for the integral Lie algebra. For example for $\mathfrak{sl}_2$
we prove the key identity $\kappa^{\frac{p-1}{2}} u = 0$ in
the $\mathbb{F}_p$-dga as the left hand side of the identity is $p$
times a generator in the corresponding $\mathbb{Z}$-dga. Detailed pictures
of the algebraic spectral sequence are available in the paper to show how this identity
helps cause a ``phase-transition'' between the char 0 and char $p$ behaviour
in the case of $\mathfrak{sl}_2$.

In the paper, all semisimple Lie algebras are studied, including $\mathfrak{sl}_n$ for $n > 2$,
and the exceptional Lie algebra $\mathfrak{g}_2$.
These are studied over $\mathbb{C}$,
$\mathbb{Z}$ and finite fields. Over $\mathbb{C}$ one finds the behaviour of $E_r^{*,*}[\mathfrak{g}]$
is like that of $EG$ which motivates us calling it the ``classifying spectral sequence for the Lie algebra''
in general.

\begin{thm}(Complex semisimple Lie algebras)
For any complex simple Lie algebra $\mathfrak{g}$ with corresponding compact form $\mathfrak{g}_{\mathbb{R}}$
and compact connected Lie group $G$ with Lie algebra $\mathfrak{g}_{\mathbb{R}}$ we have:
$$
E_1^{*,*}[\mathfrak{g}]=H^*(\mathfrak{g},U(\mathfrak{g})^*) \cong H^*(G,\mathbb{C}) \otimes H^*(BG, \mathbb{C})
\cong H^*(\Lambda BG, \mathbb{C}).
$$
where $\Lambda BG$ denotes the free loop space of the classifying space $BG$.
For example, for $\mathfrak{g}=\mathfrak{sl}_n(\mathbb{C})$ one has
\begin{align*}
\begin{split}
E_1^{*,*}=H^*(\mathfrak{sl}_n,U(\mathfrak{sl}_n)^*)
&\cong H^*(SU(n), \mathbb{C}) \otimes H^*(BSU(n),\mathbb{C}) \\
&\cong H^*(\Lambda BSU(n),\mathbb{C}) \\
&\cong
\Lambda^*(u_3,\dots,u_{2n-1}) \otimes \mathbb{C}[c_2,\dots,c_n] \\
&\cong \Lambda^*(u_3,\dots,u_{2n-1}) \otimes \mathbb{C}[\sigma_2,\dots,\sigma_n]
\end{split}
\end{align*}
where $c_i$ are the universal Chern classes and $\sigma_i$ are the adjoint invariant polynomial
functions on $\mathfrak{sl}_n$ given by the elementary symmetric functions on the eigenvalues expressed
in terms of the coefficients of the matrix. This invariant theory picture is used to obtain these results
and is explained completely in the relevant sections of the paper. In general, there are 3 core pictures
for $E_1^{*,0}=H^0(\mathfrak{g},U(\mathfrak{g})^*)$ discussed in this picture, (a) as elements dual
to central elements in $U(\mathfrak{g})$, i.e., as dual Casimirs, (b) as adjoint invariant polynomial
functions on $\mathfrak{g}$ and (c) as the cohomology algebra $H^*(BG,\mathbb{C})$.

As a final example, for $\mathfrak{g}=\mathfrak{sp}_{2n}(\mathbb{C})$ one has
\begin{align*}
\begin{split}
E_1^{*,*}=H^*(\mathfrak{sp}_{2n}(\mathbb{C}),U(\mathfrak{sp}_{2n}(\mathbb{C}))^*) &\cong
H^*(\Lambda BSp(n),\mathbb{C}) \\
&\cong
\Lambda^*(u_3,u_7,\dots,u_{4n-1}) \otimes H^*(P_1,P_2,\dots,P_n)
\end{split}
\end{align*}
where $P_i$ are the universal Pontryagin classes.

In general, $(E_0^{*,*}[\mathfrak{g}],d_0)$ is a free graded commutative dga whose cohomology
is that of the free loop space of $BG$. (However it is not a Sullivan algebra as it fails to satisfy
the nilpotency condition.)
\end{thm}
As mentioned above, each of these results over $\mathbb{C}$ yield a picture for the spectral
sequence of the associated Lie algebra over $\mathbb{F}_p$ and hence the Bockstein spectral
sequence of the corresponding $p$-group, at least for low Hodge degree and for all but finitely
many primes. This is discussed in detail in the paper, as is the behaviour of higher pages of the
spectral sequence.

Over $\mathbb{C}$, these results are probably, by and large, repackaging of classical results in the cohomology
of Lie groups and their classifying spaces, invariant theory and Casimir theory into a spectral sequence,
but we go through the process in detail in the paper as we need the spectral sequence for our work in $p$-groups.

When working in prime characteristic, computations become much more difficult and a
``weight stratification'' is required. If $R$ is a system of roots for the semisimple Lie algebra
$\mathfrak{g}$ with root lattice $\Lambda(R)$, we show that there is a decomposition of the
spectral sequence by weight which is very helpful for computations of associated
Lie algebras in prime characteristic:

\begin{thm}(Weight Stratification)
Let $\mathfrak{g}$ be a complex semisimple Lie algebra. Let $\mathfrak{g}_{\mathbb{Z}}$ be a
corresponding integral Lie algebra (always exists by Cartan-Serre basis). We can then get a corresponding
Lie algebra $\mathfrak{g}_{\bf k}$ over any ring of definition ${\bf k}$.

The root decomposition of $\mathfrak{g}$ induces a weight decomposition of spectral sequences:
$$
E_r^{*,*}[\mathfrak{g}_{\bf k}] = \bigoplus_{\alpha \in \Lambda(R)} E_r^{*,*}[\alpha]
$$
where $\Lambda(R)$ is the root lattice of $\mathfrak{g}$ and is a free abelian group of rank
equal to the rank of $\mathfrak{g}$, i.e., the dimension of a Cartan subalgebra of $\mathfrak{g}$.

In addition we show that if ${\bf k}$ is a field of characteristic zero then the $E_1$-page and beyond
only has contributions from the weight $0$ term while for a field of characteristic $p$, the ``polynomial
line'' $E_1^{*,0}$ has contributions only from weights which are zero modulo $p$.
\end{thm}

In sections~\ref{section: DeRham} and \ref{section: modp}, using
a DeRham complex and D-module language together with the weight stratification
mentioned above, explicit mod $p$ calculations are
performed for the spectral sequence
$E_r^{*,*}[\mathfrak{sl}_2(\mathbb{F}_p)]$ and the ``almost all'' caveat
is removed in a range of Hodge degrees:

\begin{thm}($\mathfrak{sl}_2(\mathbb{F}_p)$-computation)
$$
E_1^{*,*}[\mathfrak{sl}_2(\mathbb{F}_p)]=
H^*(\mathfrak{sl}_2(\mathbb{F}_p), Poly(ad^*)) \sim
\Lambda^*(u) \otimes \mathbb{F}_p[\kappa]
$$
where $\sim$ indicates isomorphism for {\bf all odd primes} in the range
of Hodge degrees strictly less than $p-1$. Thus for all odd primes $p$,
the characteristic $0$ to $p$ phase transition does not occur before
Hodge degree $p-1$ for the Lie algebra scheme $\mathfrak{sl}_2(-)$.
(For $p=2$ the breakdown occurs immediately at Hodge degree $0$.)

However a char $0$ to char $p$ phase transition occurs at Hodge degree $p-1$
and it is shown that one needs minimally 17 generators to generate \\
$H^*(\mathfrak{sl}_2(\mathbb{F}_p), Poly(ad^*))$ through Hodge degree $p$.
(see section~\ref{section: modp} for details on these generators and the algebra).

It follows trivally that $H^*(\mathfrak{sl}_2(\mathbb{Z}), Poly(ad^*))$
has $p$-torsion for {\bf every prime $p$} and hence cannot be a finitely
generated ring.
\end{thm}

From this it follows that the exceptional dimension (maximal dimension
for which exceptional high order torsion exists) for the
$p$ groups $G(\mathfrak{sl}_2(\mathbb{F}_p))$ is greater than or equal
to $2p-2$ for all odd primes $p$.

As a byproduct of the analysis of the spectral sequence, the following 
exact sequence is obtained for $\mathfrak{sl}_2=\mathfrak{sl}_2(\mathbb{F}_p)$ 
and $i \geq p-2$:
$$
0 \to H^3(\mathfrak{sl}_2, S^i) \to 
H^2(\mathfrak{sl}_2,S^{i+1}) \to 
H^1(\mathfrak{sl}_2, S^{i+2}) 
\to H^0(\mathfrak{sl}_2,S^{i+3}) \to 0,
$$ 
where $S^i$ is the module of homogeneous, degree $i$ polynomials 
on $\mathfrak{sl}_2$ equipped with the dual adjoint action.

Finally while many of the techniques apply
to solvable (or nilpotent) Lie algebras also, we concentrate on semisimple Lie algebras
in this paper, postponing the discussion for nilpotent Lie algebras such as those which would
arise from torsion-free, pro-$p$, finite index, subgroups of the Morava Stabilizer group to another
time.

\section{The Classifying Spectral Sequence $E[\mathfrak{g}]$
and Universal Coefficient Arguments}
\label{section: SpectralSequence}

Associated to any $\mathbf{k}$-Lie algebra $\mathfrak{g}$
is a spectral sequence $(E[\mathfrak{g}]^{s,t}_r, d_r)$, constructed in the appendix.
We call $t$ the exterior degree and $s$ the Hodge or polynomial
degree and $2s+t$ the total degree. This is a first quadrant spectral sequence
and often we will plot the Hodge degree along the $x$-axis and either
the exterior degree $t$ or the total degree $2s+t$ along the $y$-axis. The
reader is encouraged to draw such a diagram when following the arguments.
Often we will surpress
$\mathfrak{g}$ explicitly from the notation when it is understood.
This spectral sequence has the
following properties:
\\
(1) As a $\mathbf{k}$-algebra $E_0^{*,*}=\Lambda^*(\mathfrak{g}^*)
\otimes \mathbf{k}[\mathfrak{g}^*]$; i.e., it is the tensor product
of the exterior algebra on the dual of $\mathfrak{g}$ with the
polynomial algebra on the dual of $\mathfrak{g}$. If $s$
denotes the natural isomorphism $\Lambda^1(\mathfrak{g}^*) \to Poly^1[\mathfrak{g}^*]=\mathfrak{g}^*$, 
then we can write $E_0^{*,*}=\Lambda^*(\mathfrak{g}^*)
\otimes \mathbf{k}[s(\mathfrak{g}^*)]$ as a bigraded algebra where
$\mathfrak{g}^*$ is given bigrading $(s,t)=(0,1)$ while
$s(\mathfrak{g}^*)$ is given bigrading $(s,t)=(1,0)$. Note however
that graded commutativity of the dga is determined by total degree and
that the elements $s(\mathfrak{g})$ have even total degree as required for
them to generate a polynomial algebra. Explicitly, if $\{ e_1, \dots, e_n \}$
is a $\mathbf{k}$-basis for $\mathfrak{g}$ and $\{ x_1, \dots, x_n \}$
the canonical dual basis for $\mathfrak{g^*}$ determined by
$x_i(e_j) = \delta_{i,j}$, and we set $y_i=s(x_i)$, then
$$
E^{*,*}_0 = \Lambda^*(x_1,\dots,x_n) \otimes \mathbf{k}[y_1,\dots,y_n].
$$
\\
(2) Each page of the spectral sequence is a dga over $\mathbf{k}$
such that
$$
d_r: E^{s, t}_r \to E^{s+r, t-(2r-1)}_r
$$
i.e., the $r$th differential raises the Hodge degree by $r$
and reduces the exterior degree by $(2r-1)$, and hence raises total
degree by $1$. The differential $d_r$ is a derivation with respect to the induced algebra structure; i.e.,
$$
d_r(\alpha \beta) = d_r(\alpha)\beta + (-1)^{|\alpha|}\alpha d_r(\beta),
$$
where $\alpha, \beta$ are homogeneous elements of $E_r^{*,*}$ and $|\alpha|$ denotes
the total degree of $\alpha$.

Of course as usual $E^{*,*}_{r+1} = H^*(E_r^{*,*},d_r)$.
\\
(3) The differential $d_0$ is induced naturally as follows.
It is the unique derivation on $E_0^{*,*}$ such that
$$
d_0: \Lambda^1(\mathfrak{g}^*) \to \Lambda^2(\mathfrak{g}^*)
$$
is minus the dual of the Lie-bracket $[-,-]:
\Lambda^2(\mathfrak{g}) \to \mathfrak{g}$ and
$$
d_0: s(\mathfrak{g}^*) \to s(\mathfrak{g}^*) \otimes \Lambda^1(\mathfrak{g}^*)
$$
is the dual of the Lie-bracket followed by $Identity \otimes s^{-1}$.
More explicitly if $c_{ij}^k$ are the structure constants of
$\mathfrak{g}$ with respect to the $\mathbf{k}$-basis $\{e_1,\dots,e_n\}$, 
$\{x_i\}$ is the dual basis to $\{e_i\}$, and $y_i=s(x_i)$, then
\begin{equation}\label{eqn: d0formula}
\begin{split}
d_0(x_i) = -\sum_{j < k} c_{jk}^i x_j x_k
\text{ and } \\
d_0(y_i) = \sum_{1 \leq j,k \leq n} c_{jk}^i y_j x_k.
\end{split}
\end{equation}
\\
(4) Thus the dga $(E_0^{*,*},d_0)$ breaks up as a direct sum of
finite dimensional dga's $E_0^{*,*} = \oplus_{s=0}^{\infty}
(E_0^{s,*},d_0)$ where the decomposition is based on polynomial
degree and will be refered to as the ``Hodge decomposition''.
As explained in the appendix, each of these turns out to be the Koszul
resolution for the Lie algebra $\mathfrak{g}$ for the module
$Poly^s(\mathfrak{g}^*)$ of homogeneous degree $s$ polynomials
on the dual of $\mathfrak{g}$. (For infinite fields, $Poly^s(\mathfrak{g}^*)$ can be identified
as homogeneous degree $s$ polynomial functions $f: \mathfrak{g} \to {\bf k}$ on $\mathfrak{g}$.)
Thus one has
$$
E_1^{*,*} = \bigoplus_{s=0}^{\infty} H^*(\mathfrak{g}, Poly^s(\mathfrak{g}^*))
=H^*(\mathfrak{g},Poly(\mathfrak{g}^*)).
$$
When $\mathbf{k}$ is a field of characteristic zero,
one can identify the dual of the Universal Enveloping algebra
of $\mathfrak{g}$ i.e., $U(\mathfrak{g})^*$ with $Poly(\mathfrak{g}^*)$
as an algebra where $U(\mathfrak{g})^*$ is given the dual algebra
structure coming from the comultiplication in the primitively generated
Hopf algebra $U(\mathfrak{g})$. This follows essentially from the Poincare-Birkoff-Witt 
theorem (we need a field of characteristic zero as in general
the dual of a primitively generated Hopf algebra is a divided power algebra
which only is a polynomial algebra over fields of characteristic zero).
This isomorphism is also one of $\mathfrak{g}$-modules
as long as we equip $U(\mathfrak{g})^*$ with the dual adjoint action
coming from the adjoint action of $\mathfrak{g}$ on $U(\mathfrak{g})$.
Thus over fields of characteristic zero,
$$
E_1^{*,*} = H^*(\mathfrak{g}, U(\mathfrak{g})^*).
$$
\\
(5) The spectral sequence comes from two anticommuting differentials
on $E_0^{*,*}$, $d_0$ and $d_1$. $d_1$ is defined as follows:
$d_1$ is the unique derivation which is $s$ on $\Lambda^1(\mathfrak{g}^*)$
and zero on $Poly^1(s(\mathfrak{g}^*))$ i.e., $d_1(x_i)=y_i$
and $d_1(y_i)=0$ for all $1\leq i \leq n$. This $d_1$ induces
the differential on $E_1^{*,*}$. Then via a completely standard zig-zag
construction (see \cite{BT}), all higher differentials can be determined
from these two as follows: If $\omega$ represents a class in $E_r$ then
one inductively can define $d_1(\omega) = d_0(\alpha_1)$,
$d_1(\alpha_k) = d_0(\alpha_{k+1})$ for $1 \leq k \leq r-2$
and finally $d_1(\alpha_{r-1})$ represents $d_r(\omega)$ and, as is
typical with these sorts of constructions, the final class can be shown
to be independent of the choices made along the way.
\\
(6)
Once $2r-1 > \dime (\mathfrak{g})$ we have $d_r=0$ and so
$E_{m}^{*,*}=E_{\infty}^{*,*}$ once $m > \frac{\dime (\mathfrak{g})+1}{2}$. Furthermore as shown in the appendix,
$$
E_{\infty}^{s,t} =
\begin{cases} \mathbf{k} \text{ if } s=t=0 ; \\
0 \text{ otherwise.}
\end{cases}
$$
\\
(7) $E[-]^{*,*}$ defines a contravariant functor from the category
of $\mathbf{k}$-Lie algebras to the category of $\mathbf{k}$-differential
graded algebra spectral sequences. In other words, if
$\theta: \mathfrak{g} \to \mathfrak{h}$ is a map of $\mathbf{k}$-Lie algebras
then it is not hard to show that
$\theta^*: \mathfrak{h}^* \to \mathfrak{g}^*$ commutes with the dual
Adjoint action in the sense that if we make $\mathfrak{h}^*$ a
$\mathfrak{g}$-module using $\theta$ then $\theta^*$ is a
$\mathfrak{g}$-module map. It is then a routine exercise to check
that the algebra map induced from $\theta$,
$$
\Lambda^*(\mathfrak{h}^*) \otimes \mathbf{k}[s(\mathfrak{h}^*)]
 \to \Lambda^*(\mathfrak{g}^*) \otimes \mathbf{k}[s(\mathfrak{g}^*)],
$$
commutes with $d_0$ and $d_1$ and hence induces a morphism of spectral
sequences $(E[\mathfrak{h}]^{*,*}_r,d_r) \to (E[\mathfrak{g}]^{*,*}_r,d_r)$.
\\
(8)
The construction $E[-]^{*,*}$ is natural with respect to extension of
scalars. In other words if $\mathbf{k} \to \mathbf{K}$ is a homomorphism
of rings between two PIDs and $\mathfrak{g}$ is a $\mathbf{k}$-Lie algebra
then $\mathfrak{g} \otimes_{\mathbf{k}} \mathbf{K}$ is naturally a
$\mathbf{K}$-Lie algebra and
$$
(E[\mathfrak{g} \otimes_{\mathbf{k}} \mathbf{K}]^{*,*}_0,d_0)
= (E[\mathfrak{g}]^{*,*}_0 \otimes_{\mathbf{k}} \mathbf{K}, d_0
\otimes_{\mathbf{k}} Id)
$$
as $\mathbf{K}$-dga's. Thus if $\mathbf{K}$ is flat over $\mathbf{k}$
(this is needed
to avoid tor terms in the universal coefficient theorem) one has
$$
(E[\mathfrak{g} \otimes_{\mathbf{k}} \mathbf{K}]^{*,*}_r,d_r)
= (E[\mathfrak{g}]^{*,*}_r \otimes_{\mathbf{k}} \mathbf{K}, d_r
\otimes_{\mathbf{k}} Id)
$$
as $\mathbf{K}$-spectral sequences.
\\
(9) For $k$ a field with $p$ elements, $p$ an odd prime and
$\mathfrak{g}$ a $\mathbf{k}$-Lie algebra which lifts over
$\mathbb{Z}/p^2\mathbb{Z}$, one has the following result
proven in $\cite{BrP}$:

If $(B^*_r,\beta_r)$ is the Bockstein spectral sequence for the
$p$-group $G(\mathfrak{g})$, then
$(B_1,\beta_1)=(E_0,d_0)$ and so $B_2^{n}=\oplus_{s,t| 2s+t=n}E_1^{s,t}$.
Furthermore both $B_{\infty}$ and $E_{\infty}$ give the cohomology of
a point. Thus $E_r^{*,*}$ can be viewed as an algebraic model for
the Bockstein spectral sequence and is computationally useful
as we will see. It is known that $E_1=B_2$, and the properties of
$E_1$ will allow us to compute it more readily.
(It is unknown whether $E_r=B_{1+r}$ for $r > 1$ but it is true in all
computed examples.)  In the following we concentrate on computations
of $B_2^*=E_1^{*,*}=H^*(\mathfrak{g},Poly(\mathfrak{g}^*))$. The flexibility
to change coefficient rings in $E_1[-]^{*,*}$ will play a very fundamental
role, as will the fact that we know the differentials $d_r$ in $E_r$
raise polynomial degree by $r$.

The fundamental idea is encoded in the following application of the
universal coefficient theorem. This might be a bit abstract here but in the
next few sections where explicit examples are worked out it will become
clearer. To avoid needless generality we state this theorem for a
special case of complex simple Lie algebras though results can easily
be extended.

\begin{thm}[Fundamental Comparison Theorem]
\label{thm: FCT}
Let $\mathfrak{g}$ be a $\mathbb{Z}$-Lie algebra such that
$\mathfrak{L}=\mathfrak{g} \otimes \mathbb{C}$ is a complex simple Lie
algebra. (Cartan-Serre basis shows all complex simple Lie algebras arise this way.)
Let $H^0(\mathfrak{L},Poly(\mathfrak{L}^*))=H^0(\mathfrak{L},U(\mathfrak{L})^*) \subseteq Poly(\mathfrak{L}^*)$ be denoted by $\Gamma$.
Then we have
$$
E_1[\mathfrak{L}]^{*,*} = H^*(\mathfrak{L},U(\mathfrak{L})^*)
= H^*(\mathfrak{L},\mathbb{C}) \otimes_{\mathbb{C}} \Gamma
$$
where $H^*(\mathfrak{L},\mathbb{C})$ can be identified with the
cohomology of a certain (compact form) Lie group $G$ and is an exterior
algebra on odd degree generators. (We will see in future sections, that $\Gamma$ will correspond
to $H^*(BG,\mathbb{C})$ which is a polynomial algebra.)

Now since the part of $E_0[\mathfrak{g}]$ corresponding to elements of
Hodge degree $\leq N$ is finitely generated, its cohomology
$E_1[\mathfrak{g}]$ can have torsion only for finitely many primes.

Thus for all but a finite number of primes $p$, we have
$$
\dime(E_1^{s,t}[\mathfrak{g} \otimes \mathbb{F}_p]) =
\dime(E_1^{s,t}[\mathfrak{g} \otimes \mathbb{C}])
$$
for all $s,t$ with $s \leq N$.
On the other hand we will see that if a restriction on Hodge
degree is not imposed, this will {\bf always} break down. For example
for $\mathfrak{g}=\mathfrak{sl}_2(\mathbb{Z})$ we will see that
$E_1^{*,*}[\mathfrak{g}]=H^*(\mathfrak{g},Poly(\mathfrak{g}^*))$ has
$p$-torsion for every prime $p$
and hence is not a finitely generated ring.
\end{thm}
\begin{proof}
The proof of the form of $E_1[\mathfrak{L}]^{*,*}$ hinges on two facts.
\\
(1) As $\mathfrak{L}$ is a simple Complex Lie algebra, every finite
dimensional complex representation $V$ decomposes as a sum of
irreducible representations. \\
(2) If $V$ is an irreducible representation other than the trivial one-dimensional 
representation, then $H^*(\mathfrak{L},V)=0$. (This is a generalization of 
Whitehead's Lemma, see Proposition 3.4.2 in page 151 of \cite{GG}.)
With this we can decompose
$Poly(\mathfrak{L}^*)=\Gamma \oplus W$ as $U(\mathfrak{L})$-modules
where $\Gamma$ are the adjoint invariant
polynomials; i.e., $\Gamma=H^0(\mathfrak{L},Poly(\mathfrak{L}^*))$, as mentioned
in the statement of the theorem.  Note $W$ decomposes as a sum of
irreducibles, none of which is the trivial one-dimensional representation.
Thus
\begin{align*}
\begin{split}
H^*(\mathfrak{L},Poly(\mathfrak{L}^*))&=H^*(\mathfrak{L},\Gamma) \oplus
H^*(\mathfrak{L}, W) \\
&=H^*(\mathfrak{L},\Gamma) \text{ by Whitehead's Lemma }\\
&=H^*(\mathfrak{L},\mathbb{C}) \otimes_{\mathbb{C}} \Gamma
\text{ as the action of $\mathfrak{L}$ on $\Gamma$ is trivial.}
\end{split}
\end{align*}

For every simple complex Lie algebra $\mathfrak{L}$, 
there exists a real Lie algebra $\mathfrak{L_R}$
whose complexification is $\mathfrak{L}$ and which is the Lie algebra
of a compact connected Lie group $G$. It is known classically that
$H^*(\mathfrak{L_R},\mathbb{R})=H^*(G,\mathbb{R})$ is an exterior algebra
on odd generators (as it is a Hopf algebra). Thus its complexification
$H^*(\mathfrak{L},\mathbb{C})$ is also an exterior algebra on odd generators.

The final parts are a direct application of the universal coefficient
theorem to the complex $(E_0^{*,*}(\mathfrak{g}), d_0)$ and will be left
to the reader. The example for $\mathfrak{sl}_2(\mathbb{Z})$ will be discussed
later in the paper.

\end{proof}

In the next section we compute the spectral sequence
$E_r^{*,*}(\mathfrak{sl}_2(\mathbb{C}))$ 
completely using two separate methods, 
one using just properties of the spectral sequence itself and the
second using classical invariant theory for Lie groups. Both
recover the classical calculation of Harish-Chandra of the
Casimir algebra for $\mathfrak{sl}_2(\mathbb{C})$. The invariant
theory viewpoint provides a concrete meaning to the answer.
Then in section~\ref{section: complexsimple}, we compute the behaviour
of the spectral sequence $E_r^{*,*}(\mathfrak{L})$ for any
complex simple Lie algebra and relate the answer using invariant theory.

Using the Fundamental Comparison Theorem above we then will get
results for the corresponding Lie algebras of corresponding $p$-groups
which we will use to understand higher torsion in their integral cohomology.

\section{Computation of Dual Casimirs and Invariant Theory}

For a good background discussion of the theory of complex
simple Lie algebras and corresponding Lie groups that we use
in this paper, see \cite{FH}. Recall that Casimir elements are elements of the center of the universal enveloping algebra or in other words 
ad-invariant elements of $U(\mathfrak{L})$. Dual Casimir elements are ad-invariant elements of the filtered dual $U(\mathfrak{L})^*$.

Recall $\mathfrak{sl}_2(\mathbb{C})$ is a Lie algebra
of dimension $3$ with $\mathbb{C}$-basis
$$
\left \{x_h = \begin{bmatrix} 1 & 0 \\ 0 & -1 \end{bmatrix},
x_e = \begin{bmatrix} 0 & 1 \\ 0 & 0 \end{bmatrix},
x_f = \begin{bmatrix} 0 & 0 \\ 1 & 0 \end{bmatrix} \right \}
$$
and commutation relations
$$[x_h,x_e]=2x_e, [x_h,x_f]=-2x_f, [x_e,x_f]=x_h.$$

Let ${h, e, f}$ denote the corresponding dual basis in
$\Lambda^1(\mathfrak{L}^*)$ and let ${H, E, F}$ denote
the ``suspended'' dual basis in $Poly^1(\mathfrak{L}^*)$.

Thus
$$
E_0^{*,*}[\mathfrak{sl}_2(\mathbb{C})]
= \Lambda^{*}(h,e,f) \otimes \mathbb{C}[H,E,F]
$$
and using Theorem~\ref{thm: FCT}, we find
$$
E_1^{*.*}[\mathfrak{sl}_2(\mathbb{C})]
= \Lambda^*(u) \otimes \Gamma
$$
where $\Lambda^*(u)=H^*(\mathfrak{sl}_2(\mathbb{C}), \mathbb{C})$
with $u=hef$ (a simple computation) and \\
$\Gamma=H^0(\mathfrak{sl}_2(\mathbb{C}),U(\mathfrak{sl}_2(\mathbb{C}))^*)$
is the dual Casimir algebra that we seek to find.

It is not hard to see that $d_r=0$ on $\Gamma$ for $r \geq 1$ as these
differentials lower exterior degree and so
$\Gamma$ consists of permanent cycles in the spectral sequence
$E_r^{*,*}[\mathfrak{sl}_2(\mathbb{C})]$. However, we also know
that $E_{\infty}^{*,*}[\mathfrak{sl_2(\mathbb{C})}]$ is the cohomology
of a point, so everything needs to be killed off, thus $u$ must support
some differential. Note that $H^*(\mathfrak{L},Poly^1(ad^*))=0$
by Whitehead's lemma as $Poly^1(ad^*)=ad^*$ is a nontrivial irreducible
representation of $\mathfrak{L}$. Thus $d_1(u)=0$ and so $E_1=E_2$.
Now since the
dimension of the Lie algebra is three, and $d_2$ lowers exterior degree
by three, this is the final possible nonzero differential and so
$E_3=E_{\infty}$ and so we know $(E_2,d_2)$ is an acyclic complex.

Let $\kappa_2=d_2(u)$ then $\kappa_2 \in \Gamma$ is an invariant quadratic
polynomial. The $E_2$ page is concentrated on two horizontal lines:
a line with exterior degree 3 with term $u \otimes \Gamma$ and a line
with exterior degree 0 with term $1 \otimes \Gamma$. The
differential $d_2$ between these two lines is given by
multiplication by $\kappa_2 \cdot: \Gamma \to \Gamma$. Since the complex is
acyclic we conclude that $\kappa_2 \cdot:  \Gamma \to \Gamma^+$ is an
isomorphism where $\Gamma^+$ are the elements of $\Gamma$ of positive
degree. A simple induction then shows $\Gamma \cong \mathbb{C}[\kappa_2]$.
This yields the following theorem:

\begin{thm}{$E_r[\mathfrak{sl}_2(\mathbb{C})]$-computation:}
\label{thm: dualcasimirsl_2}
$$
E_1^{*,*}[\mathfrak{sl}_2(\mathbb{C})] \cong
H^*(\mathfrak{sl}_2(\mathbb{C}), U(sl_2(\mathbb{C}))^*)
\cong \Lambda^*(u) \otimes \mathbb{C}[
\kappa_2]
$$
where $u \in E_1^{0,3}$ and $\kappa_2 \in E_1^{2,0}$. The element $\kappa_2$
is a permanent cycle in the spectral sequence and $d_2(u)=\kappa_2$.
$\kappa_2$ is a nonzero multiple of the Killing form of
$\mathfrak{sl}_2(\mathbb{C})$ and hence is a nonzero multiple
of $H^2+EF$.
\end{thm}
\begin{proof}
Everything besides the last comment has been already proven.
Note for any semisimple complex Lie algebra $\mathfrak{g}$, the
Killing form
$$\kappa(A,B)=\trace(ad(A) \circ ad(B): \mathfrak{g}
\to \mathfrak{g})
$$ is a symmetric $2$-form on $\mathfrak{g}$ which
is nondegenerate and hence non-zero. Furthermore it has the property
that $\kappa([A,B],C)+\kappa(B,[A,C])=0$ for all $A,B,C \in \mathfrak{g}$
which implies that it represents a nonzero adjoint-invariant
element in $$H^0(\mathfrak{g}, Poly^2(\mathfrak{g}^*))=
E_1^{2,0}[\mathfrak{g}].$$
Since in our case $H^0(\mathfrak{g}, Poly^2(\mathfrak{g}^*))$
is one-dimensional and spanned by $\kappa_2$, the final comment follows. A simple computation
shows that the Killing form is a nonzero multiple of $H^2+EF$ when the
standard identification of quadratic forms with symmetric inner products
is made.
\end{proof}

A few things to note about the last calculation: \\
(1) The number of exterior/polynomial generators in
$E_1=H^*(\mathfrak{sl}_2,U(\mathfrak{sl}_2)^*)$ is one,  which 
equals the rank of $\mathfrak{sl}_2(\mathbb{C})$, i.e., the
dimension of a Cartan subalgebra in this case the span of $x_h$.
This will hold in general for any classical complex simple Lie algebra. \\
(2) The computation of the dual Casimirs can be used to find the Casimirs,
i.e., the central elements in the non-commutative algebra
$U(\mathfrak{sl}_2)$. This is because the Killing form sets
up an isomorphism between the representations $ad$ and $ad^*$ and hence
between $H^0(\mathfrak{sl}_2, U(\mathfrak{sl}_2(\mathbb{C})))$
and $H^0(\mathfrak{sl}_2, U(\mathfrak{sl}_2(\mathbb{C}))^*)$ i.e., between
the Casimir algebra and the dual Casimir algebra. One has to be a bit careful
though as $U(\mathfrak{sl}_2)$ is only a filtered module and it is
noncommutative: the Killing form isomorphism is only between the associated
graded of $U(\mathfrak{g})$ and $U(\mathfrak{g})^*$ in general.

What this means is
that when finding the corresponding central elements in $U(\mathfrak{sl}_2)$,
one has to do a few things: \\
(a) Find the duals under the Killing form and view them in the associated graded of
$U(\mathfrak{sl_2})$. Simple computations show $\kappa(x_e,x_f)=4=\kappa(x_f,x_e)$, $\kappa(x_h,x_h)=8$
and all other inner products of the three basis elements with respect to the Killing form are zero.
Thus with respect to the Killing form, the dual of $H$ is $\frac{1}{8}x_h$, the dual of $E$ is $
\frac{1}{4}x_f$ and the dual
of $F$ is $\frac{1}{4}x_e$.
Thus the dual of $H^2+EF \in H^0(\mathfrak{sl}_2,U(\mathfrak{sl}_2)^*)$ 
is $\frac{1}{64}x_h^2 + \frac{1}{16}x_fx_e$
or up to
scaling, $x_h^2 + 4x_fx_e \in H^0(\mathfrak{sl}_2, A)$ where $A$ is the associated graded of
the noncommutative filtered algebra $U(\mathfrak{sl}_2)$. \\
(b) Find the correct invariant lift from the
symmetric algebra $A$ to the noncommutative universal enveloping algebra
$U(\mathfrak{sl}_2)$. An invariant lift, i.e., a lift to a central element of $U(\mathfrak{sl}_2)$ is always
possible by semisimplicity and this procedure works in general for any complex semisimple Lie algebra.
In this case, the correct lift can be obtained by symmetry considerations as
$x_h^2 + 2x_fx_e + 2x_ex_f$. \\

The physical importance of this computation is as follows. In angular momentum or spin systems,
$x_h, x_e, x_f$ represent $L_z, L_+=L_x + iL_y,L_-=L_x-iL_y$ respectively,
angular momentum operators in 3 directions in $\mathbb{R}^3$.
Thus their span $\mathfrak{sl}_2(\mathbb{C})$ represents a space of angular momentum operators taken with
respect to all the possible axis directions. In quantum mechanics, a simultaneous measurement of two quantities
can be taken if and only if their operators commute. Thus in this case since $\mathfrak{sl}_2(\mathbb{C})$
has rank one, one can only measure one of these quantities at a time and cannot simultaneously measure
say $L_z$ and $L_x$. Since it is beneficial to be able to simultaneously measure as many observables as possible, one then is asked if there are any operators in the operator algebra
generated by $L_z,L_x,L_y$ which commute with all these axis-specific angular momentum operators.
Since $U(\mathfrak{sl}_2)$ is exactly this operator algebra, we find that mathematically we are asking for
exactly the center of $U(\mathfrak{sl}_2)$ i.e., the Casimir algebra. 

By this calculation, we see
the only such operators are polynomials in $x_h^2+ 2x_ex_f + 2x_fx_e$, which after identifications comes out
to be a scalar multiple of the total angular momentum squared $L^2=L_x^2+L_y^2+L_z^2$. Thus in spin systems
one can measure $L^2$ and, say, $L_z$ simultaneously and the corresponding values turn out to be crucial
in the physics and chemistry of spin systems. 
%(Note to physicists, this last fact was first established by 
%Harish-Chandra and in this paper when discussing weights
%and such we do not divide by $2$, thus later when root systems are discussed with respect to the adjoint
%representation, the roots are $-2,0,2$ which correspond to magnetic spin $m=-1,0,1$ and $l=1$ in the physics
%notation.) 
In general we will not comment anymore about finding Casimirs from dual Casimirs as it is
not the primary aim of this paper.

While it is possible to perform a similar calculation of the spectral sequence for $\mathfrak{sl}_3
(\mathbb{C})$ purely using properties of the spectral sequence $E_r^{*,*}$,
for $\mathfrak{sl}_n(\mathbb{C}), n > 4$ ambiguities arise in differentials.
Thus to perform the analogous calculation for an arbitrary complex
simple Lie algebra, a supplementary approach has to be taken involving
invariant theory. We will repeat the calculation for
$\mathfrak{sl}_2(\mathbb{C})$ using this invariant theory approach now before doing the general
calculations over $\mathbb{C}$ in the next section. This recalculation will
give us additional insight into what we are calculating.

Here are the basic observations needed for the invariant theory viewpoint: \\
(1) Fix a semisimple Lie algebra $\mathfrak{L}$ over $\mathbb{C}$,
then $Poly^*(\mathfrak{L}^*)$ can be naturally identified as the algebra of
polynomial functions on $\mathfrak{L}$. (This identification is natural
and injective over infinite fields in general, though over finite fields,
different polynomials can induce the same function on the underlying
vector space.) \\
(2) If $Poly^*(\mathfrak{L}^*)$ is given the dual adjoint action,
$H^0(\mathfrak{L}, Poly(\mathfrak{L}^*))=\Gamma$ can be identified
with the $ad$-invariant polynomial functions on $\mathfrak{L}$. \\
(3) If $G$ is a connected Lie group with Lie algebra $\mathfrak{L}$
then the image of the exponential map $exp: \mathfrak{L} \to G$
generates $G$. (Indeed a connected Lie group is generated by a small
neighborhood of the identity and the image of the exponential map
contains such. The exponential map itself need not be onto, indeed
the exponential map $exp: \mathfrak{sl}_2(\mathbb{C}) \to SL_2(\mathbb{C})$
is not onto.) It can then be shown that the Adjoint action of the Lie group
$G$ on $\mathfrak{L}$ has exactly the same invariant polynomial functions on
it. Thus
$$
H^0(G, Poly(\mathfrak{L}^*))=H^0(\mathfrak{L},Poly(\mathfrak{L}^*))=\Gamma.
$$ \\
(4) It follows from the remarks above that in the case
of $G=SL_n(\mathbb{C})$ and $\mathfrak{L}=\mathfrak{sl}_n(\mathbb{C})$,
$\Gamma$ can be identified with the polynomial functions
$f: \mathfrak{sl}_n(\mathbb{C}) \to \mathbb{C}$ that are invariant
under $SL_n(\mathbb{C})$-conjugation. \\
(5) Consider the characteristic polynomial of a matrix,
$$P_{\mathbb{A}}(x) = det(x\mathbb{I}-\mathbb{A})
=x^n - \sigma_1(\mathbb{A}) x^{n-1} + \sigma_2(\mathbb{A}) x^{n-2}
- \dots + (-1)^n \sigma_n(\mathbb{A}).$$
Note that $\sigma_j: \mathfrak{gl}_n \to \mathbb{C}$ is a homogeneous
polynomial
of degree $j$ in the entries of $\mathbb{A}$ with value equal
to the $j$th elementary symmetric function of the eigenvalues of
$\mathbb{A}$. Thus for example $\sigma_1=trace$ is a linear polynomial
in the entries of $\mathbb{A}$ while $\sigma_n=det$ is a degree $n$
polynomial in the entries of $\mathbb{A}$. Since the characteristic
polynomial of a matrix is unchanged by conjugation/similarity, we conclude
that
$\sigma_1, \sigma_2, \dots, \sigma_n$ represent elements in
$H^0(GL_n(\mathbb{C}), Poly(\mathfrak{gl}_n^*))=H^0(\mathfrak{gl}_n,
Poly(\mathfrak{gl}_n^*))$. It is also clear that these elements
are algebraically independent polynomials as they restrict to the
elementary symmetric functions when restricted as functions over the
diagonal matrices, and these are well-known to be algebraically independent.
\\
(6) It is not hard to show that when restricted to functions over
$\mathfrak{sl}_n(\mathbb{C})$, $\sigma_1=0$ but $\sigma_2,\dots,\sigma_n$
remain algebraically independent. (Just restrict to the diagonals
of $\mathfrak{sl}_n$ to check this). Thus we have found a polynomial
subalgebra $\mathbb{C}[\sigma_2,\dots,\sigma_n]$ of
$\Gamma=H^0(\mathfrak{sl}_n,Poly(\mathfrak{sl}_n^*))$. Comparing
with our previous calculation we see that
$$
H^0(\mathfrak{sl}_2(\mathbb{C}),Poly(\mathfrak{sl}_2(\mathbb{C}))^*)
=\mathbb{C}[\kappa_2]=\mathbb{C}[\sigma_2], 
$$
and thus the only conjugation-invariant polynomial functions on
$\mathfrak{sl}_2(\mathbb{C})$ are polynomials in the determinant function, 
which is itself a quadratic function. (In the next section, we will
show the analogous result holds for $\mathfrak{sl}_n(\mathbb{C})$, 
i.e., that the conjugation invariant polynomials will be a polynomial
algebra on $\sigma_2, \dots, \sigma_n$.)
\\
(7) From these theoretical considerations, it follows that the Killing
form and determinant are both conjugation-invariant homogeneous
quadratic functions $\mathfrak{sl}_2(\mathbb{C}) \to \mathbb{C}$
and, from the computations, must be a nonzero scalar multiple of each other.
Of course this can be explicitly checked:
$$
det\left(\begin{bmatrix} a & b \\ c & -a \end{bmatrix}\right)
= -a^2-bc=-(a^2+bc).
$$
Using the same dual basis stated in the beginning of this section,
we then see $det = -(H^2+EF)$ and so indeed we have checked
that $det$ and the Killing quadratic form $\kappa_2$ are scalar multiples of
each other as functions on $\mathfrak{sl}_2(\mathbb{C})$.

\section{Computation of $E_r^{*,*}[\mathfrak{g}]$ where
$\mathfrak{g}$ is a complex simple Lie algebra}
\label{section: complexsimple}

In this section all Lie algebras will be over $\mathbb{C}$ and so
we will suppress that from the notation.
Let us start first with the family $\mathfrak{sl}_n$. Many of the arguments
are similar for other complex simple Lie algebras.

Recall a
Cartan subalgebra $\mathfrak{h}$ (maximal abelian Lie subalgebra consisting
of semisimple/diagonalizable elements) is given by the diagonal
matrices inside $\mathfrak{sl}_n$. For a semisimple complex
Lie algebra $\mathfrak{g}$, all Cartan subalgebras are Adjoint-conjugate
and the dimension of a Cartan subalgebra is called the rank of the
Lie algebra. Thus $\mathfrak{sl}_n$ has rank $n-1$. The adjoint
action of $\mathfrak{h}$ on $\mathfrak{g}$ is diagonalizable and a
simultaneous nonzero eigenvalue function $\alpha: \mathfrak{h} \to \mathbb{C}$
is called a {\bf root} with corresponding eigenvector $v$ called a
{\bf root vector}. Thus
$$
[h, v] = \alpha(h) v
$$
for all $h \in \mathfrak{h}$. A root $\alpha$ gives a complex linear
functional on $\mathfrak{h}$, i.e., $\alpha \in \mathfrak{h}^*$.

The vector space
$\mathfrak{g}_{\alpha}=\{ v \in \mathfrak{g} | [h,v]=\alpha(h)v
\text{ for all } h \in \mathfrak{h} \}$ is called the root space
corresponding to the root $\alpha$. It is standard (see section 14.1 of \cite{FH})
that the root spaces are 1-dimensional and
$\mathfrak{g} = \mathfrak{h} \oplus (\oplus_{\alpha \in R} \mathfrak{g}_
{\alpha}$) where $R$ is the set of roots. The Weyl group $W(\mathfrak{g})$
is a
group generated by reflections which acts on $\mathfrak{h}$ via
automorphisms of the ambient Lie algebra $\mathfrak{g}$ which map
the subspace $\mathfrak{h}$ back into itself.
In the case of $\mathfrak{sl}_n$, the Weyl group is the symmetric
group on $n$ letters $\Sigma_n$ which acts on the diagonal matrices
$\mathfrak{h}$ by permuting the diagonal entries.

The following lemma is basic to the invariant theory approach to
computing the dual Casimirs of $\mathfrak{g}$:

\begin{lem}[Invariants Lemma]
\label{lem: Invariants Lemma}
Let $\mathfrak{g}$ be a complex simple Lie algebra, $\mathfrak{h}$
a Cartan subalgebra, and $W(\mathfrak{g})$ the Weyl group of $\mathfrak{g}$.
Let
$$
\Gamma=H^0(\mathfrak{g},Poly^*(\mathfrak{g}^*)) =
H^0(\mathfrak{g},U(\mathfrak{g})^*)
$$
be the algebra of adjoint-invariant polynomial functions on
$\mathfrak{g}$ i.e., the dual Casimir algebra.
Let $H^0(W(\mathfrak{g}), Poly^*(\mathfrak{h}^*))$ denote the
Weyl group invariant polynomial functions on $\mathfrak{h}$. Then there
is an algebra monomorphism
$$
\theta: \Gamma \to H^0(W(\mathfrak{g}), Poly^*(\mathfrak{h}^*)),
$$
and furthermore $H^0(W(\mathfrak{g}),Poly^*(\mathfrak{h}^*))$
can be identified as $H^*(BG,\mathbb{C})$, where $G$ is the classifying
space of a simply connected, compact Lie group $G$
corresponding to a compact form
$\mathfrak{g}_{\mathbb{R}}$ of $\mathfrak{g}$.
\end{lem}
\begin{proof}
In this proof, note that a polynomial function on $\mathfrak{g}$
being invariant under the adjoint action of $\mathfrak{g}$ is the
same as it being invariant under the Adjoint action of $G$ where
$G$ is a connected Lie group having Lie algebra $\mathfrak{g}$ (though
the meaning of invariant is slightly different for Lie algebras versus
Lie groups). We will capitalize Adjoint when we are specifically referring
to the Lie group Adjoint action.

Restriction of adjoint-invariant polynomial functions on $\mathfrak{g}$ to polynomial
functions on $\mathfrak{h}$ gives an algebra homomorphism
$$
\theta: \Gamma \to H^0(W(\mathfrak{g}), Poly^*(\mathfrak{h}^*))
$$
as the restriction of an adjoint-invariant polynomial to
$\mathfrak{h}$ will yield a polynomial function on $\mathfrak{h}$
which is invariant under the Weyl group (this is because each element
of the Weyl group is induced by an Adjoint automorphism of
$\mathfrak{g}$). Thus it remains to show injectivity.
If $f$ is a nonzero element of $\Gamma$ then $f(v) \neq 0$
for some $v \in \mathfrak{g}$. Since $f: \mathfrak{g} \to \mathbb{C}$
is given by a polynomial, it is continuous. Thus using Jordan decompositions,
we can find a semisimple (diagonalizable) element $w$ close to $v$
such that $f(w) \neq 0$. (Matrices with distinct
eigenvalues form an open dense subset of $\mathfrak{sl}_n$ for example.)
Then $w$ lies in a Cartan subalgebra $\mathfrak{h}'$ of $\mathfrak{g}$
and since any two Cartan subalgebras are conjugate, there is an Adjoint Lie
algebra automorphism taking $\mathfrak{h}' \to \mathfrak{h}$ which
takes $w$ to some element $z \in \mathfrak{h}$. By invariance under
Adjoint automorphisms, $f(w)=f(z) \neq 0$ and so $f|_{\mathfrak{h}} \neq 0$
and so $\theta$ is a monomorphism.

The final statement is not needed for any of our $p$-group results but provides a
picture for the complex results so we will only sketch it. If $G$ is a compact,
simply connected Lie group $G$ corresponding to a compact form
$\mathfrak{g}_{\mathbb{R}}$ of $\mathfrak{g}$, then a maximal torus $T$ corresponds to a
compact form for the Cartan subalgebra $\mathfrak{h}$ and
$N_G(T)/T$ corresponds to the Weyl group $W$.

Thus
$$
H^0(W,Poly^*(\mathfrak{h})^*)=H^0(W, H^*(BT,\mathbb{C}))=H^*(BG,\mathbb{C}).
$$
The final equality is a well-known result of Borel and will not be reproven here.
In the middle we used that $H^*(BT,\mathbb{C})$ can be identified
with a polynomial algebra on $H^2(BT,\mathbb{C})=H^1(T,\mathbb{C})=
H^1(\mathfrak{h},\mathbb{C})=\mathfrak{h}^*$.
\end{proof}

As there are a lot of different important equivalences going on, let us
look at the explicit case $\mathfrak{sl}_n$ first. In this case
note that a compact form is given by the real Lie algebra
$\mathfrak{su}_n$ (see section 26.1 in \cite{FH}) with corresponding compact, simply
connected Lie group $G=SU(n)$. One has
$H^*(BSU(n),\mathbb{C})=\mathbb{C}[c_2,c_3,\dots,c_n]$ a polynomial
algebra on the universal Chern classes with \\ $\degr(c_i)=2i$ and
$\theta(\Gamma) \subseteq \mathbb{C}[c_2,c_3,\dots,c_n]$.

However in the last section we saw that
$\mathbb{C}[\sigma_2, \sigma_3,\dots, \sigma_n] \subseteq \Gamma$
with $\degr(\sigma_i)=2i$. ($\sigma_i$ is a degree $i$ polynomial
on $\mathfrak{sl}_n$ but in our setup as an element in $H^0(\mathfrak{sl}_n,
Poly^*(\mathfrak{sl}_n)^*)$ is an element of total degree $2i$ due to the
suspension of the polynomial variables.) Comparing these set
inclusions of graded algebras we conclude:

\begin{thm}[$\mathfrak{sl}_n$ invariant calculation]
Let $\Gamma=H^0(\mathfrak{sl}_n, Poly^*(ad^*))=H^0(\mathfrak{sl}_n, U(\mathfrak{sl}_n)^*)$ be the algebra of adjoint-invariant polynomial functions
on $\mathfrak{sl}_n$, or equivalently the dual Casimir algebra. Then
$\Gamma=\mathbb{C}[\sigma_2,\sigma_3,\dots,\sigma_n]$ where the
$\sigma_i$ have degree $2i$ and express the $i$th elementary symmetric
polynomial of the eigenvalues of a matrix in terms of the entries of the
matrix.

Restriction defines an isomorphism
$$
\theta: H^0(\mathfrak{sl}_n,U(\mathfrak{sl}_n)^*) \to
H^0(\Sigma_n, Poly^*(\mathfrak{h}^*))=H^*(BSU(n),\mathbb{C})=
\mathbb{C}[c_2,c_3,\dots,c_n].
$$
Thus in this case the dual Casimirs of $\mathfrak{sl}_n(\mathbb{C})$
can be identified with the polynomial algebra on the universal
Chern classes $c_2,c_3,\dots,c_n$.
Finally
$$
E_1^{*,*}[\mathfrak{sl}_n]=H^*(SU(n) \times BSU(n))=\Lambda^*(u_3,u_5,\dots,
u_{2n-1}) \otimes \mathbb{C}[c_2,c_3,\dots,c_n]$$
where $u_i \in E_1^{0,i}$ and $c_i \in E_1^{i,0}$.
\end{thm}

It is now a simple matter to work out the full behaviour of the
spectral sequence $E_r^{*.*}$. We will not really use this for any
$p$-group results so we will be brief. We recall the definition of
indecomposables in a connected graded algebra:

\begin{defn}
Let $A=\oplus_{n=0}^{\infty}A_n$ be a graded algebra and $A^+$ be
the ideal of positive degree elements. (We'll assume $A_0=\mathbf{k}$
is a copy of the base field.)
$$
A^+ \cdot A^+ =\left\{ \sum_{j=1}^m a_jb_j | a_j, b_j \in A^+ \right\}
$$
is called the ideal of decomposables.
$$
Q = A/(A^+ \cdot A^+)
$$
is called the space of indecomposables.
\end{defn}

We now make a series of observations:\\
(1) Let $A^*$ be $E_1^{*,*}[\mathfrak{sl}_n]$ graded by total degree.
Then the space of indecomposables $Q$ is spanned by
$\{u_3, u_5,\dots, u_{2n-1}, c_2,c_3, \dots, c_n \}$,
and furthermore the $c_i$ are permanent cycles and algebraically independent. \\
(2) Since the $d_r$ are derivations the image of a decomposable
under $d_r$ is decomposable. \\
(3) The spectral sequence converges to the cohomology of a point and
for dimensional reasons the only differentials that can be supported
between indecomposables are $d_{k}(u_{2k-1})=\mu_kc_k$ modulo decomposables, 
where $\mu_k \in \mathbb{C}-\{0\}$ for $2 \leq k \leq n$.
As the original Chern classes were algebraically independent, redefining
$c_k$'s as necessary we can assume $\mu_k=1$ for $2 \leq k \leq n$.

Thus we conclude:
\begin{thm}
\label{thm: spectralsl_n}
The spectral sequence $E_r^{*,*}[\mathfrak{sl}_n]$ is given by
$$
E_1^{*,*}[\mathfrak{sl}_n] = \Lambda^*(u_3,u_5,\dots, u_{2n-1}) \otimes
\mathbb{C}[c_2,c_3,\dots,c_n]
$$
where $u_i \in E_1^{0,i}$ and $c_i \in E_1^{i,0}$. The higher differentials
are determined by the fact that the $c_i$ are permanent cycles and
$d_r(u_{2r-1})=c_r$ for $2 \leq r \leq n$. (In particular $d_i(u_{2r-1})=0$ for $1 \leq i < r$.) Thus in particular
$E_{n+1}^{*,*}=E_{\infty}^{*,*}$ is the cohomology of a point.

If $B^*_r$ denotes the Bockstein spectral sequence of the $p$-group
$G(\mathfrak{sl}_n(\mathbb{F}_p))$, which is the congruence kernel
of the reduction $SL_n(\mathbb{Z}/p^3\mathbb{Z}) \to SL_n(\mathbb{F}_p)$, 
then for any Hodge degree $N$,
$$
B^*_2=E_1^{*,*}[\mathfrak{sl}_2(\mathbb{F}_p)]
\sim \Lambda^*(u_3,u_5,\dots,u_{2n-1}) \otimes \mathbb{F}_p[c_2,c_3,\dots,c_n]
$$
where $\sim$ indicates isomorphism of complexes up
to Hodge degree $N$ for all but finitely many primes (depending on $N$).

\end{thm}
\begin{proof}
All but the last paragraph follows from the observations before the statement of the theorem. The last paragraph follows from the fundamental
comparison theorem \ref{thm: FCT} discussed in previous sections.
In this instance we compare $B_2^*=E_1^{*,*}[\mathfrak{sl}_n(\mathbb{F}_p)]$
to $E_1^{*,*}[\mathfrak{sl}_n(\mathbb{C})]$ through
$E_1^{*,*}[\mathfrak{sl}_n(\mathbb{Z})]$. Excluding the finite set of primes where torsion occurs
in $E_1^{*,*}[\mathfrak{sl}_n(\mathbb{Z})]$ in Hodge degrees $\leq N$, we have by Theorem \ref{thm: FCT} that
$$
\dime_{\mathbb{F}_p}(E_1^{s,t}[\mathfrak{sl}_n(\mathbb{F}_p)])=\dime_{\mathbb{C}}(E_1^{s,t}[\mathfrak{sl}_n(\mathbb{C}])
$$
for arbitrary $t$ and $s \leq N$.

There is a subtlety regarding the ring structure,
as the fundamental comparison theorem just shows that the dimensions match up; there is a divided
power issue that needs to be addressed when excluding the finite set of primes. To clarify, we know
that $A^{*,*}=E_1^{*,*}[\mathfrak{sl}_n(\mathbb{Z})]=H^*(E_0^{*,*}[\mathfrak{sl}_n(\mathbb{Z})],d_0)$
modulo torsion is a bigraded abelian group
where $A^{i,j}$ is free abelian with rank equal to the dimension (as a complex vector space) of
$E_1^{i,j}[\mathfrak{sl}_2(\mathbb{C})]$.  Thus one can choose generators in $A^{0,i}$ corresponding
to the $u_i$ (and call them the same thing) and in $A^{i,0}$ corresponding to the $c_i$ (and call them the same
thing). However in $A^{*,*}$ it is not true in general that a combination $u_3^{\alpha_2} \dots u_{2n-1}^{\alpha_n}
c_2^{a_2}\dots c_n^{a_n}$ ($0 \leq \alpha_i \leq 1$, $a_i$ nonnegative integers),
is a generator in its corresponding group; all that is known is that it will
be {\bf a nonzero integer times a generator}. Thus while this element is part of a basis
in $E_1^{*,*}[\mathfrak{sl}_n(\mathbb{C})]$ it is not necessarily a generator element in
$E_1^{*,*}[\mathfrak{sl}_n(\mathbb{Z})]$ but only a nonzero integral multiple $v$ of one.
(In the case where more than one of these elements lie in the same location $A^{i,j}$, $v$ should
be replaced by the determinant of the nonsingular matrix which expresses the collection of these
elements in $A^{i,j}$ in terms of a $\mathbb{Z}$-basis of $A^{i,j}$.)

Now when we reduce to $E_1^{*,*}[\mathfrak{sl}_n(\mathbb{F}_p)]$, if the nonzero multiple $v$ mentioned
above is a multiple of $p$, the corresponding product becomes zero (or linearly dependent in the case
of more than one element in the same position in the $E_1$-page), while if $v$ is not a multiple of $p$,
then the mod $p$ reduction will be a basis element in the corresponding position in the $E_1$-page.
Since we know the complex and mod $p$ dimensions match up through Hodge degree $N$, to ensure the correct
algebra structure, we just have to avoid all primes $p$ dividing the aforementioned multiples $v$ in the
finite number of locations up through Hodge degree $N$ in addition to all the primes where $p$-torsion
occurs in $E_1^{*,*}[\mathfrak{sl}_n(\mathbb{Z})]$ up through the same Hodge degree.
Since this collection of primes is a finite set the theorem follows.

We will see a bit more detail on this issue when working out the ring structure of
$E_1^{*,*}[\mathfrak{sl}_2(\mathbb{F}_p)]$ in later sections.

\end{proof}

Note: we would like to emphasize that it is only known
that $(E_0^{*,*},d_0)=(B_1^*,\beta)$ and so $E_1^{*,*}=B_2^*$ by
the work in \cite{BrP}. It is not known that the higher differentials
agree though often this is forced for dimensional reasons.

As a consequence of the last theorem we see for a generic odd prime $p$,
in the Bockstein spectral sequence for $G(\mathfrak{sl}_n(\mathbb{F}_p))$,
$B_2^*=E_1^{*,*}[\mathfrak{sl}_2(\mathbb{F}_p)]$ looks like
$\Lambda^*(u_3,u_5,\dots,u_{2n-1}) \otimes \mathbb{F}_p[c_2,c_3,\dots,c_n]$
for low Hodge degrees. However, 
when the Hodge degree approaches $p$, this will break down. One way to see
that this has to happen is to note that for low Hodge degree,
$E_1^{*,*}[\mathfrak{sl}_n(\mathbb{F}_p)]=B_2^*$ will
look like
it has Krull dimension $n-1$, the rank of $\mathfrak{sl}_n$.
However once the Hodge degree passes $p$, the existence of an invariant
polynomial algebra on the $p$-powers of the original dual basis of
$\mathfrak{sl}_n[\mathbb{F}_p]$ will exhibit a Krull dimension of $n^2-1$,
the dimension
of $\mathfrak{sl}_n$.

Thus as we will see explicitly in future sections, in the low dimensions
the higher torsion $B_2^*$ is governed by a ``characteristic zero''
contribution which has Krull dimension given by the {\bf rank}
of the underlying
Lie algebra. However when the Hodge degree approaches $p$, a
{\bf phase transition }
occurs and $B_2^*$ explodes behaving now like an algebra which has
Krull dimension given by the
{\bf dimension} of the underlying Lie algebra.
This phenomenon occurs for all the simple Lie algebras and we will
work it out explicitly in further sections for $\mathfrak{sl}_2$.

In the remainder of this section, we run quickly through the other
families of simple complex Lie algebras as well as the exceptional ones.
We will not put in as many details as for the family $\mathfrak{sl}_n$, 
as that would make the paper unwieldy, and our primary concern is the
characteristic $p$ considerations. All of the basic calculations needed
can be found in Mimura and Toda's treatise, \cite{MT} together
with Borel's result that for a compact Lie group $G$ with maximal
torus $T$ we have
$$
H^*(BG,\mathbb{Q}) \cong H^*(BT,\mathbb{Q})^W, 
$$
where $W$ is the Weyl group. Letting $\mathfrak{g}$ and
$\mathfrak{h}$ be the corresponding Lie algebra and Cartan subalgebra one
then can show as done above for $\mathfrak{sl}_n$ that
$$
H^*(BG,\mathbb{C}) \cong H^*(BT,\mathbb{C})^W \cong Poly(\mathfrak{h}^*)^W \cong 
H^0(\mathfrak{g},Poly(ad^*)) = H^0(\mathfrak{g},U(\mathfrak{g})^*).
$$
Thus in all cases for a complex simple Lie algebra $\mathfrak{g}$
with corresponding compact form $\mathfrak{g}_{\mathbb{R}}$ and
compact, connected Lie group $G$, we have
$$
E_1^{*,*}[\mathfrak{g}] \cong H^*(G,\mathbb{C}) \otimes H^*(BG,\mathbb{C}).
$$
However, it is important to note that everything in $E_1^{*,*}$ is expressed
in terms of the Lie algebra $\mathfrak{g}$ and invariant polynomial
functions and forms on it. In each case the differentials transgress
from the ``fiber'' $H^*(G,\mathbb{C})$ to the ``base'' $H^*(BG,\mathbb{C})$
for pretty much the same reasons as in the $\mathfrak{sl_n}$ case.
Thus over the complex numbers, $E_r^{*,*}[\mathfrak{g}]$
functions very much like
the spectral sequence for the fibration $G \to EG \to BG$.

This is why we will refer to $E_r^{*,*}[\mathfrak{g}]$ as the {\bf classifying
spectral sequence} of the Lie algebra $\mathfrak{g}$. It provides an
algebraic model which can be used even when the Lie algebras are
defined over fields of characteristic $p$.

Now without further ado, we summarize the calculation of this
spectral sequence for the rest of the complex simple Lie algebras,
in each case this gives a picture of the higher torsion in the
Bockstein spectral sequence of corresponding $p$-groups,  before
the ``char 0 to char p'' phase transition.

The family $\mathfrak{sp}_{2n}(\mathbb{C})$ can be described as follows
after suitable choice of basis (see section 16.1 of \cite{FH}):
$$
\mathfrak{sp}_{2n}(\mathbb{C}) =
\left\{ \begin{bmatrix} A & B \\ C & D \end{bmatrix} | B^T=B, C^T=C, D=-A^T,
A,B,C,D \in \mathfrak{gl}_n(\mathbb{C}) \right\}.
$$
It has dimension $2n^2+n$ and rank $n$ with a Cartan subalgebra
given by diagonal matrices subject to the constraint $D=-A^T$. The
Weyl group in this picture is generated by permutations of the entries
of $A$ and $D$ simultaneously and consistently with $D=-A^T$ and also
by pairwise swaps between diagonal elements in $A$ with their negatives
in $D$. With this, it is not hard to show that Weyl group invariant
polynomial functions on $\mathfrak{h}$ restrict to polynomials
in the elementary symmetric functions on the $n$ diagonal entries in $A$
together with the condition that they be invariant under negation
of any variable, and thus polynomials in $S_2,S_4,\dots$ where
$S_{2i}$ is the $i$th elementary symmetric polynomial in the squares
of the variables.
Note $S_{2i} \in H^0(\mathfrak{g},Poly(ad^*))$ is polynomial
of degree $2i$ and hence in $E_1^{2i,0}$, and of total degree $4i$.

It turns out $\mathfrak{so}_{2n+1}(\mathbb{C})$ has the same
rank $n$ and Weyl group as $\mathfrak{sp}_{2n}(\mathbb{C})$
even though they are not isomorphic as Lie algebras in general.
This however means that the two corresponding compact forms,
the compact symplectic group $Sp(n)$ for $Sp_{2n}(\mathbb{C})$
and $SO(2n+1)$ for $\mathfrak{so}_{2n+1}(\mathbb{C})$ have the
same cohomology with rational (and hence complex) coefficients.
The same thing goes for the cohomology of their classifying spaces
by Borel's theorem. Thus we get:

\begin{thm}[Symplectic and odd orthogonal Lie algebras]
\label{thm: ComplexSpSpectral}

Let $\mathfrak{g}$ be either $\mathfrak{so}_{2n+1}(\mathbb{C})$ or
$\mathfrak{sp}_{2n}(\mathbb{C})$. Then

\begin{align*}
\begin{split}
E_1^{*,*}[\mathfrak{g}] &\cong H^*(SO(2n+1),\mathbb{C}) \otimes
H^*(BSO(2n+1),\mathbb{C}) \\
&\cong H^*(Sp(n),\mathbb{C}) \otimes H^*(BSp(n),\mathbb{C}) \\
&\cong  \Lambda^*(u_3,u_7,\dots,u_{4n-1}) \otimes
\mathbb{C}[P_1,P_2,P_3,\dots,P_n],
\end{split}
\end{align*}
where $u_i \in E_1^{0,i}$ and $P_i \in E_1^{2i,0}$ correspond to the
universal Pontryagin classes and $Sp(n)$ is the compact symplectic
group.
In the spectral sequence, the $P_i$'s are permanent cycles and we have
$d_{2r}(u_{4r-1})=P_r$ for $1 \leq r \leq n$. Thus $E_{2n+1}^{*,*}=E_{\infty}^{*,*}$ is the cohomology of a point. Note the rank of
$\mathfrak{so}_{2n+1}(\mathbb{C})$ and $\mathfrak{sp}_{2n}(\mathbb{C})$
is $n$ while the dimension is $2n^2+n$.

As usual this means that the 2nd term of the Bockstein spectral
sequence for the $p$-groups $G(\mathfrak{so}_{2n+1}(\mathbb{F}_p))$
and $G(\mathfrak{sp}_{2n}(\mathbb{F}_p))$ satisfy:
$$
B_2^*=E_1^{*,*}[\mathfrak{sp}_{2n}(\mathbb{F}_p)] \sim \Lambda^*(u_3,u_7,\dots,u_{4n-1})
\otimes \mathbb{F}_p[P_1,P_2,\dots,P_n],
$$
where $\sim$ means that the two complexes are isomorphic in the range of Hodge degrees $\leq N$
for all but finitely many primes (that depend on $N$).

\end{thm}

We now will just list the results for the other complex simple Lie
algebras, the implications for higher torsion in the corresponding $p$-groups
being understood:

\begin{thm}(Even orthogonal Lie algebras)
For $n \geq 4$,
\begin{align*}
\begin{split}
E_1^{*,*}[\mathfrak{so}_{2n}] &\cong H^*(SO(2n),\mathbb{C}) \otimes
H^*(BSO(2n),\mathbb{C}) \\
&\cong  \Lambda^*(u_3,u_7,\dots,u_{4n-5},v) \otimes
\mathbb{C}[P_1,P_2,P_3,\dots,P_{n-1},e],
\end{split}
\end{align*}
where $u_i \in E_1^{0,i}$, $v \in E_1^{0,2n-1}$.
$P_i \in E_1^{2i,0}$ correspond to the
universal Pontryagin classes and $e \in E_1^{n,0}$
corresponds to the universal Euler class and both are permanent
cycles in this spectral sequence. Furthermore we have
$d_{2r}(u_{4r-1})=P_r$ for $1 \leq r \leq n-1$ and $d_{n}(v)=e$.
Note $\mathfrak{so}_{2n}$
has rank $n$ and dimension $2n^2-n$.
\end{thm}

\begin{thm}(Exceptional simple Lie Algebras)
\label{thm: ExcepSpec}
$$
E_1^{*,*}[\mathfrak{g}_2] \cong \Lambda^*(u_3,u_{11}) \otimes
\mathbb{C}[K_2,K_6] 
$$
$$
E_1^{*,*}[\mathfrak{f}_4] \cong \Lambda^*(u_3,u_{11},u_{15},u_{23}) \otimes
\mathbb{C}[K_2,K_6,K_8,K_{12}] 
$$
$$
E_1^{*,*}[\mathfrak{E}_6] \cong \Lambda^*(u_3,u_9,u_{11},u_{15},u_{17},u_{23}) \otimes
\mathbb{C}[K_2,K_5,K_6,K_8,K_9,K_{12}] 
$$
$$
E_1^{*,*}[\mathfrak{E}_7] \cong \Lambda^*(u_3,u_{11},u_{15},u_{19},u_{23},u_{27},u_{35})\
 \otimes
\mathbb{C}[K_2,K_6,K_8,K_{10},K_{12},K_{14},K_{18}] 
$$
$
E_1^{*,*}[\mathfrak{E}_8] \cong
$ \\ 
$
\Lambda^*(u_3,u_{15},u_{23},u_{27},u_{35},u_{39},u_{47},u_{59})
\otimes
\mathbb{C}[K_2,K_8,K_{12},K_{14},K_{18},K_{20},K_{24},K_{30}]
$
\vspace{0.1in}

In all cases, the dimension of the Lie algebra is the sum of the subscripts of the $u_i \in E_1^{0,i}$
in its cohomology.
As usual, the dual Casimirs $K_i \in E_1^{i,0}$ represent homogeneous degree $i$ adjoint invariant polynomial functions on
the Lie algebra and $K_2$ always is the quadratic form corresponding to the Killing form.
\end{thm}

The results above cover all the complex simple Lie algebras. More generally a semisimple Lie algebra is
a direct sum of simple Lie algebras and is covered by the fact that
if $\mathfrak{g}=\mathfrak{g}_1 \oplus \mathfrak{g}_2$ then 
$E_0^{*,*}[\mathfrak{g}] = E_0^{*,*}[\mathfrak{g}_1] \otimes E_0^{*,*}[\mathfrak{g}_2]$
as differential graded algebras. Furthermore the complex also decomposes as a tensor product with respect
to the commuting differential used to define the spectral sequence for the double complex. Thus one has
that
$$
E_r^{*,*}[\mathfrak{g}] = E_r^{*,*}[\mathfrak{g}_1] \otimes E_r^{*,*}[\mathfrak{g}_2]
$$
as dga's for any $r \geq 0$, i.e., the spectral sequence of the sum is the tensor product of the
spectral sequences of the factors. (These comments in fact work over any base field $\bf{k}$.)
Thus one can work out the spectral sequence for any complex semisimple Lie algebra from the results in
this section. A general complex Lie algebra $\mathfrak{L}$ fits in a short exact sequence
$$
0 \to \rad(\mathfrak{L}) \to \mathfrak{L} \to \mathfrak{g} \to 0
$$
where $\rad(\mathfrak{L})$ is the largest solvable ideal of $\mathfrak{L}$ and
$\mathfrak{g}$ is semisimple, so it would remain to treat solvable (and in particular nilpotent) Lie
algebras. Generally the cohomology of these is much more complicated, though definitely interesting
but we will not consider these much in this paper nor the corresponding $p$-groups as the treatment
would require different methods.

We have shown that for any complex semisimple Lie algebra $\mathfrak{g}$, with compact form
$\mathfrak{g}_{\mathbb{R}}$ corresponding to a compact connected Lie group $G$, that
$$
E_1^{*,*}[\mathfrak{g}]=H^*(\mathfrak{g},U(\mathfrak{g})^*)=H^*(G \times BG, \mathbb{C}).
$$
We note finally that it is well known that
$$
H^*(\Lambda BG, \mathbb{C}) \cong
H^*(G \times BG, \mathbb{C})
$$
by considering the fibration $G \simeq \Omega BG \to \Lambda BG \to BG$.

A quick comment regarding related physics: For $\mathfrak{g}$ a complex semisimple Lie algebra we have
seen that the Casimir algebra will be a polynomial algebra on $h$ generators where $h$ is the rank of the Lie algebra.
In general this means the $h$ linearly independent, commuting operators represented by the Cartan algebra $\mathfrak{h}$
will commute with an additional $h$ Casimir operators which generate the Casimir algebra. This gives a total
of $2h$ commuting and simultaneously measurable operators coming from the Cartan algebra and the Casimir
generators, which form (part of) a good system of variables in the underlying physical system.

Thus for example $\mathfrak{sl}_3(\mathbb{C})$ which plays a role in quark/anti-quark theory as well
as meson/baryon classification will have two variables coming from the Cartan algebra, the corresponding
measurements of which will correspond to the root weights described in later sections of this paper and
an additional two Casimir generators dual to the Chern classes $c_2, c_3$, to give a total
of 4 commuting operators from this discussion. For a nice mathematical 
article on the related physics, see \cite{BH}.

\section{Exponent theory for finite $p$-groups}

In this section we recall the exponent theory of finite groups. Many of these results were
introduced in work of A. Adem and Ian Leary (see \cite{Adem} and \cite{Le}) while the fact that $e_{\infty}(P) \neq e(P)$ was
first shown in \cite{Pk}. We summarize the main results regarding these exponents in this section and include a
few more for the category of $p$-groups we will be dealing with.

In this section let $G$ be a finite group. It is well known that the integral cohomology groups
$H^n(G,\mathbb{Z}), n > 0$ are finite abelian groups and $|G| \cdot H^n(G,\mathbb{Z}) = 0$ for all $n > 0$.
Let $\bar{H}(G) = \oplus_{n=1}^{\infty} H^n(G,\mathbb{Z})$ denote the total reduced
integral cohomology of $G$. We define the following exponents:

\begin{defn}(Exponent Definitions) \\
$exp(G) = min \{n \geq 1 | g^n=e \text{ for all } g \in G \}$. \\
$e_{\infty}(G)= min \{ n \geq 1 | n\bar{H}(G) \text{ is a finite set} \}$. \\
$e(G)=min \{ n \geq  1 | n\bar{H}(G)=0 \}$. \\
\end{defn}

The following propositions are discussed in detail in \cite{Pk} so we will not reproduce proofs here.
As mentioned before many of them occur in previous independent work of Adem and Leary.
Failure of various propositions for exactly one of $e_{\infty}$ or $e$ were shown in \cite{Pk}.
We write $n | m$ for ``the integer $n$ divides the integer $m$''.

\begin{pro}(Basic Relationship between Exponents)
If $P$ is a finite $p$-group then
$$
exp(P) \text{ } \Big{|} \text{ } e_{\infty}(P) \text{ } 
\Big{|} \text{ } e(P) \text{ } \Big{|} \text{ } |P|
$$
and furthermore there are examples of $p$-groups that show that these four quantities
are different in general.
\end{pro}

\begin{pro}(Subgroups)
Let $P_1 \leq P_2$ be finite $p$-groups. Then \\
$e_{\infty}(P_1) | e_{\infty}(P_2)$, $exp(P_1) | exp(P_2)$ and $|P_1| 
\Big{|} |P_2|$. However
$e(P_1) \not| \text{ } e(P_2)$ in general.
\end{pro}

\begin{pro}(Cyclic groups and Symmetric Sylow subgroups)
Let $C$ be a finite cyclic group and $S(m)$ be the Sylow $p$-subgroup of
the symmetric group on $m$ letters. \\
(1) $exp(C)=e_{\infty}(C)=e(C)=|C|$ \\
(2) $p^n=exp(S(p^n))=e_{\infty}(S(p^n)) = e(S(p^n))$.
\end{pro}

\begin{pro}(Products)
If $G=G_1 \times \dots \times G_n$, then
$$exp(G)=lcm_{i=1,\dots,n}\{ exp(G_i) \},$$
$$e_{\infty}(G)=lcm_{i=1,\dots,n}\{ e_{\infty}(G_i) \},$$
$$e(G)=lcm_{i=1,\dots,n}\{ e(G_i) \}$$
where $lcm$ stands for least common multiple and can be replaced with $max$
in the case of $p$-groups.
Thus if $p^n$ is the largest power of $p$ less than or equal to $m$, then
$$
p^n = exp(S(m))=e_{\infty}(S(m))=e(S(m)).
$$
\end{pro}
\begin{proof}
The first statement follows from an application of K\"unneth's Theorem. The second follows from
the structure of Sylow $p$-subgroups of symmetric groups. More precisely, if the $p$-adic expansion of
$m$ is given by $m=a_0 + a_1p + a_2p^2 + \dots + a_np^n$ then $S(m)$ is the direct product
of $a_0$ copies of $S_1$, $a_1$ copies of $S(p) \dots$, up to $a_n$ copies of $S(p^n)$.
\end{proof}

\begin{thm}(Exponent Characterizations)
Let $P$ be a $p$-group. \\
(1)[Nakayama-Rim] If $e_{\infty}(P)=1$ then $P=\{e\}$ and $e(P)=1$ also. \\
(2)[Adem] If $e_{\infty}(P)=p$ then $P$ is elementary abelian and $e(P)=exp(P)=p$ also. \\
(3)[Pakianathan] For odd primes $p$, the group $G(\mathfrak{sl}_2(\mathbb{F}_p))$ has
$exp(P)=e_{\infty}(P)=p^2$ and $e(P)=p^3$ while $|P|=p^6$. Thus $e_{\infty}(P) \neq e(P)$ in general.
\end{thm}

\begin{pro}(Faithful actions on finite sets)
Let $P$ be a finite $p$-group. \\
(1) If $P$ has a faithful action on a set of size $m$ and $p^n$ is the largest power of
$p$ less than or equal to $m$ then $e_{\infty}(P) | p^n$. Thus $e_{\infty}(P)$ is a lower bound
on the size of a set on which $P$ can act faithfully. \\
(2) If the intersection of all subgroups of $P$ of index $p^n$ is trivial then
$e_{\infty}(P) | p^n$. This statement does not hold with $e(P)$ replacing $e_{\infty}(P)$.
\end{pro}

\begin{thm}(Browder's exponent theorem and consequences)
Let $G$ be a finite group and $X$ a finite dimensional {\bf free} $G-CW$-complex, with
homologically trivial $G$-action. Then
$$
|G| \text{ } \Big{|} \text{ } 
\prod_{n=2}^{\infty} exp(H^n(G,H_{n-1}(X))) \text{ } \Big{|} 
\text{ } e(G)^{s(X)}
$$
where $s(X)$ is the number of positive dimensions in which the integral homology of $X$
is nontrivial.

A corollary of this is if $E$ is an elementary abelian group acting freely and homologically
trivially on a product of $N$ equal dimensional spheres then $rank(E) \leq  N$.
\end{thm}
\begin{proof}
The short and elegant proof of this theorem and corollary can be found in \cite{Br}.
The second division comes from the definition of $e(G)$. It is unknown to the authors at this time
whether the theorem would hold for $e_{\infty}(G)$ but we see no reason why it should.
\end{proof}

The two fundamental exponents can be recharacterized in terms of the behaviour of the Bockstein
spectral sequence of the $p$-group $P$. A discussion of this spectral sequence can be found in many
places in the literature and is described for example in \cite{BrP}. Recall that when applied to the
classifying space of a finite $p$-group $P$ one has that
$B_1^*=H^*(P,\mathbb{F}_p)$ and $\beta_1$ is the Bockstein. Nonzero permanent cycles in $B_r$ represent
elements of order $p^r$ or greater in the integral cohomology of $P$ and $B_{\infty}$ is the cohomology
of a point. Thus one has the following:

\begin{lem}(Bockstein Spectral Sequence Characterization)
If $P$ is a finite $p$-group, then:\\
(1) $e_{\infty}(P)=p^s$ if and only if $B_{s+1}^*$ is the first page of the Bockstein spectral sequence
whose total complex is finite. \\
(2) $e(P)=p^k$ if and only if $B_{k+1}^*$ is the first page of the Bockstein spectral sequence whose
total complex is concentrated in degree 0. (equivalently is 1-dimensional).
\end{lem}

Finally for $p$-groups $G(\mathfrak{L})$ corresponding to $\mathbb{F}_p$-Lie algebras 
$\mathfrak{L}$ as in
\cite{BrP}, where $0 \to \mathfrak{L} \to G(\mathfrak{L}) \to \mathfrak{L} \to 0$, we have the following:

\begin{pro}(Exponent bounds for $G(\mathfrak{L})$-$p$-groups.)
Let $\mathfrak{L}$ be a $\mathbb{F}_p$ Lie algebra such that the intersection of
index $p^{\beta}$ sub Lie algebras is zero. Then
$e_{\infty}(G(\mathfrak{L})) | p^{2\beta}$.
\end{pro}
\begin{proof}
In the construction of the $p$-groups $G(\mathfrak{L})$ one has in general that
$|G(\mathfrak{L})| = |\mathfrak{L}|^2$ and so an index $p^{\beta}$ sub Lie algebra of
$\mathfrak{L}$
gives a corresponding index $p^{2\beta}$ subgroup of $G(\mathfrak{L})$. Since the zero
Lie subalgebra corresponds to the trivial group, the rest follows from the previous theorems on faithful
group actions.
\end{proof}

As the reader can see, these exponents play an interesting role in the cohomology of groups limited
only by the difficulty in carrying out complete calculations of the integral cohomology of groups.
In some results the integral cohomology exponent $e(G)$ arises while in others the asymptotic
exponent $e_{\infty}(G)$ arises. Since these are not equal in general it is useful to define the extent
to which they differ, and the following concept helps do that.

\begin{defn}(Exceptional Torsion elements and Exceptional Dimension)
Let $P$ be a finite $p$-group. \\
By definition $e_{\infty}(P) \cdot \bar{H}^*(P,\mathbb{Z})$ is a finite graded ring.
Let the {\bf exceptional dimension} of $P$, denoted by $ED(P)$, be defined as the largest
dimension in which $e_{\infty}(P) \cdot H^*(P,\mathbb{Z})$ is nonzero. Note that $ED(P)$ is always a finite
nonnegative integer and is zero if and only if $e_{\infty}(P)=e(P)$.

Any element $x \in H^*(P,\mathbb{Z})$ with $0 < * \leq ED(P)$ such that $e_{\infty}(P) \cdot x \neq 0$
is called an {\bf exceptional torsion element}. Note if $ED(P) > 0$ then an exceptional torsion element
always exists in $H^{ED(P)}(P,\mathbb{Z})$ by definition and no exceptional torsion elements exist
in dimensions higher than $ED(P)$, i.e., $e_{\infty}(P) \cdot \oplus_{n=ED(P)+1}^{\infty}H^{n}(P,\mathbb{Z})=0$.

\end{defn}

There is no bound on exceptional dimension in general, i.e., it can be arbitrarily high as the next
example shows:

\begin{thm}
Let $P$ be the $N$-fold direct product of $G(\mathfrak{sl}_2)$'s.
Then $e_{\infty}(P)=p^2$, $e(P)=p^3$. $|P|=p^{6N}$ and $ED(P)=N \cdot ED(G(\mathfrak{sl}_2)) > 0$. \\
Thus $\lim_{N \to \infty} ED(P) = \infty$.
\end{thm}
\begin{proof}
Follows from the results in this section together with the fact that \\
$e_{\infty}(G(\mathfrak{sl}_2))=p^2$
and $e(G(\mathfrak{sl}_2))=p^3$ proven in \cite{Pk}. Note that the Bockstein spectral sequence of $P$
is the $N$-fold tensor product of the Bockstein spectral sequence of $G(\mathfrak{sl}_2)$ (this means at
each page it's the tensor product as dga's.). Also note that $ED(G(\mathfrak{sl}_2))$ is the maximum
nonzero dimension of $B_3^*(G(\mathfrak{sl}_2))$ while $ED(P)$ is the maximum nonzero dimension of
$B_3^*(P) = \otimes_{j=1}^N B_3^*(G(\mathfrak{sl}_2))$ from which the results for essential
dimension follow.
\end{proof}

The last example is not so strong in the sense that $\log_p(|P|)$ has to increase in order to get the
increase in $ED(P)$. In the next sections we will study the Bockstein spectral sequence for
$G(\mathfrak{sl}_2)$ again more carefully using the machinery of Lie algebras we have mentioned in previous
sections. More analysis shows in fact the following stronger facts: \\
(1) $ED(G(\mathfrak{sl}_2(\mathbb{F}_p))) \geq 2p-2$ for odd primes $p$.
Thus the essential dimension grows with the prime of
definition of $G(\mathfrak{sl}_2(\mathbb{F}_p))$ and is therefore, in some sense, not even bounded for the
fixed ``{\bf group scheme}'' $G(\mathfrak{sl}_2(-))$. \\
(2) Furthermore we will see that the ``exceptional torsion elements correspond to the characteristic
zero contribution of the Lie algebra scheme $\mathfrak{sl}_2(-)$ and has Krull dimension given
by the rank of the Lie algebra (scheme). Then in the vicinity of the essential dimension a ``char 0 to char p''
phase transition occurs and the asymptotic torsion has a larger Krull dimension equal to
the dimension of the underlying Lie algebra.  \\
(3) In this sense, the word ``exceptional torsion'' transition to ``asymptotic torsion'' can be seen
in some situations at least, to correspond to a char 0 to char p phase transition in corresponding
Lie algebra schemes.

This should become clearer after a few examples.

\section{Non-abelian Lie algebra of dimension 2}
We will work out an easy example which will introduce
the useful concept of weight stratification of the spectral sequence.

Let $\mathfrak{g}$ be the non-abelian Lie algebra (scheme) of dimension 2 with
basis $x_h, x_e$ and commutator given by $[x_e,x_h]=x_e$. Note $\mathfrak{g}$ can be defined over
any base ring {\bf k} and we will write $\mathfrak{g}(\bf{k})$ when we want to make the base ring clear.
Let $h,e$ denote the dual basis in $\mathfrak{g}^*$ and $H, E$ suspended copies in $Poly^1(\mathfrak{g}^*)$.
Then
$$
E_0^{*,*}[\mathfrak{g}] = \Lambda^*(h,e) \otimes {\bf k}[H,E]
$$
with differential given by
$d_0(h)=d_0(H)=0$ and $d_0(e)=he, d_0(E)=Eh-He$. The (anti)commuting differential $d_1$ giving rise
to the spectral sequence is given by $d_1(h)=H, d_1(e)=E, d_1(E)=d_1(H)=0$.
Furthermore by \cite{BrP}, $(E_0,d_0)=(H^*(G(\mathfrak{g}),\mathbb{F}_p),\beta)$ where $\beta$ is the
Bockstein operator when $\bf{k}=\mathbb{F}_p$.

Now if we assign $h, H$ to have weight $0$ regarded as an integer and $e, E$ to have weight $1$,
and extend weight to $E_0^{*,*}[\mathfrak{g}]$ so that the weight of a product is the sum of the weights
of the individual factors, then it is easy to see that $d_0$ and $d_1$ preserve weight and
hence so do all differentials of the spectral sequence. Thus the spectral sequence decomposes into a
direct sum of spectral sequences given by isolating terms of specific weight:
$$
E_r^{*,*}[\mathfrak{g}] = \bigoplus_{w \in \mathbb{N}} E_r^{*,*}[w]
$$
We will see this occurs in general in the next section, and in the case of semisimple Lie algebras,
the sum will be indexed by the {\bf root lattice} of the corresponding complex semisimple Lie algebra
which will be a free abelian group of rank equal to the rank of the Lie algebra. The weight stratification 
is overkill in this example, but becomes necessary in more complicated computations.

Let us first look at the weight 0 contribution; since the weights in this example are all nonnegative
integers, it is easy to see that
$$
E_0^{*,*}[0] = \Lambda^*(h) \otimes {\bf k}[H],
$$
with $d_0=0$ identically and $d_1(h)=H$. Thus $E_0^{*,*}[0]=E_1^{*,*}[0]$ and
$E_2^{*,*}[0]=E_{\infty}^{*,*}[0]$ is the cohomology of a point.

Now consider a nonzero weight $n > 0$. Then $E_0^{*,*}[n]$ is a right $E_0^{*,*}[0]$-module
and since $d_0$ is a derivation which vanishes on $E_0^{*,*}[0]$, we have that
$$
d_0: E_0^{*,*}[n] \to E_0^{*,*}[n]
$$
is a right $E_0^{*,*}[0]$-module map. Furthermore $E_0^{*,*}[n]$ is a free $E_0^{*,*}[0]$-module
of rank 2 with (ordered) basis \{ $eE^{n-1}, E^n$ \}. Thus we can represent $d_0$ on this weight $n$-piece
as a $2 \times 2$ matrix with entries in $E_0^{*,*}[0]=\Lambda^*(h) \otimes {\bf k}[H]$.

Computing we have:
$$
d_0(eE^{n-1}) = (he)E^{n-1} - e(n-1)E^{n-2}(Eh-He) = eE^{n-1}(-nh)
$$
and
$$
d_0(E^n)=nE^{n-1}(Eh-He)=eE^{n-1}(-nH) + E^n (nh)
$$
and hence
$$
d_0(eE^{n-1}\gamma + E^n \delta)) = eE^{n-1}\gamma' + E^n \delta'
$$
for any $\gamma, \delta \in E_0^{*,*}[0]$ with
$$
\begin{bmatrix}
\gamma' \\ \delta'
\end{bmatrix}
= \begin{bmatrix}
-nh & -nH \\
0 & +nh
\end{bmatrix}
\begin{bmatrix}
\gamma \\ \delta
\end{bmatrix},$$
and thus $d_0: E_0^{*,*}[n] \to E_0^{*,*}[n]$ is represented by the matrix
$\begin{bmatrix}
-nh & -nH \\
0 & +nh
\end{bmatrix}$ with respect to the ordered basis mentioned above.

Restricting to the case where ${\bf k}$ is a field, there are now two cases
to consider: \\
Case 1: $n=0$ in ${\bf k}$. \\
In this case
$d_0 = 0$ on $E_0^{*,*}[n]$ and $E_0^{*,*}[n]=E_1^{*,*}[n]$. 
Direct computation shows $d_1(eE^{n-1})=E^n$ and we conclude that $E_2^{*,*}[n]=E_{\infty}^{*,*}[n]=0$.  \\
\\
Case 2: $n \neq 0$ in the field ${\bf k}$. \\
Note that a typical element in $E_0^{*,*}[0]=\Lambda^*(h) \otimes {\bf k}[H]$ can be written
in the form $\gamma_1 + h\gamma_2$ where $\gamma_i$ are ${\bf k}$-polynomials in $H$. If
$\begin{bmatrix} \gamma_1 + h\gamma_2 \\ \delta_1 + h \delta_2 \end{bmatrix}$ is in the kernel of the matrix then
$$
\begin{bmatrix}
0 \\ 0
\end{bmatrix}
= \begin{bmatrix}
-nh & -nH \\
0 & nh
\end{bmatrix}
\begin{bmatrix}
\gamma_1 + h\gamma_2 \\ \delta_1 + h\delta_2
\end{bmatrix}
=
\begin{bmatrix}
-nh(\gamma_1 + H\delta_2) -nH\delta_1 \\
+nh\delta_1
\end{bmatrix}.
$$
Since $n$ is a unit in ${\bf k}$ in this case, it follows quickly that $\delta_1=0$ and $\gamma_1+H\delta_2=0$.
Thus the typical element in the kernel of $d_0$ looks like $\begin{bmatrix} -H\delta_2 + h\gamma_2 \\
h\delta_2 \end{bmatrix}$. However such an element is in the image of $d_0$ as
$$
\begin{bmatrix}
-nh & -nH \\
0 & +nh
\end{bmatrix}
\begin{bmatrix}
\frac{1}{n}(-\gamma_2) \\ \frac{1}{n}(\delta_2)
\end{bmatrix}
= \begin{bmatrix} -H\delta_2 + h\gamma_2 \\
h\delta_2
\end{bmatrix}.
$$
Thus the cohomology of the complex $(E_0^{*,*}[n],d_0)$ is identically zero, i.e.,
$E_1^{*,*}[n]=E_{\infty}^{*,*}[n]=0$.

Summarizing to the main fields of interest $\mathbb{C}$ and $\mathbb{F}_p$ we get:
\begin{thm}[Non-abelian Lie algebra of dimension 2]
Let $\mathfrak{g}$ be the ${\bf k}$-Lie algebra with basis $x_h, x_e$ and bracket given by
$[x_e,x_h]=x_e$. Then $E_0^{*,*}=\Lambda^*(h,e) \otimes {\bf k}[H,E]$.
For ${\bf k}$ a field of characteristic zero we have
$$
E_1^{*,*}=H^*(\mathfrak{g},U(\mathfrak{g}^*)) = E_1^{*,*}[0]=\Lambda^*(h) \otimes {\bf k}[H]
$$
and $E_2^{*,*}=E_{\infty}^{*,*}$ is the cohomology of a point. Note the nonzero contribution to
$E_1^{*,*}$ is only from weight $0$ and that the Krull dimension of
$E_1$ is the rank of the Lie algebra which is one.

For ${\bf k}$ a field of characteristic $p$, we have
$$
E_1^{*,*}=H^*(\mathfrak{g},Poly(\mathfrak{g}^*)) = \Lambda^*(h, eE^{p-1}) \otimes {\bf k}[H, E^p]
$$
with $d_1(h)=H$ and $d_1(eE^{p-1})=E^p$. Note the nonzero contribution to $E_1^{*,*}$ is only from weights
that are congruent to $0$ modulo $p$. Also the Krull dimension is now two, which is the dimension of the
Lie algebra. The phase transition
from the characteristic $0$ answer first occurs in dimension where $eE^{p-1}$ lies
i.e., in total dimension $2p-1$ (Hodge degree $p-1$).
\end{thm}

Following \cite{BrP}, this yields the following corollary:
\begin{cor}
Let $\mathfrak{g}$ be the nonabelian Lie algebra above defined over $\mathbb{F}_p$, $p$ an odd prime.
Let $G(\mathfrak{g})$ be the corresponding $p$-group of order $p^4$. Then
$E_r^{*,*}=B_{r+1}^{*}$ for all $r$ where $B_r^*$ is the Bockstein spectral sequence of the group.
Thus
$$
E_1^{*,*}=B_2^*=\Lambda^*(h,eE^{p-1}) \otimes \mathbb{F}_p[H, E^p]
$$
and $E_2^{*,*}=B_3^*$ is the cohomology of a point.
Thus $e_{\infty}(G(\mathfrak{g}))=e(G(\mathfrak{g}))=p^2=exp(G(\mathfrak{g}))$.
\end{cor}
\begin{proof}
As mentioned before it was shown in \cite{BrP} that $(E_0^{*,*},d_0)=(B_1^*,\beta)$ in general
and hence that $E_1^{*,*}=B_2^*$. Using comparisons to the cyclic $p^2$-subgroups generated by the kernel
of $e$ and $h$ respectively (note $\Lambda^1(h,e)=H^1(G(\mathfrak{g}),\mathbb{F}_p)=
\Hom(G(\mathfrak{g}),\mathbb{F}_p)$), one sees that one must have $\beta_2(h)=H$ and $\beta_2(eE^{p-1})=E^p$
at least up to nonzero scalars. This shows that $B_3^*$ has the cohomology of a point and that the two
spectral sequences coincide on all pages. The rest follows immediately.

\end{proof}

Even though $e=e_{\infty}$ in this example, note that there is still a transition in the higher torsion $B_2^*$
from behaving like the characteristic zero contribution (krull dimension one) before Hodge degree $p-1$
and then changing to a Krull dimension two behaviour after Hodge degree $p$. Thus for $E_1^{*,*}$
the contribution $\Lambda^*(h,H)$ is universal while the contribution $\Lambda^*(eE^{p-1}) \otimes {\bf k}[E^p]$
only occurs in characteristic $p$.

We will see that this behaviour is relatively generic and occurs also in higher dimensional examples.

\section{Weight Stratification of the Spectral Sequence $E_r^{*,*}$}

Let $\mathfrak{g}$ be any complex semisimple Lie algebra and let $\mathfrak{h}$ be a Cartan subalgebra.
The adjoint action of $\mathfrak{h}$ on $\mathfrak{g}$ is (simultaneously) diagonalizable and decomposes
$\mathfrak{g}$ as
$$
\mathfrak{g}=\mathfrak{h} \oplus (\oplus_{\alpha \in R} \mathfrak{g}_{\alpha})
$$
where $R$ is the set of roots of the Lie algebra. Recall a root $\alpha$ with its root vector $v$ satisfies
$$
[h,v]=\alpha(h)v
$$
for all $h \in \mathfrak{h}$. It turns out that for a complex semisimple Lie algebra, there always exists a choice of 
basis for $\mathfrak{h}$ such that the values of $\alpha$ are integral on the $\mathbb{Z}$-span of this basis, and so
each root $\alpha$ can be viewed in $\mathbb{Z}^n$ where $n$ is the rank of the Lie algebra, i.e., the dimension
of its Cartan subalgebra. The correspondence comes by taking such a basis $\{ h_1,\dots,h_n \}$ of $\mathfrak{h}$
and sending $\alpha$ to $(\alpha(h_1),\dots,\alpha(h_n))$.

The Cartan algebra itself corresponds to $v$ such that $[h,v]=0=0v$ for all
$h \in \mathfrak{h}$, so we will say it has weight $0$ (though $0$ is usually not considered a root
for various reasons).

For a semisimple Lie algebra the set of roots is always finite, and if $\alpha$ is a root, then so
is $-\alpha$. Furthermore, the sublattice of $\mathbb{Z}^n$ spanned by the roots is of rank $n$ also
and is called the {\bf root lattice}.

We now extend this root decomposition to a weight decomposition of $E_0^{*,*}=\Lambda^*(\mathfrak{g}^*)
\otimes Poly(\mathfrak{g}^*)$ as follows. Decompose $\Lambda^1(\mathfrak{g}^*)=\mathfrak{g}^*=
Poly^1(\mathfrak{g}^*)$ by declaring something to have weight $\alpha$ if it is dual to a (nonzero) element
of weight $-\alpha$ in $\mathfrak{g}$. Then extend this weighing to all {\bf homogeneous elements} of $E_0^{*,*}$ by declaring
the weight of a product to be the sum of the weights. It is not hard to see that this decomposes
$E_0^{*,*}$ into a direct sum of ``weight subspaces'' where the weights are any nonnegative integer combination
of the root weights, i.e., are the elements of the {\bf root lattice}.

Though the minus sign is not really important, the reason the dual of a weight $\alpha$ element is
said to have weight $-\alpha$ is that if $v \in \mathfrak{g}$ is a root vector, 
$[h, v]=\alpha(h)v$ for all $h \in \mathfrak{h}$, and
$\phi \in \mathfrak{g}^*$ is a dual functional to $v$ (which vanishes on any root vector of different weight), 
then in the dual adjoint action,
$$
h \cdot \phi(s)=\phi([s,h])=-\phi([h,s])
$$
and so we have $h \cdot \phi = -\alpha(h)\phi$ for all $h \in \mathfrak{h}$. Thus it is reasonable to define
$\phi$ to have weight $-\alpha$. (However, the weight stratification and all the important properties
we need would still hold if we dropped the minus sign.)

Finally, it is well known that
$[\mathfrak{g}_{\alpha},\mathfrak{g}_{\beta}] \subseteq \mathfrak{g}_{\alpha + \beta}$. From this it is clear
from the definition of $d_0$ that it preserves weight as it does on the generating set.
Note $d_0(x_k)=-\sum_{i < j}^n c_{ij}^k x_i x_j$
and it is clear that $c_{ij}^k=0$ unless the weight of $x_k$ is the sum of the weights of $x_i$ and $x_j$.
Since $d_1$ also preserves weight trivially, we see since all differentials are derived from $d_0$ and $d_1$
that the spectral sequence respects the weight decomposition.

Thus letting $\Lambda(R)$ denote the root lattice we have a direct sum decomposition of spectral sequences
$$
E_r^{*,*}[\mathfrak{g}] = \bigoplus_{\alpha \in \Lambda(R)} E_r^{*,*}[\alpha].
$$

Note it is not necessary for a Lie algebra to be semisimple to have a weight decomposition. Indeed the example
in the last section is not semisimple. All that is needed is a decomposition of $\mathfrak{g}$ into
$\mathfrak{g}_{\alpha}$ over some abelian semigroup index $R \subseteq \mathbb{Z}^n$ with
$[\mathfrak{g}_{\alpha}, \mathfrak{g}_{\beta}] \subseteq \mathfrak{g}_{\alpha + \beta}$. Of course
the corresponding weight decomposition of $E_0^{*,*}$ will use as index set the set of nonnegative integer
combinations of $R$ which might not be a lattice if $-R \neq R$. This was the case in the example in the last
section as we only got a ray and not a lattice in $\mathbb{Z}$ as our weights.

Over a field of characteristic zero, $E_1^{*,*}=H^*(\mathfrak{g},U(\mathfrak{g}^*))=H^*(\mathfrak{g}) \otimes
\Gamma$ where $\Gamma$ are the adjoint invariant polynomial functions on $\mathfrak{g}$. Note
that if $f$ is an adjoint invariant polynomial function, then by definition $[g,f]=0=0f$ for all
$g \in \mathfrak{g}$ and so in particular for all $g \in \mathfrak{h}$. Thus an adjoint invariant polynomial
is by necessity of weight zero. So $\Gamma \subseteq E_1^{*,*}[0]$. However we have previously seen
that the generators in $H^*(\mathfrak{g})$ transgress to elements in $\Gamma$ in this spectral sequence
and so they also must have weight 0. Thus we conclude for a complex semisimple Lie algebra,
$H^*(\mathfrak{g},U(\mathfrak{g})^*)=E_1^{*,*}=E_1^{*,*}[0]$ and so only the weight zero contributions
survive in the $E_1$-page and beyond.

Over a field of characteristic $p$, it is still true that $E_1^{*,0}$ consists of invariant polynomials
and hence must have eigenvalue $0$ under the adjoint action of $\mathfrak{g}$ on itself. However since
our original weight decomposition used integral weights, we can only conclude that the weight of an invariant
polynomial is congruent to $0$ modulo $p$ as all those weights will be zero in the field.
Furthermore not every element of the spectral sequence transgresses to the line of invariant polynomials in
general but those which do will have weight congruent to $0$ mod $p$ of course.

We summarize this section's discussion in:

\begin{thm}(Weight stratification)
\label{thm: weight}
Let $\mathfrak{g}$ be a complex semisimple Lie algebra. Let $\mathfrak{g}_{\mathbb{Z}}$ be a corresponding
integral Lie algebra (always exists by Cartan-Serre basis). We can then get a corresponding lie algebra
$\mathfrak{g}_{\bf k}$ over any ring of definition ${\bf k}$.

The root decomposition of $\mathfrak{g}$ induces a weight decomposition of spectral sequences:

$$
E_r^{*,*}[\mathfrak{g}_{\bf k}] = \bigoplus_{\alpha \in \Lambda(R)} E_r^{*,*}[\alpha]
$$
where $\Lambda(R)$ is the root lattice of $\mathfrak{g}$ and is a free abelian group of rank equal
to the rank of $\mathfrak{g}$.

If ${\bf k}$ is a field of characteristic zero then
$$
E_1^{*,*}[\mathfrak{g}_{\bf k}]=E_1^{*,*}[0]=H^*(\mathfrak{g},U(\mathfrak{g})^*) ;
$$
i.e., only weight zero terms contribute anything to the $E_1$-page or beyond.

If ${\bf k}$ is a field of characteristic $p$ then
$E_1^{*,0}$ consists of polynomials of weight zero modulo $p$. Furthermore any element that transgresses 
to the polynomial line also has to have weight zero modulo $p$.
\end{thm}

Though this decomposition might seem technical, as the example in the last section shows, it helps organize
computations very effectively.

\section{Witt Lie algebras and Algebraic de Rham Cohomology}
\label{section: DeRham}
Since the computations in prime characteristic will be more difficult by far than in characteristic
zero, we will have to bring in the additional important computational viewpoint of algebraic $D$-modules.
We use only the most basic algebraic settings of this theory and recall the essential definitions below, 
but for an excellent introduction in the basic theory the reader is referred to \cite{Co}.
A basic discussion about a large family of infinite semisimple Lie algebras related to the classical
Witt Lie algebra can be found in \cite{NP}.

Let ${\bf k}$ be a field and $P={\bf k}[x_1,\dots,x_n]$ be a polynomial algebra in $n$ variables.
Formal differentiation $\frac{\partial}{\partial x_i}$ and multiplication by $x_i$ define
operators i.e., elements of $\Endo_{\bf k}(P,P)$. The operators $\frac{\partial}{\partial x_i}$ and multiplication
operators generate under composition a formal operator algebra on the polynomial algebra
called the {\bf Weyl algebra in $n$ variables}. A typical element is a formal differential
operator with polynomial coefficients, e.g. $x_1x_2 \frac{\partial^2}{\partial x_1 \partial x_2} -
x_3\frac{\partial}{\partial x_2}$.

If we define the usual commutator Lie bracket of operators $[\psi, \tau]=\psi \circ \tau - \tau \circ \psi$
then the multiplication operators and formal differentiation operators $\frac{\partial}{\partial x_i}$, 
$1 \leq i \leq n$, can be used to define a 
Lie subalgebra of the Weyl algebra called the
{\bf Witt algebra in $n$ variables} which 
will be denoted by $\mathbb{W}_n({\bf k})$.
The typical element of the Witt Lie algebra looks like
$\sum_{i=1}^n P_i \frac{\partial}{\partial x_i}$ where $P_i$ are polynomials, i.e.,
is a first order formal differential operator.
Equivalently the Witt Lie algebra can be thought of as the Lie algebra of polynomial vector fields on
affine space ${\bf k}^n$. The Lie bracket is given by
$$
\left[P\frac{\partial}{\partial x_i}, Q\frac{\partial}{\partial x_j}\right]= \left(P\frac{\partial Q}{\partial x_i}
\right) \frac{\partial}{\partial x_j}
-\left(Q\frac{\partial P}{\partial x_j}\right) \frac{\partial}{\partial x_i}.
$$

Before we see how the Witt Lie algebra plays a role in our computations, we have to look at the dga
$(E_0^{*,*}[\mathfrak{g}],d_0)$ as an algebraic De Rham complex.

Let
$$\Lambda^k_{DR}(\mathfrak{g})=E_0^{*,k}=\Lambda^k(\mathfrak{g}^*) \otimes {\bf k}[\mathfrak{g}^*]
=\Lambda^k(\mathfrak{g}^*) \otimes P,
$$
where we have let $P$ denote the polynomial algebra ${\bf k}[\mathfrak{g}^*]=Poly(\mathfrak{g}^*)$.
The differential $d_0$ then defines a formal algebraic De Rham complex
$$
0 \to \Lambda_{DR}^0 \to \Lambda_{DR}^1 \to \dots \to \Lambda_{DR}^n \to 0, 
$$
where $n$ is the dimension of $\mathfrak{g}$. Recall elements in $Poly(\mathfrak{g}^*)$ can be thought
of as formal polynomial functions on $\mathfrak{g}$ (formal in the case of finite fields) and so
the elements in $\Lambda^m_{DR}$ can be thought of as polynomial $m$-forms on $\mathfrak{g}$.

Let us call the cohomology of this complex 
the De Rham cohomology of $\mathfrak{g}$, and denote it $H^*_{DR}(\mathfrak{g})$.
Note the (De Rham) cohomology is then
$$
H^m_{DR}(\mathfrak{g})=E_1^{*,m}[\mathfrak{g}]=H^m(\mathfrak{g}, Poly(\mathfrak{g}^*)).
$$

Although, this is essentially just rephrasing, the viewpoint will be useful when we consider
$\mathfrak{g}=\mathfrak{sl}_2(\mathbb{F}_p)$, which we do next as an explicit example.
Recall in this case $E_0^*[\mathfrak{sl}_2(\mathbb{F}_p)]=\Lambda^*(h,e,f) \otimes \mathbb{F}_p[H,E,F]$
is the cohomology of the $p$-group $G(\mathfrak{sl}_2(\mathbb{F}_p))$ and $d_0$ is the Bockstein $\beta$.

Here we are using the same generators as were used in the previous discussion of $\mathfrak{sl}_2(\mathbb{C})$.
Since we will be using the weight stratification a lot in the computations, we will rename the generators
now.

The generators $h$ and $H$ are dual to the Cartan algebra element
$v_h=\begin{bmatrix} 1 & 0 \\ 0 & -1 \end{bmatrix}$ and so have weight $0$. We will denote them by
$x_0$ and $y_0$ respectively now.
The generators $e$ and $E$ are dual to the weight 2 element
$v_e=\begin{bmatrix} 0 & 1 \\ 0 & 0 \end{bmatrix}$ and so have weight $-2$.
We will denote them by $x_{-}$ and $y_{-}$ respectively now.
The generators $f$ and $F$ are dual to the weight $-2$ element
$v_f=\begin{bmatrix} 0 & 0 \\ 1 & 0 \end{bmatrix}$ and so have weight $+2$. We will denote
them by $x_{+}$ and $y_{+}$ respectively now.

Thus explicitly $H^*(G(\mathfrak{sl}_2(\mathbb{F}_p)),\mathbb{F}_p)=E_0^{*,*}
=\Lambda^*(x_0,x_{-},x_{+}) \otimes \mathbb{F}_p[y_0,y_{-},y_+]$.
Since $[v_h,v_e]=2v_e, [v_h,v_f]=-2v_f$ and $[v_e,v_f]=v_h$, the
(Bockstein) differential $d_0=\beta$ can be obtained from equation~\ref{eqn: d0formula} in Section~\ref{section: SpectralSequence} 
and is given by

\begin{align*}
\begin{split}
\beta(x_0)&=x_{+}x_{-} \\
\beta(x_{+})&=2x_{0}x_{+} \\
\beta(x_{-})&=-2x_0x_{-} \\
\beta(y_0)&=y_{-}x_{+}-y_{+}x_{-} \\
\beta(y_{+})&=-2y_{0}x_{+}+2y_{+}x_{0} \\
\beta(y_{-})&=2y_{0}x_{-}-2y_{-}x_{0}
\end{split}
\end{align*}
Here anything with a subscript of zero has weight zero, a negative subscript has weight $-2$ and a
positive subscript has weight $+2$. Note the root lattice of $\mathfrak{sl}_2$ is $2\mathbb{Z}$.

Since $\mathfrak{sl}_2(\mathbb{F}_p)$ is 3-dimensional, we can make the following identifications. \\
We can identify $\Lambda^0_{DR}(\mathfrak{sl}_2(\mathbb{F}_p))$ with the polynomial algebra
$P=\mathbb{F}_p[y_0,y_{-},y_{+}]$. \\
We can identify a typical element $f_0x_0 + f_-x_- + f_+x_+$ of $\Lambda^1_{DR}(\mathfrak{sl}_2(\mathbb{F}_p))$
with $\hat{f}=(f_0,f_-,f_+) \in P^3$. We will view $\hat{f}$ is a vector-valued polynomial function. \\
We can identify a typical element $g_0x_-x_+ + g_-x_+x_0+g_+x_0x_-$ of
$\Lambda^2_{DR}(\mathfrak{sl}_2(\mathbb{F}_p))$ with $\hat{g}=(g_0,g_-,g_+) \in P^3$. \\
Finally we can identify a typical element $\theta x_0x_-x_+ \in \Lambda^3_{DR}(\mathfrak{sl}_2(\mathbb{F}_p))$
with $\theta \in P$. \\

We then define the $\mathfrak{sl}_2$ gradient, curl and divergence, $\nabla_{\mathfrak{sl}_2}, Curl_{\mathfrak{sl}_2},
\nabla_{\mathfrak{sl}_2} \cdot$ by saying the following diagram commutes:

\begin{equation}\label{eqn:sl2deRham}
\begin{CD}
\Lambda^3_{DR} @>>> P \\
@A\beta AA @AA\nabla_{\mathfrak{sl}_2}\cdot A \\
\Lambda^2_{DR} @>>> P^3 \\
@A\beta AA @AACurl_{\mathfrak{sl}_2} A \\
\Lambda^1_{DR} @>>> P^3 \\
@A\beta AA @AA\nabla_{\mathfrak{sl}_2} A \\
\Lambda^0_{DR} @>>> P
\end{CD}
\end{equation}
where the horizontal maps are the identifications mentioned in the preceding paragraph.

For $P \in \Lambda^0_{DR}$ we can write $\beta(P)=x_0\beta_0(P) + x_-\beta_-(P) + x_+\beta_+(P)$
where $\beta_i: P \to P$ are easily seen to be derivations. Since the derivations $P \to P$ can in general
be identified with the Witt Lie algebra, let us find such identifications for $\beta_0, \beta_-, \beta_+$.
Since these derivations are uniquely determined by their action on the generators $y_0,y_-,y_+$ of $P$,
simple computations using the Bockstein formulas listed above show that we have the following identifications:
\begin{equation}\label{eqn: DmoduleBockstein}
\begin{split}
\beta_0 = 2y_+\frac{\partial}{\partial y_+}-2y_-\frac{\partial}{\partial y_-} \\
\beta_- = -y_+\frac{\partial}{\partial y_0}+2y_0\frac{\partial}{\partial y_-} \\
\beta_+ = y_-\frac{\partial}{\partial y_0}-2y_0\frac{\partial}{\partial y_+}.
\end{split}
\end{equation}
It is a simple computation to find the operator commutators $[\beta_0,\beta_-]=2\beta_-$,
$[\beta_0,\beta_{+}]=-2\beta_{+}$ and $[\beta_{-},\beta_{+}]=\beta_0$ and hence these operators
give a faithful representation of the Lie algebra $\mathfrak{sl}_2$ in the Witt Lie algebra on 3 generators.

Straightforward computations then express the $\mathfrak{sl}_2$-gradient, curl and divergence in terms
of these linear differential operators.
\begin{align*}
\begin{split}
\nabla_{\mathfrak{sl}_2}(\theta)&=(\beta_0(\theta),\beta_-(\theta),\beta_+(\theta)) ;\\
Curl_{\mathfrak{sl}_2}(\hat{f})&=(\beta_-f_+-\beta_+f_--f_0 , \beta_+f_0-
\beta_0f_+-2f_+ , \beta_0f_--\beta_-f_0-2f_- ) ;\\
\nabla_{\mathfrak{sl}_2} \cdot \hat{f} &= \beta_0 f_0 + \beta_- f_- + \beta_+ f_+ ,
\end{split}
\end{align*}
for all vector polynomials $\hat{f}=(f_0,f_-,f_+) \in P^3$ and scalar polynomials $\theta \in P$.

Since $\beta \circ \beta = 0$, one has the usual calculus identities:
\begin{align*}
\begin{split}
\nabla_{\mathfrak{sl}_2} \cdot Curl_{\mathfrak{sl}_2}(\hat{f})&=0 \text{ and } \\
Curl_{\mathfrak{sl}_2}(\nabla_{\mathfrak{sl}_2}(\theta))&=0
\end{split}
\end{align*}
for all $\hat{f} \in P^3$ and $\theta \in P$. Thus $H^3_{DR}=H^3(\mathfrak{sl}_2,Poly(\mathfrak{sl}_2)^*)$
measures the obstructions to writing a polynomial as a divergence of a vector polynomial,
$H^2_{DR}=H^2(\mathfrak{sl}_2,Poly(\mathfrak{sl}_2)^*)$ measures the obstructions to writing a divergence-free
polynomial vector as a curl of another polynomial vector, \\
$H^1_{DR}=H^1(\mathfrak{sl}_2,Poly(\mathfrak{sl}_2)^*)$
measures the obstructions to writing a curl-free polynomial vector as the gradient of a polynomial, and
$H^0_{DR}=H^0(\mathfrak{sl}_2,Poly(\mathfrak{sl}_2)^*)$ are the polynomials with gradient identically
equal to zero. Thus the DeRham cohomology plays exactly the same role as it does in calculus, except that
the basic notions of gradient, curl and divergence have been altered in accordance with the Lie algebra
$\mathfrak{sl}_2$.

The goal now is to compute $E_1^{*,*}[\mathfrak{sl}_2(\mathbb{F}_p)]=H^*(\mathfrak{sl}_2(\mathbb{F}_p),
Poly(\mathfrak{sl}_2(\mathbb{F}_p))^*)$ as it gives $B_2^*$ in the Bockstein spectral sequence of the
$p$-group $G(\mathfrak{sl}_2(\mathbb{F}_p))$. This is equivalent to computing
$H_{DR}^*(\mathfrak{sl}_2(\mathbb{F}_p))$ as mentioned above. We have seen general comparison theorems
that give the answer in low Hodge degrees relative to the prime $p$ but we seek now to
compute the answer more explicitly in order to understand the ``char 0 to char $p$'' phase transition.

We record some basic facts about $\Lambda_{DR}^*(\mathfrak{sl}_2)$ that will be used in our computations
in the next section. As mentioned before, we weigh the polynomials in $P=\mathbb{F}_p[y_0,y_-,y_+]$
by giving $y_0$ weight 0, $y_-$ weight $-2$ and $y_+$ weight $+2$ and declaring the weight of
a product to be the sum of the weights.

\begin{pro}(Eigenfunctions of $\beta_0$)
\label{pro: eigenfunctions}
If $f \in P$ is a polynomial of weight $w$, then
$\beta_0(f)=wf$. Furthermore the monomial basis of $P$ is a basis of $\beta_0$-eigenfunctions.
\end{pro}
\begin{proof}
A simple computation shows that the identity holds for the generators $y_0, y_-,y_+$.
To finish the proof we note that if $\phi_1$ has weight $a_1$ and $\phi_2$ has weight $a_2$
then as $\beta_0$ is a derivation, we have $\beta_0(\phi_1\phi_2)=\beta_0(\phi_1)\phi_2 + \phi_1\beta_0(\phi_2)
=(a_1+a_2)\phi_1\phi_2$ and the proposition then follows from simple induction and the fact that $P$
has a basis of monomials which are products of generators.
\end{proof}

\begin{pro}(Weight of exterior 0 and 3 lines.) \\
\label{pro: weight03}
$H^3_{DR}=H^3(\mathfrak{sl}_2(\mathbb{F}_p), Poly(\mathfrak{sl}_2(\mathbb{F}_p))^*)$
and
$H^0_{DR}=H^0(\mathfrak{sl}_2(\mathbb{F}_p), Poly(\mathfrak{sl}_2(\mathbb{F}_p))^*)$
consist completely of elements of weights congruent to $0$ mod $p$.

The polynomials $\kappa=y_0^2 + y_+y_-, y_0^p,y_-^p,y_+^p$ lie in $H^0_{DR}$.
\end{pro}
\begin{proof}
If $\theta \in H^0_{DR}-\{0\}$ is of weight $w$ then $\beta(\theta)=0$ so in particular $\beta_0(\theta)=w\theta=0$.
Thus $w=0$ in $\mathbb{F}_p$. Since $w$ is regarded as an integer, this is the same as saying
$w$ must be congruent to $0$ mod $p$.

Let $u = x_0x_-x_+$, and let $\psi u$ with $\psi \in P$ be a typical homogeneous element in $\Lambda^3_{DR}$.
Suppose the weight $w$ of $\psi$ is not zero modulo $p$; then $\beta_0\psi=w\psi \neq 0$
and $\beta_0(\frac{\psi}{w})=\psi$. It is then easy to check that
$\nabla_{\mathfrak{sl}_2} \cdot (\frac{\psi}{w},0,0)=\psi$ and so $\psi u$ is a coboundary and represents
zero in $H^3_{DR}$. Thus $H^3_{DR}$ consists exclusively of elements of weight congruent to $0$ mod $p$ as
$u$ itself has weight $0$.

The final line of the proposition is a trivial computation.
\end{proof}

If $f_0, f_-, f_+ \in P$ have weights $m,n,k$ respectively, we will refer to the
weight of the vector polynomial $\hat{f}=(f_0,f_-,f_+)$ as $(m,n,k) \in \mathbb{Z}^3$.

\begin{pro}(Weight of exterior 1 line.)
\label{pro: weight1}
Let $\theta=f_0x_0 + f_-x_- + f_+x_+ \in \Lambda^1_{DR}$ have weight $w \in \mathbb{Z}$. The
corresponding vector $\hat{f}=(f_0,f_-,f_+)$ has weight $(w,w+2,w-2)$. Then we have
$\beta(\theta)=0 \leftrightarrow Curl_{\mathfrak{sl}_2}(\hat{f})=0$ if and only if the following equations hold:
\begin{align*}
\begin{split}
f_0 &= \beta_-(f_+) - \beta_+(f_-) \\
\beta_+ (f_0) &= wf_+ \\
\beta_-(f_0) &= wf_-
\end{split}
\end{align*}
$H^1_{DR}=H^1(\mathfrak{sl}_2(\mathbb{F}_p), Poly^*(\mathfrak{sl}_2(\mathbb{F}_p)^*))$
consists completely of elements of weight congruent to $0$ mod $p$.
\end{pro}
\begin{proof}
The equations follow from $Curl_{\mathfrak{sl}_2}(\hat{f})=0$ using the formula for $Curl_{\mathfrak{sl}_2}$ listed above together
with simplifications which come from Proposition~\ref{pro: eigenfunctions}.

Thus it remains to show that $H_{DR}^1$ consists only of elements of weight congruent to $0$ mod $p$.
Let $\theta \in \Lambda^1_{DR}$ be an element with $\beta(\theta)=0$ and weight $w$ {\bf not} congruent
to $0$ mod $p$. Since $w \neq 0 \in \mathbb{F}_p$ we have
$$
\nabla_{\mathfrak{sl}_2}\left(\frac{f_0}{w}\right)=\frac{1}{w}(\beta_0(f_0), \beta_-(f_0), \beta_+(f_0))=\frac{1}{w}(wf_0,wf_-,wf_+)=
\hat{f}
$$
where the second to last equality follows from the equations stated in this proposition together with
$\beta_0(f_0)=wf_0$ as $f_0$ has weight $w$. Thus $[\theta]=0 \in H^1_{DR}$. So we can conclude that
the (nonzero) elements of $H^1_{DR}$ all have weight congruent to $0$ mod $p$.
\end{proof}

\begin{cor}(Weight of $H^*_{DR}$)
$H^*_{DR}=E_1^{*,*}[\mathfrak{sl}_2(\mathbb{F}_p)]=H^*(\mathfrak{sl}_2(\mathbb{F}_p), Poly(ad^*))$ consists
entirely of elements of weight $0$ mod $p$. Thus the $E_1$ and higher pages of the spectral sequence
are concentrated completely in weights $0$ modulo $p$.
\end{cor}
\begin{proof}
From proposition~\ref{pro: weight03} and proposition~\ref{pro: weight1}, it follows that any elements
of weight {\bf not} congruent to $0$ mod $p$ in $E_1^{*,*}=H_{DR}^*$
must lie on the exterior 2 line, i.e., on $E_1^{*,2}$.
Since all differentials involving this line either come from or go to elements not on this line, and since
we know the spectral sequence respects weight, we conclude that any elements of weight not congruent to
$0$ mod $p$ on the exterior $2$ line would survive until $E_{\infty}$. Since we know $E_{\infty}$ is the cohomology
of a point, we conclude that all (nonzero) elements on the exterior 2 line must also have weight $0$ mod $p$.
Thus the corollary is proven.
\end{proof}

\begin{lem}
\label{lem: weight0poly}
(Weight 0 polynomials) Let $V_0^N$ denote the vector space of (integral) weight $0$, degree $N$
polynomials inside $P$. The set $\theta_{\ell}=y_0^{N-2\ell}(y_-y_+)^{\ell}, 0 \leq \ell \leq \frac{N}{2}$ is a
basis of $V_0^N$. The set $\Gamma_{\ell}=y_0^{N-2\ell}\kappa^{\ell}, 0 \leq \ell \leq \frac{N}{2}$ is also a basis,
where $\kappa=y_0^2+y_-y_+$.

The algebra $P_0$ of integral weight $0$ polynomials is $\mathbb{F}_p[\kappa, y_0]$.
\end{lem}
\begin{proof}
Any monomial of integral weight zero must involve the same number of $y_-$ and $y_+$ variables
from which it follows the set $\theta_{\ell}, 0 \leq \ell \leq \frac{N}{2}$ is a basis of $V_0^N$.

Expanding $\kappa^{\ell}=(y_0^2+y_-y_+)^{\ell}$, it is easy to see that
$\Gamma_{\ell} = \sum_{j=0}^{\frac{N}{2}} c_{\ell j} \theta_j$ where
$c_{\ell j}=0$ if $j > \ell$ and $c_{\ell \ell}=1$. 
Thus the components of $\Gamma_{\ell}$ with respect to the basis $\{ \theta_{\ell} \}$
form a lower triangular matrix with 1's along the diagonal. Since this matrix is invertible, the set
$\{ \Gamma_j \}$ is also a basis of $V_0^N$. The final line follows from these remarks as $\kappa$
is algebraically independent from $y_0$ since $y_-y_+$ is.
\end{proof}

\begin{lem}(Weight $0$ mod $p$ monomials)
\label{lem: weight0modppoly}
Let $s_-=y_-^p$ and $s_+=y_+^p$. If $f$ is a monomial of weight $0$ mod $p$ then
$f=s_-^ks_+^{\ell} f_0$ for some nonnegative integers $k,\ell$ and integral weight $0$ polynomial $f_0$.
\end{lem}
\begin{proof}
Follows easily as any monomial with weight $0$ mod $p$, must involve a number of $y_-$ variables
and a number of $y_+$ variables that are congruent modulo $p$.
\end{proof}

\begin{pro}(Euler-Poincare Count Formula)
\label{pro: EulerPoincare}
Let ${\bf k}$ be a field, $\mathfrak{g}$ a finite dimensional ${\bf k}$-Lie algebra
of dimension $n > 0$ and $M$ a finite dimensional $\mathfrak{g}$-module.
Then
$$
\sum_{i=0}^{n} (-1)^i \dime(H^i(\mathfrak{g},M)) = 0.
$$
Thus in particular, $\sum_{i=0}^{\dime(\mathfrak{g})} (-1)^i \dime (E_1^{m,i}[\mathfrak{g}]) = 0$
for all nonnegative integers $m$.
\end{pro}
\begin{proof}
$H^*(\mathfrak{g},M)$ can be computed as the cohomology of a Koszul complex of the form
$\Lambda^*(\mathfrak{g}^*) \otimes M$. Letting $C^i = \Lambda^i(\mathfrak{g}^*) \otimes M$,
one has from the Euler-Poincare lemma:
$$
\sum_{i=0}^{n} (-1)^i \dime(C^i) = \sum_{i=0}^{n} (-1)^i \dime(H^i(\mathfrak{g},M)).
$$
Since $\dime(C^i)$ is easily seen to be $\begin{pmatrix} n \\ i \end{pmatrix}  \dime(M)$,
the alternating sum on the left comes out to
$$
\dime(M) \sum_{i=0}^n (-1)^i \begin{pmatrix} n \\ i \end{pmatrix} =\dime(M) (1 - 1)^n = 0.
$$
The last statement follows as $E_1^{m,i}[\mathfrak{g}]=H^i(\mathfrak{g}, Poly^m(ad^*))$.
\end{proof}

\begin{pro}
\label{pro: 0-line computation}
($H_{DR}^0$-computation)
Let $\kappa=y_0^2 + y_-y_+$, $s_0=y_0^p$, $s_-=y_-^p$, and $s_+=y_+^p$. Then
$H_{DR}^0=H^0(\mathfrak{sl}_2(\mathbb{F}_p),Poly(ad^*))$ is generated as an $\mathbb{F}_p$-algebra
by $\kappa, s_0, s_+$,and $s_-$. Furthermore, $s_0,s_-$,and $s_+$ generate a polynomial algebra 
and $\kappa$ is integral over this polynomial algebra satisfying minimal relation
$\kappa^p = s_0^2 + s_-s_+$.
\end{pro}
\begin{proof}
Let $A \subseteq \mathbb{F}_p[y_0,y_-,y_+]$ be the $\mathbb{F}_p$-algebra generated by $\kappa, s_0, s_-, s_+$.
$A  \subseteq H_{DR}^0$ by proposition~\ref{pro: weight03}. Also we have seen that $H_{DR}^0$ is a subset
of the weight $0$ mod $p$ polynomials which lie in the $\mathbb{F}_p$-algebra generated by $y_0$ and $A$.
Let $f \in H_{DR}^0$ have degree $N$. Then we can write
$$
f=\sum_{i=0}^N y_0^i P_i =\sum_{i=0}^{p-1} y_0^i Q_i, 
$$
where $P_i, Q_i \in A$ for all $i$. (We used that $y_0^p=s_0 \in A$ in the second equality.)
Thus $f = Q_0 + \sum_{i=1}^{p-1} y_0^i Q_i$. Since $\beta(f)=\beta(Q_i)=0$, we have
$$
0 = \sum_{i=1}^{p-1} \beta_+(y_0^i)Q_i
  = \sum_{i=1}^{p-1} iy_0^{i-1}y_- Q_i ,
$$
and so
$\sum_{i=1}^{p-1} iy_0^i Q_i=0$, 
where in the last step we scaled by $y_0$ and used that $y_-$ is not a zero divisor in $\mathbb{F}_p[y_0,y_-,y_+]$
to cancel it. If $G=\sum_{i=1}^{p-1} y_0^i Q_i=f-Q_0$, then we have seen that since $\beta(G)=0$, 
$\sum_{i=1}^{p-1} iy_0^i Q_i=0$. Since the numbers $0 < i < p$ are nonzero in $\mathbb{F}_p$, we can use
this relation to solve for $y_0^{p-1} Q_{p-1}$ as a linear combination of $y_0^i Q_i$ for $0 < i < p-1$.
Thus $G=\sum_{i=1}^{p-2} y_0^i T_i$ for new elements $T_i \in A$. We can now repeat the argument
until we see that $G= y_0 T$ for some $T \in A$. However $\beta_+(G)=0=y_-T$ then gives $T=0$ and hence $G=0$.
Thus $f=Q_0 + G=Q_0 \in A$ and hence we have shown that $H_{DR}^0 \subseteq A$. Thus
$A=H_{DR}^0$.

Finally $s_0, s_-, s_+$ are algebraically independent as they are the image of the algebraically independent
elements $y_0, y_-, y_+$ under the injective Frobenius algebra endomorphism of $\mathbb{F}_p[y_0,y_-,y_+]$ given by
$\theta(\alpha)=\alpha^p$. Since $\kappa=y_0^2+y_-y_+$, the Frobenius endomorphism also shows that
$\kappa^p=s_0^2+s_-s_+$ and hence $\kappa$ is integral over $\mathbb{F}_p[s_0,s_-,s_+]$ as it satisfies
the monic polynomial $t^p-(s_0^2+s_-s_+)$. It is clear that this is the minimal polynomial as $\kappa$
has degree $2$ and the $s_i$ have degree $p$, and $p$ is an odd prime.
\end{proof}

\begin{pro}
\label{pro: ext3linecomputation}
($H_{DR}^3$-computation) Let $u \in E_1^{0,3}$ and $\tau \in E_1^{p-1,3}$
be $\beta$-cohomology classes representing $x_0x_-x_+$ and
$y_0^{p-1}x_0x_-x_+$ respectively. Then $u, \tau \neq 0$ and
$E_1^{*,3}=H_{DR}^3$ is generated
as a $E_1^{*,0}=H_{DR}^0$-module by $\tau$ and $u$. Furthermore
we have the relations $us_{+}=us_-=us_0=0$ and the fundamental relation
$\kappa^{\frac{p-1}{2}}u=0$.

This implies $\tau H_{DR}^0 \cap u H_{DR}^0=0$ and so
$H_{DR}^3 = \tau H_{DR}^0 \oplus u H_{DR}^0$.
\end{pro}
\begin{proof}
By proposition~\ref{pro: weight03}, any element of $H_{DR}^3$ has
weight congruent to $0$ mod $p$. Since $x_0x_-x_+$ has weight $0$,
such an element has to be represented by $x_0x_-x_+F$ where $F$
is a polynomial of weight $0$ mod $p$. Thus $F$ is a linear combination
of terms of the form
$\kappa^{\ell}y_0^js_0^as_-^bs_+^c$ for $\ell,j,a,b,c$ nonnegative integers
and $0 \leq j \leq p-1$. We can write $F$ as $F=\sum_{i=0}^{p-1} f_i y_0^i$
where $f_i \in H_{DR}^0$. (Note $\beta$ is a $H_{DR}^0$-module map so
the elements $f_i$ of gradient identically zero function as ``constants''.)
Thus we see $H_{DR}^3$ is generated as a $H_{DR}^0$-module by
$[x_0x_-x_+y_0^i]$ $0 \leq i \leq p-1$, where $[-]$ means ``cohomology class
represented by''. Note $u=[x_0x_-x_+y_0^0]$ and $\tau=[x_0x_-x_+y_0^{p-1}]$.

Let us write $F_1 \sim F_2$ if they differ by $\nabla_{\mathfrak{sl}_2} \cdot \hat{G}$
for some $\hat{G} \in P^3$. In this case $x_0x_-x_+F_1$ and
$x_0x_-x_+F_2$ are $\beta$-cohomologous and so $uF_1=uF_2 \in H_{DR}^3$.

Let us fix a weight $w$ which is $0$ mod $p$
for $F$ and ask about the image of $\nabla_{\mathfrak{sl}_2} \cdot$
in this weight $w$ component. To hit this weight component, the input
$\hat{G}=(g_0,g_-,g_+)$ has to have weight $(w,w-2,w+2)$ and the corresponding
output will be
$$
\beta_0 g_0 + \beta_-(g_-) + \beta_+(g_+)=\beta_-(g_-)+\beta_+(g_+), 
$$
as $\beta_0g_0=0$. Since $g_-$ has weight $w-2$ congruent to $-2$ mod $p$,
it is a combination of terms of the form
$y_-y_0^i\alpha$ or $y_+^{p-1}y_0^i\alpha$ with
$\alpha \in H_{DR}^0$ and $0 \leq i \leq p-1$. Since $g_+$ has weight $w+2$
congruent to $+2$ mod $p$, it is a combination of terms of the form
$y_+y_0^i\alpha$ or $y_-^{p-1}y_0^i\alpha$, $\alpha \in H_{DR}^0, 0
\leq i \leq p-1$. Thus, as $g_-$ and $g_+$ can be chosen independently, 
we see that the weight $w$ component of the image of the divergence
$\nabla_{\mathfrak{sl}_2} \cdot$ is the $H_{DR}^0$-module spanned by the elements
$\beta_+(y_+y_0^i), \beta_+(y_-^{p-1}y_0^i), \beta_-(y_-y_0^i),
\beta_-(y_+^{p-1}y_0^i)$ with $0 \leq i \leq p-1$ (or equivalently
for all nonnegative integers $i$ as $y_0^p=s_0 \in H_{DR}^0$).

After simple computations we find that the image of the divergence is
generated as a $H_{DR}^0$-module by the elements
$i\kappa y_0^{i-1}-(i+2)y_0^{i+1}, s_-iy_0^{i-1}, s_+iy_0^{i-1}$
for $0 \leq i \leq p-1$ (or equivalently all nonnegative $i$).

Setting $\theta_j= jy_0^{j-1}$ we find that $\theta_j=0$ if $j$ is congruent
to $0$ mod $p$ and $\frac{\theta_j}{j}=y_0^{j-1}$ if $j$ is not congruent
to $0$ mod $p$. In this new language the image of the divergence
can be computed as the $H_{DR}^0$-module spanned by
$\kappa \theta_j - \theta_{j+2}, s_+\theta_j, s_-\theta_j$
for all nonnegative integers $j$. Thus
$\kappa \theta_j \sim \theta_{j+2}$ for all nonnegative $j$, and
by a simple induction we have that
all $\theta_{2k} \sim \kappa^k \theta_0 = 0$ and
$\theta_{2k+1} \sim \kappa^k \theta_1 = \kappa^k$.

Thus $u [\theta_{2k}]=0$ and $u[\theta_{2k+1}]=u\kappa^k$ in $H_{DR}^3$.
Since $\theta_{i+1}$ are nonzero multiples of $y_0^{i}$ when $1 \leq i < p-1$
it follows that $u[y_0^{i}]=0$ when $i$ is odd in this range and
$u[y_0^{2k}]=\frac{1}{2k+1} u\kappa^k$ when $i=2k$ in this range. Since we had previously
seen that $H_{DR}^3$ is generated by $u[y_0]^i$ for $0 \leq i \leq p-1$
as $H_{DR}^0$-module, we now see that all these generators except
$u[y_0^0]=u$ and $u[y_0^{p-1}]=\tau$ are redundant.

Hence we have shown that $H_{DR}^3$ is generated by $u$ and $\tau$ as
a $H_{DR}^0$-module. Since $1$ is not in the image of the divergence,
$u \neq 0$. Since $y_0^{p-1}$ is not in the image of the divergence,
$\tau \neq 0$. Since $s_+\theta_1=s_+, s_-\theta_1=s_-$
are in the image
of the divergence, $s_+u=s_-u=0$. Since $y_0^p=\theta_{p+1} \sim \kappa^{\frac{p+1}{2}}\theta_0=0$, we also have $s_0u=0$.
Finally since $0=\theta_p \sim \kappa^{\frac{p-1}{2}}\theta_1=\kappa^{\frac{p-1}{2}}$ we have $\kappa^{\frac{p-1}{2}}u=0$.

Note the relations that we just found imply that $uH_{DR}^0$ is concentrated
in Hodge degrees less than or equal to $p-2$ as $u, u\kappa, \dots,
u\kappa^{\frac{p-1}{2}-1}$ are. On the other hand $\tau H_{DR}^0$ is
concentrated in Hodge degrees greater than or equal to $p-1$, the Hodge
degree of $\tau$. Thus $\tau H_{DR}^0 \cap u H_{DR}^0=0$ and the proof
of the proposition is complete.
\end{proof}

Since it is important to know that $\kappa^i u \neq 0$ for $0 \leq i <
\frac{p-1}{2}$ in order to get lower bounds on essential dimension,
we show that with a dedicated proof for clarity:

\begin{pro}
\label{pro: fundnonvanishing}
(Fundamental nonvanishing)
In $H_{DR}^3$ we have
$$\kappa^{\frac{p-3}{2}} u \neq 0.$$
\end{pro}
\begin{proof}
To ease notation, we set $i = \frac{p-3}{2}$.  Assume to the contrary that there is an element $x_0 x_+ F + x_0 x_- G + x_+ x_- H$ with
\[ \beta(x_0 x_+ F + x_0 x_- G + x_+ x_- H) = \kappa^i u, \]
where $F$, $G$, and $H$ are polynomials of degree $p - 3$.  Then $F$ must have weight $-2$, and hence is a linear combination of terms of the form $y_0^{2j + 1} y_+^{i - j - 1} y_-^{i - j}$ for $0 \leq j < i$.  Similarly, $G$ is a linear combination of terms of the form $y_0^{2j + 1} y_+^{i - j} y_-^{i - j - 1}$, and $H$ is a linear combination of terms of the form $y_0^{2j} y_+^{i - j} y_-^{i - j}$.  A straightforward induction shows that
\begin{eqnarray*}
\beta(x_0 x_+ y_0^k y_+^{\ell - 1} y_-^{\ell}) &=& x_0 x_+ x_- (2\ell y_0^{k + 1} y_+^{\ell - 1} y_-^{\ell - 1} - k y_0^{k - 1} y_+^{\ell} y_-^{\ell}) \\
&=& \beta(x_0 x_- y_0^k y_+^{\ell} y_-^{\ell - 1}) ;\\
\beta(x_+ x_- y_0^k y_+^{\ell} y_-^{\ell}) &=& 0 .\\
\end{eqnarray*}
Thus, without loss of generality, we may assume $G = H = 0$.

Now we compute
\begin{eqnarray*}
&& \beta \left( \sum_{j=0}^{i-1} a_j x_0 x_+ y_0^{2j + 1} y_+^{i - j - 1} y_-^{i - j} \right) \\
&&= \sum_{j=0}^{i-1} a_j x_0 x_- x_+ \Big( (2j - 2i) y_0^{2j + 2} y_+^{i - j - 1} y_-^{i - j - 1} + (2j + 1) y_0^{2j} y_+^{i - j} y_-^{i - j} \Big) \\
&&= u \left( a_0 y_+^i y_-^i - 2a_{i-1} y_0^{2i} + 
\sum_{j=1}^{i-1} \Big( a_j(2j + 1) + a_{j-1} (2j - 2i - 2) \Big) y_0^{2j} y_+^{i - j} y_-^{i - j} \right)
\end{eqnarray*}
Setting this equal to
\[ u\kappa^i = u \sum_{j=0}^i \binom{i}{j} y_0^{2j} y_+^{i - j} y_-^{i - j}, \]
and comparing the $y_0$-degree $0$ terms, we obtain $a_0 = 1$.  We claim that $a_j = \tbinom{i}{j}$ for $j \leq i$, which we show by induction.

Assume that $a_{j-1} = \tbinom{i}{j-1} = \tbinom{i}{j} \tfrac{j}{i - j + 1}$.  Then the coefficient of the $y_0$-degree $2j$ term 
for $1 \leq j \leq i-1$ is
\[ a_j(2j + 1) + a_{j-1} (2j - 2i - 2) = a_j(2j + 1) + \binom{i}{j-1} (2j - 2i - 2). \]
Now, $2j + 1 \not \equiv 0 \pmod{p}$, since $j \leq i = \tfrac{p-3}{2}$, so upon setting this equal to $\tbinom{i}{j}$, we obtain
\begin{eqnarray*}
a_j &=& \dfrac{1}{2j + 1} \left( \binom{i}{j} + \binom{i}{j} \dfrac{j}{i - j + 1} (2i - 2j + 2) \right) \\
&=& \dfrac{1}{2j + 1} \left( \binom{i}{j} + 2j \binom{i}{j} \right) = \binom{i}{j}.
\end{eqnarray*}

Finally, we compare the $y_0$-degree $2i$ terms, and find that $-2 a_{i-1} = \tbinom{i}{i} = 1$.  However since we also 
have $a_j=\binom{i}{j}$, we get  
$-2a_{i-1}=-2\binom{i}{i-1}=-2i = 3 - p \equiv 3 \not \equiv 1 \pmod{p}$, and this contradiction completes the proof.
\end{proof}

We now show that for every odd prime $p$, the behaviour of
$E_1^{*,*}[\mathfrak{sl}_2(\mathbb{F}_p)]$ is the same as
that of $E_1^{*,*}[\mathfrak{sl}_2(\mathbb{C})]$ for Hodge degrees
less than $p-1$.

\begin{pro}(Low Hodge degree computation)
\label{pro: lowHodgecomputation}
For {\bf every} odd prime $p$,
$E_1^{*,*}[\mathfrak{sl}_2(\mathbb{F}_p)] \cong
\Lambda^*(u) \otimes \mathbb{F}_p[\kappa]$
for Hodge degrees less than $p-1$. (The isomorphism breaks down
at Hodge degree $p-1$, for example we have seen $\kappa^{\frac{p-1}{2}}u=0$).
\end{pro}
\begin{proof}
We have seen that the $H_{DR}^0=E_1^{*,0}$-line is generated by
$\kappa, s_0, s_+, s_-$ and so for Hodge degree less than $p$,
the only elements are $\kappa^i, 0 \leq i \leq \frac{p-1}{2}$.
We have seen that the $H_{DR}^3$-line is generated over $H_{DR}^0$
by $\tau$ and $u$ and so for Hodge degree less than $p-1$, the only
elements are $u\kappa^i, 0 \leq i < \frac{p-1}{2}$.

Let $a_i=\dime(E_1^{i,0}), b_i=\dime(E_1^{i,1}), c_i=\dime(E_1^{i,2}),
d_i=\dime(E_1^{i,3})$ then we have for $0 \leq i \leq \frac{p-1}{2}$,
$a_i=d_i=1$ if $i$ even and $a_i=d_i=0$ for $i$ odd.
Thus in this same range, the Euler-Poincare identity gives
$
0=a_i-b_i+c_i-d_i=-b_i+c_i
$
and so $b_i=c_i$ for $0 \leq i \leq \frac{p-1}{2}$.

Note since $E_1^{1,0}=0$, we know that any nonzero element of $E_1^{0,1}$
would survive to $E_{\infty}$ contradicting that $E_{\infty}$ is the
cohomology of a point. Thus $E_1^{0,1}=0$ and so $b_0=c_0=0$.

Now $d_1(u)=d_1(x_0x_-x_+)=[y_0x_-x_+-x_0y_-x_++x_0x_-y_+]=0$ as
$y_0x_-x_+-x_0y_-x_++x_0x_-y_+=\beta(x_0y_0+\frac{1}{2}(x_-y_++x_+y_-))$.
Note as $d_1(x_0y_0+\frac{1}{2}(x_-y_++x_+y_-))=
y_0^2 + y_-y_+$, we see directly
that $d_2(u)= \kappa$.

This establishes then that $d_1(u\kappa^i)=0$ and
$d_2(u\kappa^i)=[\kappa^{i+1}]$ in $E_2$.
Since by proposition~\ref{pro: fundnonvanishing}, 
$u \kappa^i \neq 0$ in $E_1^{2i,3}$ 
for $0 \leq i < \frac{p-1}{2}$, this forces
$d_2$ to be nontrivial in this range, as these elements must support some differential by
$E_{\infty}$ and $d_2$ is the last possible nonzero differential.
Note this means also that $\kappa^{i+1}$ is not hit by any $d_1$-differential
for this range of $i$, as it must survive until $E_2$ to be hit by $d_2(u\kappa^i)$.

Since this accounts for the behaviour of all of the exterior $0$ and $3$ line
in the range of Hodge degree less than $p-1$, we conclude that we must have
$d_1: E_1^{i,2} \to E_1^{i+1,1}$ is an isomorphism in the range
$0 \leq i < p-2$. Thus $c_i=b_{i+1}$ in this range.

Thus $0=b_0=c_0=b_1=c_1=b_2=c_2=\dots=b_{p-2}=c_{p-2}$ which completes
the proof of the proposition.

\end{proof}

This completes the bulk of the computations. In the next section we
lay out a picture of generators of $E_1^{*,*}=H^*(\mathfrak{sl}_2,Poly(ad^*))$
with tables and spectral sequence diagrams to help summarize
the picture. Results about the corresponding $p$-groups are then found.

\section
{Computation of $E_1^{*,*}[\mathfrak{sl}_2(\mathbb{F}_p)]$}
\label{section: modp}

The table below gives a list of important elements
in $E_1^{*,*}[\mathfrak{sl}_2(\mathbb{F}_p)]$. All elements
are $d_0=\beta$-cycles and so represent elements in $E_1$.
We have seen that the elements $\kappa, u, \tau, s_0,s_-,s_+$ generate
the exterior $0$ and $3$ lines as well as all elements of Hodge degree
less than $p-1$. In the table, $s$ is the Hodge degree, $t$ is the exterior degree 
and $w$ is the weight.

\begin{center}
\begin{tabular}{c|c|c|c|l}
& $s$ & $t$ & $w$ \\
\hline
$u$ & $0$ & $3$ & $0$ & $x_0 x_- x_+$ \\
$\kappa$ & $2$ & $0$ & $0$ & $y_0^2 + y_+ y_-$ \\
$\lambda_0$ & $p - 1$ & $1$ & $0$ & $x_0 \kappa^{(p-1)/2} - x_+ y_0 y_+^{-1} \left( \kappa^{(p-1)/2} - y_0^{p-1} \right)$ \\
$\lambda_+$ & $p - 1$ & $1$ & $2p$ & $x_+ y_+^{p-1}$ \\
$\lambda_-$ & $p - 1$ & $1$ & $-2p$ & $x_- y_-^{p-1}$ \\
$\mu_0$ & $p - 1$ & $2$ & $0$ & $x_0y_0^{p-2}\beta(y_0)=x_0 y_0^{p-2} (x_+ y_- - x_-y_+)$ \\
$\mu_+$ & $p - 1$ & $2$ & $2p$ & $x_+y_+^{p-2}\beta(y_+)=-2x_0 x_+ y_+^{p-1}$ \\
$\mu_-$ & $p - 1$ & $2$ & $-2p$ & $x_-y_-^{p-2}\beta(y_-)=2x_0 x_- y_-^{p-1}$ \\
$\tau$ & $p - 1$ & $3$ & $0$ & $x_0 x_- x_+ y_0^{p-1}$ \\
$s_0$ & $p$ & $0$ & $0$ & $y_0^p=d_1(\lambda_0)$ \\
$s_+$ & $p$ & $0$ & $2p$ & $y_+^p=d_1(\lambda_+)$ \\
$s_-$ & $p$ & $0$ & $-2p$ & $y_-^p=d_1(\lambda_-)$ \\
$f_0$ & $p$ & $1$ & $0$ & $y_0^{p-1}\beta(y_0)=y_0^{p-1} (x_+ y_- - x_-y_+)=d_1(\mu_0)$ \\
$f_+$ & $p$ & $1$ & $2p$ & $y_+^{p-1}\beta(y_+)=-2y_+^{p-1} (x_+ y_0 - x_0 y_+)=d_1(\mu_+)$ \\
$f_-$ & $p$ & $1$ & $-2p$ & $y_-^{p-1}\beta(y_-)=2y_-^{p-1} (x_- y_0 - x_0 y_-)=d_1(\mu_-)$ \\
$\gamma$ & $p$ & $1$ & $0$ & $x_+ y_+^{-1} \left( \kappa^{(p+1)/2} - y_0^{p+1} \right) + x_0y_0^p$ \\
$\epsilon$ & $p$ & $2$ & $0$ & $y_0^{p-1} (x_-x_+y_0-x_0 x_+ y_- + x_0 x_- y_+)=d_1(\tau)$ \\
\end{tabular}
\end{center}
Note that multiplication by $y_+^{-1}$ is defined where it appears, since every term it is multiplied by contains a factor of $y_+$.

From the table above, we see the introduction of some new variables
$\mu_i, \lambda_i, f_i$, $\epsilon, \gamma$. Let us explain why these
give the full picture up to Hodge degree $p$.

First note that at the location $E_1^{p-1,1}$, no differentials can be
incoming from the left. Since all elements must eventually die, we see
that $d_1: E_1^{p-1,1} \to E_1^{p,0}$ is an isomorphism.
Since $E_1^{p,0}$ is spanned by $s_0, s_-,s_+$ by
proposition~\ref{pro: 0-line computation}, we see that $E_1^{p-1,1}$
is spanned by elements $\lambda_0, \lambda_-, \lambda_+$
in $E_1$ such that $d_1(\lambda_i)=s_i$. Formulas for
these elements are listed in the table above.

Now note as $\dime(E_1^{p-1,0})=1=\dime(E_1^{p-1,3})$ generated by
$\kappa^{\frac{p-1}{2}}$ and $\tau$ respectively, by the Poincare-Euler
count, we find that 
$$
\dime(E_1^{p-1,2})=\dime(E_1^{p-1,1})=3.
$$
Let $\mu_0, \mu_-, \mu_+$ be a basis of $E_1^{p-1,2}$.
(Candidates are listed in the table above but we still need to show
these candidates are linearly independent in $E_1$. This will 
be proven in Lemma~\ref{lem: mu-f-basis} below.)

Note the explicit element $\gamma$ in the table has
$d_1(\gamma)=\kappa^{\frac{p+1}{2}}$ and so we see explicitly
that $[\kappa^{\frac{p+1}{2}}]=0$ in $E_2$. This now leaves
$\tau$ in a dilemma: it can support no nonzero $d_2$ differential as
the corresponding target location is zero (spanned by
$[\kappa^{\frac{p+1}{2}}]=0$.). Thus $d_1(\tau)=\epsilon \neq 0$ in $E_1$.

From what we have so far, the following short exact sequence is forced:
$$
0 \to E_1^{p-1,2} \to E_1^{p,1} \to E_1^{p+1,0} \to 0.
$$
We know that the left hand group has basis the $\mu_i$ and dimension $3$.
We also know the right hand group has basis $\kappa^{\frac{p+1}{2}}$
and dimension $1$. Thus $E_1^{p,1}$ has dimension exactly four
generated by the element $\gamma$ and the $d_1$-images of the $\mu_i$.
In the table above we have denoted these $d_1$-images as $f_i$ i.e.,
$f_i=d_1(\mu_i)$.

Now note that $E_1^{p,3}=0$ as the exterior $3$ line is generated
by $\tau$ and $u$ over $E_1^{*,0}$ and $u$'s contribution lies entirely
before Hodge degree $p-1$, while nothing nonzero
times $\tau$ lies in that location. Thus we know for certain that
the dimensions of $E_1^{p,0}, E_1^{p,1}$ and $E_1^{p,3}$ are $3, 4$
and $0$ respectively. The Euler-Poincare count now forces
$\dime(E_1^{p,2})=1$. Since we have already argued for the existence
of the nonzero element $\epsilon$ in that location, it must be a
basis for $E_1^{p,2}$.

We now prove that the $\mu_i$ and hence the $f_i$ listed in the table 
are indeed the basis for $E_1^{p-1,2}$ and $E_1^{p,1}$ respectively:

\begin{lem}
\label{lem: mu-f-basis} 
The elements $\{ \mu_0, \mu_+, \mu_- \}$ in the table above are a basis for 
$E_1^{p-1,2}$ and the elements $\{ f_0, f_+, f_- \}$ are a basis 
for $E_1^{p,1}$.
\end{lem}
\begin{proof}
Since the elements have distinct integral weight, to show they 
are linearly independent, it is sufficient to show they are nonzero. 
Since we also have argued that $\dime E_1^{p-1,2} = \dime E_1^{p,1}=3$ 
this is enough to show they are a basis and complete the lemma.

Furthermore since $d_1(\mu_i)=f_i$ for $i=0,-,+$, we can show either 
$\mu_i$ or $f_i$ is nonzero, and it will follow that the other in the pair 
is also nonzero.

Recall the root lattice of $\mathfrak{sl}_2$ is $2\mathbb{Z}$, it is 
easy to check that $E_0^{p-1,1}[2p]$, the weight $2p$ component of 
$E_0^{p-1,1}$, is one dimensional, spanned by $\lambda_+=x_+y_+^{p-1}$. 
A direct check shows that $d_0(x_+y_+^{p-1})=0$ and so we find that 
$d_0: E_0^{p-1,1}[2p] \to E_0^{p-1,2}[2p]$ is identically zero. Since 
$\mu_+ \in E_0^{p-1,2}[2p]$ and the spectral sequence preserves weight, 
this shows $\mu_+$ is not a $d_0$-coboundary and so $\mu_+ \neq 0 \in 
E_1^{p-1,2}$ as desired. A similar proof works to show 
$\mu_- \neq 0 \in E_1^{p-1,2}$. 

It remains to show $\mu_0 \neq 0 \in E_1^{p-1,2}$. We will do this 
by showing $d_1(\mu_0)=f_0 \neq 0 \in E_1^{p,1}$. Since $f_0$ is represented 
by $y_0^{p-1}\beta(y_0)$ and has integral weight $0$, it is enough to 
show that $y_0^{p-1}\beta(y_0)$ is not in the image of 
$d_0: E_1^{p,0}[0] \to E_1^{p,1}[0]$. The typical element $\alpha$ of 
$E_1^{p,0}[0]$ is of the form $\alpha=\sum_{\ell=0}^{\frac{p-1}{2}} c_{\ell} \kappa^{\ell} 
y_0^{p-2\ell}$. Computing we find:
$$
d_0(\alpha)=
(\sum_{\ell=0}^{\frac{p-1}{2}} c_{\ell} \kappa_{\ell} (p-2\ell) y_0^{p-1-2\ell}) \beta(y_0).
$$
Now $\beta(y_0)=y_-x_+-y_+x_-$ has the easily verified property that for any 
nonzero polynomial $f \in \mathbb{F}_p[\kappa, y_0]$, 
$\beta(y_0) f \neq 0 \in E_0^{*,*}$.
Thus if $d_0(\alpha)=y_0^{p-1}\beta(y_0)$ we would have 
$$
(\sum_{\ell=0}^{\frac{p-1}{2}} c_{\ell} (p-2\ell) \kappa_{\ell} y_0^{p-1-2\ell} - y_0^{p-1}) = 0.
$$
Since $y_0$ and $\kappa$ are algebraically independent, equating coefficients 
of $y_0^{p-1}$ we get $c_0 p - 1=0$, i.e., $-1 \equiv 0$ mod $p$ which is 
ridiculous. Thus $f_0=y_0^{p-1}\beta(y_0)$ is not in the image 
of $d_0$ and hence $f_0 \neq 0 \in E_1^{p,1}$ and we are done.

\end{proof}

Thus after lots of work we have  found:

\begin{thm}($H^*(\mathfrak{sl}_2,Poly(ad^*))$ description).
The 17 elements listed in the table above generate all parts of
$E_1^{*,*}[\mathfrak{sl}_2(\mathbb{F}_p)]=
H^*(\mathfrak{sl}_2,Poly(ad^*))$ in Hodge degrees $p$ or less,
and also on the exterior $0$ and $3$ lines. They are a minimal set of
generators that do so. At most a {\bf finite} number
of additional generators lying on exterior $1$, $2$ lines of Hodge
degree $> p$ are required to generate the whole algebra.

Furthermore,
$$
E_2^{*,*}[\mathfrak{sl}_2(\mathbb{F}_p)] =
\Lambda^*(u) \otimes \mathbb{F}_p[\kappa]
/<u\kappa^{\frac{p-1}{2}}, \kappa^{\frac{p+1}{2}}>.
$$
\end{thm}
\begin{proof}
Most of the theorem has been proved already; we need only make the following
additional observations.
As $E_0$ is finitely generated as a module
over the Noetherian polynomial algebra
$R=\mathbb{F}_p[s_0,s_-,s_+]$ and $d_0$ is a $R$-module homomorphism,
we conclude $\ker(d_0)$ and hence $E_1$ are finitely generated
$R$-modules. It follows that $E_1$ is finitely generated as an algebra
and that at most a finite number of additional generators might be required.

Regardless, since we know $[u\kappa^{\frac{p-1}{2}}]
=[\kappa^{\frac{p+1}{2}}]=[s_0]=[s_-]=[s_+]=0$
in $E_2$, we know that no $d_2$ differential is supported on the 
exterior 3 line, in Hodge
degree $p-2$ or more as $E_2^{s,0}=0$ for $s \geq p$. 
From this it is an easy argument to conclude
that the only terms that survive to the $E_2$ page are as described
in the theorem. ($E_{\infty}$ is known to be the cohomology
of a point and $d_2$ is the last possible nonzero differential.)

\end{proof}

As a byproduct of all this analysis we have also shown:
\begin{cor}
\label{cor: exact seq}
Let $\mathfrak{sl}_2$ denote $\mathfrak{sl}_2(\mathbb{F}_p)$ for 
$p$ an odd prime and let $S^i$ denote the module of homogeneous, degree 
$i$ polynomials equipped with the dual adjoint action.
Then for $i \geq p-2$, the following sequence is exact:
$$
0 \to H^3(\mathfrak{sl}_2,S^i) \to H^2(\mathfrak{sl}_2, S^{i+1}) 
\to H^1(\mathfrak{sl}_2,S^{i+2}) \to H^0(\mathfrak{sl}_2,S^{i+3}) \to 0.
$$
\end{cor}
\begin{proof}
The above sequence is just 
$$
0 \to E_1^{i,3} \xrightarrow{d_1} E_1^{i+1,2} \xrightarrow{d_1} 
E_1^{i+2,1} \xrightarrow{d_1} E_1^{i+3,0} \to 0
$$
in the spectral sequence $E_r^{*,*}[\mathfrak{sl}_2(\mathbb{F}_p)]$.

It is exact as $E_{\infty}^{s,t}=E_2^{s,t}=0$ for $s+t > p$.
\end{proof}

Note that $d_1(H_{DR}^3)) = d_1(H_{DR}^0 \tau + H_{DR}^0 u)
=H_{DR}^0 \epsilon$ so any missing ``ghost generator'' on the exterior $2$
line has to inject under $d_1$ to a missing ``ghost generator'' on the
exterior $1$ line as the image of $d_1$ from the exterior $3$ line is
accounted for.

The following diagram illustrates $E_1^{*,*}[\mathfrak{sl}_2(\mathbb{F}_p)]$ 
through Hodge degree $p + 1$ for a typical odd prime $p$.  
The given elements are cohomology generators in Hodge degree $s$ and exterior degree $t$.  
The dashed arrows represent the action of the differential $d_2$, and the solid arrows 
represent the action of the differential $d_1$.

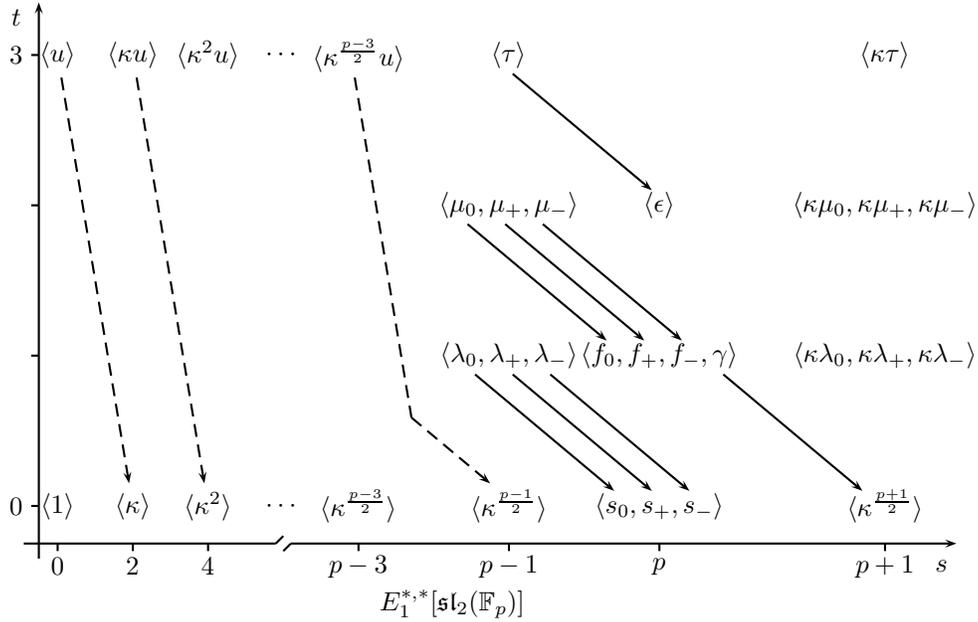
\begin{center}
\begin{pspicture}(0,-1)(12,7.2)
  \rput{0}(5.5,-0.8){$E_1^{*,*}[\mathfrak{sl}_2(\mathbb{F}_p)]$}
  \psline{->}(0,-0.2)(0,7.2)
  \rput{0}(-0.3,7){$t$}
  \psline(-0.1,0.5)(0,0.5)
  \rput{0}(-0.3,0.5){$0$}
  \psline(-0.1,2.5)(0,2.5)
  \psline(-0.1,4.5)(0,4.5)
  \psline(-0.1,6.5)(0,6.5)
  \rput{0}(-0.3,6.5){$3$}
  \psline(-0.2,0)(3.15,0)
  \psline(3.15,0)(3.25,0.1)
  \psline(3.25,-0.1)(3.35,0)
  \psline{->}(3.35,0)(12.2,0)
  \rput{0}(12,-0.3){$s$}
  \psline(0.25,-0.1)(0.25,0)
  \rput{0}(0.25,-0.3){$0$}
  \psline(1.25,-0.1)(1.25,0)
  \rput{0}(1.25,-0.3){$2$}
  \psline(2.25,-0.1)(2.25,0)
  \rput{0}(2.25,-0.3){$4$}
  \psline(4.25,-0.1)(4.25,0)
  \rput{0}(4.25,-0.3){$p-3$}
  \psline(6.25,-0.1)(6.25,0)
  \rput{0}(6.25,-0.3){$p-1$}
  \psline(8.25,-0.1)(8.25,0)
  \rput{0}(8.25,-0.3){$p$}
  \psline(11.25,-0.1)(11.25,0)
  \rput{0}(11.25,-0.3){$p+1$}
  \rput{0}(0.25,0.5){$\langle 1 \rangle$}
  \rput{0}(0.25,6.5){$\langle u \rangle$}
  \rput{0}(1.25,0.5){$\langle \kappa \rangle$}
  \rput{0}(1.25,6.5){$\langle \kappa u \rangle$}
  \rput{0}(2.25,0.5){$\langle \kappa^2 \rangle$}
  \rput{0}(2.25,6.5){$\langle \kappa^2 u \rangle$}
  \psline[linestyle=dashed]{->}(0.3,6.2)(1.2,0.8)
  \psline[linestyle=dashed]{->}(1.3,6.2)(2.2,0.8)
  \rput{0}(3.25,0.5){$\cdots$}
  \rput{0}(3.25,6.5){$\cdots$}
  \rput{0}(4.25,0.5){$\langle \kappa^{\frac{p-3}{2}} \rangle$}
  \rput{0}(4.25,6.5){$\langle \kappa^{\frac{p-3}{2}} u \rangle$}
  \psline[linestyle=dashed](4.2,6.2)(4.954,1.676)
  \psline[linestyle=dashed]{->}(4.954,1.676)(6,0.8)
  \rput{0}(6.25,0.5){$\langle \kappa^{\frac{p-1}{2}} \rangle$}
  \rput{0}(6.25,2.5){$\langle \lambda_0, \lambda_+, \lambda_- \rangle$}
  \rput{0}(6.25,4.5){$\langle \mu_0, \mu_+, \mu_- \rangle$}
  \rput{0}(6.25,6.5){$\langle \tau \rangle$}
  \psline{->}(5.8,2.25)(7.65,0.7)
  \psline{->}(6.3,2.25)(8.15,0.7)
  \psline{->}(6.8,2.25)(8.65,0.7)
  \psline{->}(5.7,4.25)(7.55,2.7)
  \psline{->}(6.2,4.25)(8.05,2.7)
  \psline{->}(6.7,4.25)(8.55,2.7)
  \psline{->}(6.3,6.25)(8.15,4.7)
  \rput{0}(8.25,0.5){$\langle s_0, s_+, s_- \rangle$}
  \rput{0}(8.25,2.5){$\langle f_0, f_+, f_-, \gamma \rangle$}
  \rput{0}(8.25,4.5){$\langle \epsilon \rangle$}
  \psline{->}(9.1,2.25)(10.95,0.7)
  \rput{0}(11.25,0.5){$\langle \kappa^{\frac{p+1}{2}} \rangle$}
  \rput{0}(11.25,2.5){$\langle \kappa \lambda_0, \kappa \lambda_+, \kappa \lambda_- \rangle$}
  \rput{0}(11.25,4.5){$\langle \kappa \mu_0, \kappa \mu_+, \kappa \mu_- \rangle$}
  \rput{0}(11.25,6.5){$\langle \kappa \tau \rangle$}
\end{pspicture}
\end{center}

The next diagram illustrates $E_1^{*,*}[\mathfrak{sl}_2(\mathbb{F}_5)]$ through Hodge degree $20$.  
The numbers indicate cohomology dimensions in Hodge degree $s$ and exterior degree $t$.  
The dashed arrows represent the differential $d_2$, which is an isomorphism in each case.  
The solid arrows represent the differential $d_1$, which gives a short exact sequence along each diagonal line.

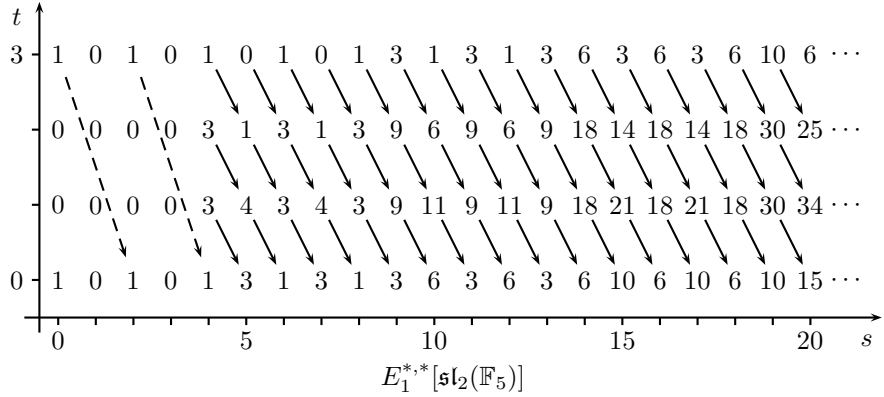
\begin{center}
\begin{pspicture}(0,-1)(11,4.2)
  \rput{0}(5.5,-0.8){$E_1^{*,*}[\mathfrak{sl}_2(\mathbb{F}_5)]$}
  \psline{->}(0,-0.2)(0,4.2)
  \rput{0}(-0.3,4){$t$}
  \psline(-0.1,0.5)(0,0.5)
  \rput{0}(-0.3,0.5){$0$}
  \psline(-0.1,1.5)(0,1.5)
  \psline(-0.1,2.5)(0,2.5)
  \psline(-0.1,3.5)(0,3.5)
  \rput{0}(-0.3,3.5){$3$}
  \psline{->}(-0.2,0)(11.2,0)
  \rput{0}(11,-0.3){$s$}
  \psline(0.25,-0.1)(0.25,0)
  \rput{0}(0.25,-0.3){$0$}
  \psline(0.75,-0.1)(0.75,0)
  \psline(1.25,-0.1)(1.25,0)
  \psline(1.75,-0.1)(1.75,0)
  \psline(2.25,-0.1)(2.25,0)
  \psline(2.75,-0.1)(2.75,0)
  \rput{0}(2.75,-0.3){$5$}
  \psline(3.25,-0.1)(3.25,0)
  \psline(3.75,-0.1)(3.75,0)
  \psline(4.25,-0.1)(4.25,0)
  \psline(4.75,-0.1)(4.75,0)
  \psline(5.25,-0.1)(5.25,0)
  \rput{0}(5.25,-0.3){$10$}
  \psline(5.75,-0.1)(5.75,0)
  \psline(6.25,-0.1)(6.25,0)
  \psline(6.75,-0.1)(6.75,0)
  \psline(7.25,-0.1)(7.25,0)
  \psline(7.75,-0.1)(7.75,0)
  \rput{0}(7.75,-0.3){$15$}
  \psline(8.25,-0.1)(8.25,0)
  \psline(8.75,-0.1)(8.75,0)
  \psline(9.25,-0.1)(9.25,0)
  \psline(9.75,-0.1)(9.75,0)
  \psline(10.25,-0.1)(10.25,0)
  \rput{0}(10.25,-0.3){$20$}
  \rput{0}(0.25,0.5){$1$}
  \rput{0}(0.25,1.5){$0$}
  \rput{0}(0.25,2.5){$0$}
  \rput{0}(0.25,3.5){$1$}
  \rput{0}(0.75,0.5){$0$}
  \rput{0}(0.75,1.5){$0$}
  \rput{0}(0.75,2.5){$0$}
  \rput{0}(0.75,3.5){$0$}
  \rput{0}(1.25,0.5){$1$}
  \rput{0}(1.25,1.5){$0$}
  \rput{0}(1.25,2.5){$0$}
  \rput{0}(1.25,3.5){$1$}
  \psline[linestyle=dashed]{->}(0.35,3.2)(1.15,0.8)
  \psline[linestyle=dashed]{->}(1.35,3.2)(2.15,0.8)
  \rput{0}(1.75,0.5){$0$}
  \rput{0}(1.75,1.5){$0$}
  \rput{0}(1.75,2.5){$0$}
  \rput{0}(1.75,3.5){$0$}
  \rput{0}(2.25,0.5){$1$}
  \rput{0}(2.25,1.5){$3$}
  \rput{0}(2.25,2.5){$3$}
  \rput{0}(2.25,3.5){$1$}
  \rput{0}(2.75,0.5){$3$}
  \rput{0}(2.75,1.5){$4$}
  \rput{0}(2.75,2.5){$1$}
  \rput{0}(2.75,3.5){$0$}
  \psline{->}(2.35,1.3)(2.65,0.7)
  \psline{->}(2.35,2.3)(2.65,1.7)
  \psline{->}(2.35,3.3)(2.65,2.7)
  \psline{->}(2.85,1.3)(3.15,0.7)
  \psline{->}(2.85,2.3)(3.15,1.7)
  \psline{->}(2.85,3.3)(3.15,2.7)
  \rput{0}(3.25,0.5){$1$}
  \rput{0}(3.25,1.5){$3$}
  \rput{0}(3.25,2.5){$3$}
  \rput{0}(3.25,3.5){$1$}
  \rput{0}(3.75,0.5){$3$}
  \rput{0}(3.75,1.5){$4$}
  \rput{0}(3.75,2.5){$1$}
  \rput{0}(3.75,3.5){$0$}
  \psline{->}(3.35,1.3)(3.65,0.7)
  \psline{->}(3.35,2.3)(3.65,1.7)
  \psline{->}(3.35,3.3)(3.65,2.7)
  \psline{->}(3.85,1.3)(4.15,0.7)
  \psline{->}(3.85,2.3)(4.15,1.7)
  \psline{->}(3.85,3.3)(4.15,2.7)
  \rput{0}(4.25,0.5){$1$}
  \rput{0}(4.25,1.5){$3$}
  \rput{0}(4.25,2.5){$3$}
  \rput{0}(4.25,3.5){$1$}
  \rput{0}(4.75,0.5){$3$}
  \rput{0}(4.75,1.5){$9$}
  \rput{0}(4.75,2.5){$9$}
  \rput{0}(4.75,3.5){$3$}
  \psline{->}(4.35,1.3)(4.65,0.7)
  \psline{->}(4.35,2.3)(4.65,1.7)
  \psline{->}(4.35,3.3)(4.65,2.7)
  \psline{->}(4.85,1.3)(5.15,0.7)
  \psline{->}(4.85,2.3)(5.15,1.7)
  \psline{->}(4.85,3.3)(5.15,2.7)
  \rput{0}(5.25,0.5){$6$}
  \rput{0}(5.25,1.5){$11$}
  \rput{0}(5.25,2.5){$6$}
  \rput{0}(5.25,3.5){$1$}
  \rput{0}(5.75,0.5){$3$}
  \rput{0}(5.75,1.5){$9$}
  \rput{0}(5.75,2.5){$9$}
  \rput{0}(5.75,3.5){$3$}
  \psline{->}(5.35,1.3)(5.65,0.7)
  \psline{->}(5.35,2.3)(5.65,1.7)
  \psline{->}(5.35,3.3)(5.65,2.7)
  \psline{->}(5.85,1.3)(6.15,0.7)
  \psline{->}(5.85,2.3)(6.15,1.7)
  \psline{->}(5.85,3.3)(6.15,2.7)
  \rput{0}(6.25,0.5){$6$}
  \rput{0}(6.25,1.5){$11$}
  \rput{0}(6.25,2.5){$6$}
  \rput{0}(6.25,3.5){$1$}
  \rput{0}(6.75,0.5){$3$}
  \rput{0}(6.75,1.5){$9$}
  \rput{0}(6.75,2.5){$9$}
  \rput{0}(6.75,3.5){$3$}
  \psline{->}(6.35,1.3)(6.65,0.7)
  \psline{->}(6.35,2.3)(6.65,1.7)
  \psline{->}(6.35,3.3)(6.65,2.7)
  \psline{->}(6.85,1.3)(7.15,0.7)
  \psline{->}(6.85,2.3)(7.15,1.7)
  \psline{->}(6.85,3.3)(7.15,2.7)
  \rput{0}(7.25,0.5){$6$}
  \rput{0}(7.25,1.5){$18$}
  \rput{0}(7.25,2.5){$18$}
  \rput{0}(7.25,3.5){$6$}
  \rput{0}(7.75,0.5){$10$}
  \rput{0}(7.75,1.5){$21$}
  \rput{0}(7.75,2.5){$14$}
  \rput{0}(7.75,3.5){$3$}
  \psline{->}(7.35,1.3)(7.65,0.7)
  \psline{->}(7.35,2.3)(7.65,1.7)
  \psline{->}(7.35,3.3)(7.65,2.7)
  \psline{->}(7.85,1.3)(8.15,0.7)
  \psline{->}(7.85,2.3)(8.15,1.7)
  \psline{->}(7.85,3.3)(8.15,2.7)
  \rput{0}(8.25,0.5){$6$}
  \rput{0}(8.25,1.5){$18$}
  \rput{0}(8.25,2.5){$18$}
  \rput{0}(8.25,3.5){$6$}
  \rput{0}(8.75,0.5){$10$}
  \rput{0}(8.75,1.5){$21$}
  \rput{0}(8.75,2.5){$14$}
  \rput{0}(8.75,3.5){$3$}
  \psline{->}(8.35,1.3)(8.65,0.7)
  \psline{->}(8.35,2.3)(8.65,1.7)
  \psline{->}(8.35,3.3)(8.65,2.7)
  \psline{->}(8.85,1.3)(9.15,0.7)
  \psline{->}(8.85,2.3)(9.15,1.7)
  \psline{->}(8.85,3.3)(9.15,2.7)
  \rput{0}(9.25,0.5){$6$}
  \rput{0}(9.25,1.5){$18$}
  \rput{0}(9.25,2.5){$18$}
  \rput{0}(9.25,3.5){$6$}
  \rput{0}(9.75,0.5){$10$}
  \rput{0}(9.75,1.5){$30$}
  \rput{0}(9.75,2.5){$30$}
  \rput{0}(9.75,3.5){$10$}
  \psline{->}(9.35,1.3)(9.65,0.7)
  \psline{->}(9.35,2.3)(9.65,1.7)
  \psline{->}(9.35,3.3)(9.65,2.7)
  \psline{->}(9.85,1.3)(10.15,0.7)
  \psline{->}(9.85,2.3)(10.15,1.7)
  \psline{->}(9.85,3.3)(10.15,2.7)
  \rput{0}(10.25,0.5){$15$}
  \rput{0}(10.25,1.5){$34$}
  \rput{0}(10.25,2.5){$25$}
  \rput{0}(10.25,3.5){$6$}
  \rput{0}(10.75,0.5){$\cdots$}
  \rput{0}(10.75,1.5){$\cdots$}
  \rput{0}(10.75,2.5){$\cdots$}
  \rput{0}(10.75,3.5){$\cdots$}
\end{pspicture}
\end{center}

Computer calculations indicate that the 17 generators we have listed
are sufficient to generate the whole algebra through high Hodge degree
so we conjecture:

\begin{conjecture}
\label{conj: generators}
The 17 elements $u, \kappa, \tau, \epsilon, \gamma, \lambda_i,
\mu_i, f_i, s_i$ generate the algebra $H^*(\mathfrak{sl}_2(\mathbb{F}_p), Poly(ad^*))$.
\end{conjecture}

\section{Bounds on the exceptional torsion in congruence subgroups}

The $p$-group $G(\mathfrak{sl}_2(\mathbb{F}_p))$ is the kernel
of the mod $p$ reduction $SL_2(\mathbb{Z}/p^3\mathbb{Z}) \to
SL_2(\mathbb{F}_p)$. Since $G(\mathfrak{g}) \times \mathbb{Z}/p\mathbb{Z}$
(where $\mathfrak{g}$ is the nonabelian Lie algebra of dimension $2$)
is an index $p$ subgroup of $G(\mathfrak{sl}_2(\mathbb{F}_p))$, it follows
that $e(G(\mathfrak{sl}_2(\mathbb{F}_p))) \leq pe(G(\mathfrak{g}))=p^3$.
In $\cite{Pk}$, it was shown that $e_{\infty}(G(\mathfrak{sl}_2(\mathbb{F}_p)))
=p^2$.

By the results in \cite{BrP}, it is known that
$B_2^*(G(\mathfrak{sl}_2(\mathbb{F}_p)))=E_1^{*,*}[
\mathfrak{sl}_2(\mathbb{F}_p)]$ which was computed (partially) in the last
section.
Using comparisons to the cyclic group of order $p^2$
defined by $E=F=0$, it was shown in $\cite{BrP}$ that
$\beta_2(u)=\beta_2(\kappa)=0$ and $\beta_3(u)$ is a nonzero
multiple of $\kappa$. Thus in particular, as has been the case in all other
computations, we have $(E_r^{*,*},d_r)=(B_{r+1}^*,\beta_{r+1})$
for $\mathfrak{sl}_2(\mathbb{F}_p)$ at least in the range between
Hodge degree $0$ and $p-2$.

Since we have shown that $u\kappa^{\frac{p-3}{2}} \neq 0$ in
$E_1=B_2^*$, it follows that we have argued for a nonzero differential
$\beta_3(u\kappa^{\frac{p-3}{2}})=c\kappa^{\frac{p-1}{2}}$ with
$c \neq 0$. Since it is known from \cite{Pk} that $B_3^*$ is finite
dimensional, it follows that $\kappa^{\frac{p-1}{2}}$
lifts to an exceptional integral cohomology class of order $p^3$ in
$H^{2p-2}(G(\mathfrak{sl}_2(\mathbb{F}_p)), \mathbb{Z})$.

Thus $ED(G(\mathfrak{sl}_2(\mathbb{F}_p))) \geq 2p-2$.

If Conjecture~\ref{conj: generators} holds {\bf and} furthermore
$E_r^{*,*}=B_{r+1}^*$ in all Hodge degrees, then the inequality above would
be an equality. (We feel this is almost certainly true but are unable to
prove it at this time.)

\section{Comments on Modular Forms and Orthogonal Steenrod Structures}
\subsection{Modular Forms}
Following \cite{FTY}, $H^*(SL_2(\mathbb{Z}), Poly_{\mathbb{C}}(V))$
where $V$ is the canonical complex representation of $SL_2(\mathbb{Z})$
can be identified directly via an ``Eichler-Shimura'' correspondence (see \cite{S}) 
with modular forms of certain flavors.

It is well known that the Hodge decomposition
$Poly_{\mathbb{C}}(V)=\oplus_{i=0}^{\infty} Sym^i(V)$ yields all the finite
dimensional complex irreducible representations of $SL_2(\mathbb{C})$ and
equivalently $\mathfrak{sl}_2(\mathbb{C})$. In this picture
$V=Sym^1(V)$ and $ad=Sym^2(V)$. Whitehead's lemma shows that
$H^0(SL_2(\mathbb{C}), Poly(V))=
H^0(\mathfrak{sl}_2(\mathbb{C}), Poly(V))$ is the cohomology of a point
and hence uninteresting. Thus it is crucial in the Eichler-Shimura correspondence
that one has passed to the arithmetic subgroup $SL_2(\mathbb{Z})$.

The closest analog in our work would be
$E_1^{*,*}[\mathfrak{sl}_2(\mathbb{Z})]=  
H^*(\mathfrak{sl}_2(\mathbb{Z}), Poly(ad^*))$,
which we have shown has $p$-torsion for all primes $p$.
Here there is no longer a nice exponential correspondence between
$\mathfrak{sl}_2(\mathbb{Z})$ and $SL_2(\mathbb{Z})$ and the adjoint
representation has been used, rather than the canonical representation, so there
is no direct relation to modular forms, though it does seem
reasonable to suspect an indirect one.

However, this arithmetic object shares properties with
$H^*(SL_2(\mathbb{Z}), Poly(W))$, where $W$ is the canonical integral 
representation of $SL_2(\mathbb{Z})$, which has also been shown to have
$p$-torsion for all primes $p$ in unpublished
work of F. Cohen, M. Salvetti and F. Callegaro.

\subsection{Orthogonal Steenrod Structures.}
The spectral sequence $E_r^{*,*}$ used in this paper was motivated
by the Bockstein spectral sequence of $p$-groups associated to
$p$-adic Lie algebras.

For a Lie algebra $\mathfrak{g}$, the spectral sequence
$E_r^{*,*}[\mathfrak{g}]$ arises as that of a double complex with
differentials $d_0$ and $d_1$.

Over $\mathbb{F}_p$, $E_0^{*,*}=\Lambda^*(V) \otimes Poly(V)$
and $d_0$ is the Bockstein arising from the $p$-group $G(\mathfrak{sl}_2(\mathbb{F}_p))$ while $d_1$ is the Bockstein arising from the elementary abelian
$p$-group. The higher Steenrod $P$-power operations are axiomatically
determined and agree for the two $p$-groups.

Thus while the two $p$-groups have the same $\mathbb{F}_p$-cohomology
and Steenrod $p$-power operations, their Bocksteins act ``orthogonally''
in the following sense:

\begin{defn}(Orthogonal Steenrod Structures)
Let $\mathfrak{U}$ be the category of unstable modules over the
mod $p$ Steenrod algebra $A_p$. A bigraded module $F^{m,n}$
is an {\bf orthogonal Steenrod bimodule} if \\
(1) There are two actions of $A_p$ on $F^{m,n}$
such that the $P$-power operations agree under the two actions but
the Bockstein operators act differently say via $\beta_0$ and $\beta_1$
respectively. \\
(2) $\beta_0$ raises $m$-degree by one and preserves $n$-degree while
$\beta_1$ raises $n$-degree by one and preserves $m$-degree. \\
(3) $\beta_0 \circ \beta_1 = - \beta_1 \circ \beta_0.$ \\
(4) The total complex determined by the pair of commuting
differentials $\beta_0$ and $\beta_1$ is acyclic.
\end{defn}
(Note the bigrading above does not coincide with the one we have been using in the paper, a 
regrading was made for convenience and is explained in the appendix.)

In our case the resulting spectral sequence derived from these anti-commuting
Bocksteins proved very useful. Considering that every indecomposable
injective $\mathfrak{U}$-module is of the form $L \otimes J[n]$
where $L$ is a summand of the cohomology of an elementary abelian $p$-group
and $J[n]$ is a Brown-Gitler module (see \cite{LS}), it seems that considering
which equivalence classes of extensions in $Ext_{\mathfrak{U}}(k,k)$
can be equipped with such ``orthogonal structures'' might be interesting.

\section{Acknowledgments}
We would like to thank Fred Cohen and John Harper
for many interesting and useful discussions regarding the topics in this
paper. We would also like to thank the anonymous referee who reviewed the first draft for many useful comments 
and suggestions. 

\appendix
\section{Construction of the Algebraic Spectral Sequence $E[\mathfrak{g}]$}

In this section, we will derive a spectral sequence for Lie algebra
cohomology which is based on the Bockstein spectral sequence of the
groups mentioned in the introduction but is defined for Lie algebras
over an arbitrary PID $\mathbf{k}$. We give what are probably too many
details so the reader is referred to the short summary of the final results
given in an earlier section.

Let $\mathbf{k}$ be an arbitrary PID.
Let $\mathfrak{L}$ be a $\mathbf{k}$-Lie algebra
and $V$ be a finite dimensional $\mathfrak{L}$-module.
Set
$V^{(s)}=V\otimes\dots\otimes V$ (s times) for $s \geq 1$
and $V^{(0)}=\mathbf{k}$.
$V^{(s)}$ is given the usual $\mathfrak{L}$-module structure, i.e.\
\begin{align*}
\begin{split}
u \cdot (v_1\otimes\dots\otimes v_s)=&(u \cdot v_1\otimes\dots\otimes v_s)\\
& + (v_1\otimes u \cdot v_2 \otimes\dots\otimes v_s) + \dots +
(v_1\otimes\dots\otimes u \cdot v_s)
\end{split}
\end{align*}
for all $u \in \mathfrak{L}$. ($V^{(0)}$ is given the trivial action).
The $\mathfrak{L}$-action on $V$ is extended to $V^{(s)}$
by saying it acts via derivations.

Let $T^*(V)=\mathbf{k} \oplus V \oplus V^{(2)} \oplus \dots$ be
the usual tensor algebra. We will make it a graded algebra by
setting $V^{(n)}$ to have grading $2n$.
Thus $T^*(V)$ becomes a $\mathfrak{L}$-module
by the actions defined above. Let $I$ be the ideal generated by elements
$v\otimes w - w\otimes v$ for all $v,w \in V$. Then it is easy to check
that the action of $\mathfrak{L}$ preserves $I$. Thus $S^*(V)=T^*(V)/I$ is
a graded algebra which inherits an $\mathfrak{L}$ action from $T^*(V)$.
Note that the $\mathfrak{L}$ action preserves the grading of $S^*(V)$ so
that $S^i(V)$ is a $\mathfrak{L}$-module for all $i \geq 0$. It is well
known that $S^*(V)$ is a polynomial algebra on $k$ variables where
$k=\dime(V)$. Fix $u \in \mathfrak{L}$; then for $v_1,\dots,v_s \in V$ we have
$$
u \cdot (v_1\dots v_s) = \sum_{i=1}^sv_1\dots v_{i-1} \cdot
(u\cdot v_i) \cdot v_{i+1} \dots v_s, 
$$
so $u$ acts as a derivation on $S^*(V)$.
Given a $\mathfrak{L}$-module $M$ let $\wedge^*(\mathfrak{L},M)$ be the usual
Koszul resolution for $H^*(\mathfrak{L};M)$. Thus $\wedge^i(\mathfrak{L},M)$
consists of $M$-valued alternating $i$-forms on $\mathfrak{L}$ and
\begin{align*}
\begin{split}
(d\omega)(x_0,\dots,x_s) =& \sum_{i < j}(-1)^{i+j}\omega([x_i,x_j],x_0,\dots
\Hat{x_i},\dots,\Hat{x_j},\dots,x_s) \\
&+ \sum_{i=0}^s(-1)^ix_i\cdot\omega(x_0,\dots,\Hat{x_i},\dots,x_s)
\end{split}
\end{align*}
for all $\omega \in \wedge^s(\mathfrak{L},M)$ and all $x_0,\dots,x_s \in \mathfrak{L}$.

Consider
$$
H = \bigoplus_{i=0}^{\infty} \wedge^*(\mathfrak{L},S^i(V)),
$$
the direct
sum of Koszul complexes equipped with the differential $d$ which is the direct sum
of the differentials of each of the complexes involved. Using the algebra
structure of $S^*(V)$ one can give $H$ an algebra structure with grading
where the grading of $\wedge^1(\mathfrak{L},S^0(V))$ is 1 and that of
$\wedge^0(\mathfrak{L},S^1(V))$ is 2. With these conventions
$H$ is isomorphic to $\wedge^*(x_1,\dots,x_n) \otimes \mathbf{k}[s_1,\dots,s_k]$
where $n=\dime(\mathfrak{L})$, $x_1,\dots,x_n$ a dual basis to a basis of
$\mathfrak{L}$, $k=\dime(V)$ and $s_1,\dots,s_k$ a basis of $V$.
We will show that the differential $d$ on $H$ is a derivation with respect to
this algebra structure. Thus
$$
d(uv)=(du)v + (-1)^{\degr(u)}u(dv)
$$
for $u,v$ homogeneous elements in $H$. First note that $d$ is a derivation
on $\wedge^*(\mathfrak{L},S^0(V))$. This is because this is the same as
$d$ for $\wedge^*(\mathfrak{L},\mathbf{k})$ which is well known to be a derivation.  
Also on
$\wedge^0(\mathfrak{L},S^*(V))=S^*(V)$ we have
\begin{eqnarray*}
d(s_{i_1}\dots s_{i_m})(u) &=& u \cdot (s_{i_1} \dots s_{i_m}) \\
&=& \sum_{l=1}^m s_{i_1} \dots u\cdot s_{i_l} \dots s_{i_m} \\
&=& \left[\sum_{l=1}^m s_{i_1} \dots (ds_{i_l}) \dots s_{i_m}\right](u)
\end{eqnarray*}
for all $u \in \mathfrak{L}$. Thus $d$ is a derivation on
$\wedge^0(\mathfrak{L},S^*(V))$.
Now the elements
$$
x_{\mu_1}\dots x_{\mu_t}s_{\lambda_1}\dots s_{\lambda_m}
$$
for $\mu_1 < \dots < \mu_t$ and $\lambda_i$ arbitrary form a basis for
$H$.
Let $x_{\bar{\mu}}=x_{\mu_1}\dots x_{\mu_t}$ and
$s_{\bar{\lambda}}=s_{\lambda_1}\dots s_{\lambda_m}$. We also say
$l(\bar{\lambda})=m$ etc. So using this sort
of notation let $x_{\bar{\mu}}s_{\bar{\lambda}}$ and
$x_{\bar{\eta}}s_{\bar{\kappa}}$ be two arbitrary basis elements.
Then
$$
d(x_{\bar{\mu}}s_{\bar{\lambda}}x_{\bar{\eta}}s_{\bar{\kappa}})
=d(x_{\bar{\mu}}x_{\bar{\eta}}s_{\bar{\lambda}}s_{\bar{\kappa}}).
$$
Suppose we have the identity
\begin{align}
\begin{split}
\label{eq: did}
d(x_{\bar{\mu}}s_{\bar{\lambda}})
=& d(x_{\bar{\mu}})s_{\bar{\lambda}}
 + (-1)^{l(\bar{\mu})}x_{\bar{\mu}}d(s_{\bar{\lambda}}),
\end{split}
\end{align}
then we have
\begin{eqnarray*}
d(x_{\bar{\mu}}s_{\bar{\lambda}}x_{\bar{\eta}}s_{\bar{\kappa}})
&=& d(x_{\bar{\mu}}x_{\bar{\eta}})s_{\bar{\lambda}}s_{\bar{\kappa}}
 + (-1)^{l(\bar{\mu}) + l(\bar{\eta})}x_{\bar{\mu}}x_{\bar{\eta}}
d(s_{\bar{\lambda}}s_{\bar{\kappa}}) \\
&=& d(x_{\bar{\mu}})x_{\bar{\eta}}s_{\bar{\lambda}}s_{\bar{\kappa}}
+ (-1)^{l(\bar{\mu})}x_{\bar{\mu}}d(x_{\bar{\eta}})s_{\bar{\lambda}}
s_{\bar{\kappa}} \\
&& + (-1)^{l(\bar{\mu}) + l(\bar{\eta})}x_{\bar{\mu}}x_{\bar{\eta}}
[d(s_{\bar{\lambda}})s_{\bar{\kappa}} + s_{\bar{\lambda}}d(s_{\bar{\kappa}})].
\end{eqnarray*}
On the other hand,
\begin{eqnarray*}
&& d(x_{\bar{\mu}}s_{\bar{\lambda}})x_{\bar{\eta}}s_{\bar{\kappa}}
+ (-1)^{l(\bar{\mu})}x_{\bar{\mu}}s_{\bar{\lambda}}
d(x_{\bar{\eta}}s_{\bar{\kappa}}) \\
&&= d(x_{\bar{\mu}})x_{\bar{\eta}}s_{\bar{\lambda}}s_{\bar{\kappa}}
+ (-1)^{l(\bar{\mu})+l(\bar{\eta})}x_{\bar{\mu}}x_{\bar{\eta}}s_{\bar{\kappa}}
d(s_{\bar{\lambda}}) 
 + (-1)^{l(\bar{\mu})}x_{\bar{\mu}}d(x_{\bar{\eta}})s_{\bar{\lambda}}
s_{\bar{\kappa}} \\
&& \text{ } + (-1)^{l(\bar{\mu}) +
l(\bar{\eta})}x_{\bar{\mu}}x_{\bar{\eta}}s_{\bar{\lambda}}d(s_{\bar{\kappa}}).
\end{eqnarray*}
Comparing the two expressions we get
\begin{align*}
\begin{split}
d(x_{\bar{\mu}}s_{\bar{\lambda}}x_{\bar{\eta}}s_{\bar{\kappa}})
=d(x_{\bar{\mu}}s_{\bar{\lambda}})x_{\bar{\eta}}s_{\bar{\kappa}}
+ (-1)^{l(\bar{\mu})}x_{\bar{\mu}}s_{\bar{\lambda}}
d(x_{\bar{\eta}}s_{\bar{\kappa}}).
\end{split}
\end{align*}
So using an easy linearity argument we see $d$ is a derivation on $H$.
So it remains to prove equation~\ref{eq: did}. To do this let $\{ e_1,\dots e_n\}$ be
a basis of $\mathfrak{L}$ with $x_i(e_j)=\delta_{i,j}$ the Kronecker delta
function. Then for $\alpha_0 < \dots < \alpha_t$ we have
\begin{eqnarray*}
d(x_{\bar{\mu}}s_{\bar{\lambda}})(e_{\alpha_0},\dots,e_{\alpha_t})
&=& \sum_{i<j}(-1)^{i+j}(x_{\bar{\mu}}s_{\bar{\lambda}})([e_{\alpha_i},
e_{\alpha_j}],e_{\alpha_0},..,\Hat{e_{\alpha_i}},..,
\Hat{e_{\alpha_j}},..,e_{\alpha_t}) \\
&& + \sum_{i=0}^t(-1)^i [x_{\bar{\mu}}e_{\alpha_i} \cdot s_{\bar{\lambda}}]
(e_{\alpha_0},\dots,\Hat{e_{\alpha_i}},\dots,e_{\alpha_t}) \\
&=& d(x_{\bar{\mu}})s_{\lambda}(e_{\alpha_0},\dots,e_{\alpha_t}) \\
&& + \sum_{i=0}^t(-1)^i [x_{\bar{\mu}}e_{\alpha_i} \cdot s_{\bar{\lambda}}]
(e_{\alpha_0},\dots,\Hat{e_{\alpha_i}},\dots,e_{\alpha_t}).
\end{eqnarray*}
So to show equation~\ref{eq: did}, it is enough to show
$$
\sum_{i=0}^t(-1)^i [x_{\bar{\mu}}e_{\alpha_i} \cdot s_{\bar{\lambda}}]
(e_{\alpha_0},\dots,\Hat{e_{\alpha_i}},\dots,e_{\alpha_t})
= (-1)^t x_{\bar{\mu}}(ds_{\bar{\lambda}})(e_{\alpha_0},\dots,e_{\alpha_t}).
$$
This is done by an easy case-by-case analysis and is left to the reader.

\subsection{Differentials}
In the last section we showed that given a finite dimensional
$\mathfrak{L}$-module $V$, the differential $d$ on
$H=\oplus_{i=0}^{\infty}\wedge^*(\mathfrak{L},S^i(V))$ given by the
direct sum of the Koszul differentials is a derivation.
We now discuss a method of
putting differentials $D : H \rightarrow H$ on $H$. This means $D$ is linear,
$D \circ D = 0$ and $D$ is a derivation with respect to the graded algebra
structure of $H$. We choose a basis $\{e_1,\dots,e_n\}$ of $\mathfrak{L}$ and
let $\{x_1,\dots,x_n\}$ be a dual basis. Let $\{s_1,\dots,s_k\}$ be a basis of
$V$. Then, as discussed before,
$$
H=\wedge^*(x_1,\dots,x_n)\otimes \mathbf{k}[s_1,\dots,s_k]
$$
as graded algebras where $\degr(x_i)=1,\degr(s_i)=2$.
We will use this notation freely in the proof of the next lemma.
\begin{lem}
\label{lem: diff}
Let $H$ be as in the paragraph above. Then if
$$
\psi : \wedge^1(\mathfrak{L},S^0(V))
\rightarrow \wedge^0(\mathfrak{L},S^1(V))
$$
is a linear map then it extends
to a unique derivation
$$
D : H \rightarrow H
$$
which vanishes on
$$
\wedge^0(\mathfrak{L},S^*(V))=S^*(V).
$$
Furthermore the derivation
is a differential i.e.\ $D \circ D = 0$.
\end{lem}
\begin{proof}
First let us define $D$ as a linear map $H \rightarrow H$. For
$\alpha_0 < \dots < \alpha_s$ define
$$
D(x_{\alpha_0}\dots x_{\alpha_s})= \sum_{l=0}^s(-1)^{l}x_{\alpha_0}\dots
\psi(x_{\alpha_l})\dots x_{\alpha_s}
$$
and
$$
D(1)=0.
$$
As $\{ x_{\alpha_0}\dots x_{\alpha_s} : -1 \leq s < n ,
\alpha_0 < \dots < \alpha_s \}$ is a basis
for $\wedge^*(\mathfrak{L},S^0(V))$ we see
we have defined a unique linear map
$$
D : \wedge^*(\mathfrak{L},S^0(V)) \rightarrow H
$$
extending $\psi$. It is routine to show (and is left to the reader)
that $D$ defined in this way is a derivation $\wedge^*(\mathfrak{L},S^0(V))
\rightarrow H$. Now we extend $D$ to $\wedge^0(\mathfrak{L},S^*(V))=S^*(V)$ by
defining it to be identically 0 on this part. Finally we define $D$
on an arbitrary basis element $x_{\bar{\alpha}}s_{\bar{\lambda}}$ of $H$ by
$$
D(x_{\bar{\alpha}}s_{\bar{\lambda}})=D(x_{\bar{\alpha}})s_{\bar{\lambda}}.
$$
Note this agrees with the parts already defined and hence defines a linear
map $D : H \rightarrow H$. It is now routine (and left to the reader) to
verify that $D$ is a derivation with respect to the graded algebra structure
on $H$. Now we argue why this is the unique extension of $\psi$ to
such a derivation which vanishes on $\wedge^0(\mathfrak{L},S^*(V))$.
Clearly the definition of $D$ on
$\wedge^*(\mathfrak{L},S^0(V))$ was forced if we want a derivation. Then as
we require $D$ to vanish on $\wedge^0(\mathfrak{L},S^*(V))$ we see that
the definition on an arbitrary basis element is also forced, 
and uniqueness follows. Note that $D$ raises grading by 1.
Now we are left only to show $D \circ D =0$ on $H$.
As $D$ is linear it is enough to check on basis elements
$x_{\bar{\alpha}}s_{\bar{\lambda}}$ of $H$. We have
\begin{eqnarray*}
D \circ D (x_{\bar{\alpha}}s_{\bar{\lambda}}) &=&
D(D(x_{\bar{\alpha}})s_{\bar{\lambda}}) \\
&=& D(D(x_{\bar{\alpha}}))s_{\bar{\lambda}} + (-1)^{l(\bar{\alpha})+1}
D(x_{\bar{\alpha}})D(s_{\bar{\lambda}}) \\
&=& D(D(x_{\bar{\alpha}}))s_{\bar{\lambda}}.
\end{eqnarray*}
So we will be done once we show $D(D(x_{\bar{\alpha}}))=0$.
For this let $\bar{\alpha}$ be $\alpha_1 < \dots
< \alpha_s$ for some $1 \leq s \leq n$. Note for $s=1$ we have
$D(x_{\alpha_1}) \in \wedge^0(\mathfrak{L},S^1(V))$ so we have
$D(D(x_{\alpha_1}))=0$. Also $D(D(1))=0$. So we will prove
$D(D(x_{\bar{\alpha}}))=0$ by induction on $s=l(\bar{\alpha})$.
Assume we have shown it for $s<l$ where $l > 1$; then
\begin{eqnarray*}
D(D(x_{\alpha_1}\dots x_{\alpha_l})) &=&
D(D(x_{\alpha_1})x_{\alpha_2}\dots x_{\alpha_l} - x_{\alpha_1}D(x_{\alpha_2}
\dots x_{\alpha_l})) \\
&=& D(D(x_{\alpha_1}))x_{\alpha_2}\dots x_{\alpha_l}
+ D(x_{\alpha_1})D(x_{\alpha_2}\dots x_{\alpha_l}) \\
&& -D(x_{\alpha_1})D(x_{\alpha_2}\dots x_{\alpha_l})
+ x_{\alpha_1}D(D(x_{\alpha_2}\dots x_{\alpha_l})) \\
&=& 0,
\end{eqnarray*}
where in the last step we used cancellation and the inductive hypothesis.
Thus by induction we are done and we have shown $D \circ D =0$ on $H$, as
desired.
\end{proof}
Now we study when the differential $D$ constructed in Lemma~\ref{lem: diff}
has the additional property $d \circ D = - D \circ d$ on $H$, where $d$
is the Koszul differential of $H$ as in section 1. Let $\{e_1,\dots,e_n\}$
be a basis for $\mathfrak{L}$ as before, $\{x_1,\dots,x_n\}$ the dual basis, 
and $c_{ij}^k=[e_i,e_j]_k$ the structure constants of the Lie algebra
$\mathfrak{L}$. Let $\{s_1,\dots,s_k\}$ be a basis of $V$ as before.
\begin{lem}
\label{lem: comdiff}
Assume the notation of the preceding paragraph.
Let
$$
\psi : \wedge^1(\mathfrak{L},S^0(V)) \rightarrow
\wedge^0(\mathfrak{L},S^1(V))
$$ be a linear map. Let $D$ be the differential
extending $\psi$ as given by Lemma~\ref{lem: diff}. Then
$d \circ D = - D \circ d$ if and only if
$$
e_j \cdot \psi(x_i) =
\sum_{l=1}^nc_{lj}^i\psi(x_l)
$$
and
$$
\sum_{l=1}^n(e_l \cdot s_t)\psi(x_l)=0
$$
for all $1 \leq i,j \leq n$ and $1 \leq t \leq k$.
\end{lem}
\begin{proof}
Since $D$ and $d$ are linear it is enough to check $d \circ D = - D \circ d$
on basis elements $x_{\bar{\alpha}}s_{\bar{\lambda}}$. So we have
$$
(d \circ D)(x_{\bar{\alpha}}s_{\bar{\lambda}}) =
d(D(x_{\bar{\alpha}})s_{\bar{\lambda}}) 
= d(D(x_{\bar{\alpha}}))s_{\bar{\lambda}} + (-1)^{l(\bar{\alpha})+1}
D(x_{\bar{\alpha}})d(s_{\bar{\lambda}}).
$$
On the other hand we have
\begin{eqnarray*}
(- D \circ d)(x_{\bar{\alpha}}s_{\bar{\lambda}}) &=&
-D(d(x_{\bar{\alpha}})s_{\bar{\lambda}} + (-1)^{l(\bar{\alpha})}
x_{\bar{\alpha}}d(s_{\bar{\lambda}})) \\
&=& - D(d(x_{\bar{\alpha}}))s_{\bar{\lambda}} +(-1)^{l(\bar{\alpha})+1}
D(x_{\bar{\alpha}})d(s_{\bar{\lambda}}) 
 - x_{\bar{\alpha}}D(d(s_{\bar{\lambda}})).
\end{eqnarray*}
So we see by comparing the two expressions that we have $d \circ D =
- D \circ d$ if we can show $ d(D(x_{\bar{\alpha}})) = -D(d(x_{\bar{\alpha}}))$
and $D(d(s_{\bar{\lambda}}))=0$.
Let us work with the second condition. Let
$s_{\bar{\lambda}}=s_{\lambda_1}\dots s_{\lambda_m}$ for some $m \geq 1$ then
\begin{eqnarray*}
D(d(s_{\bar{\lambda}})) &=& D\left(\sum_{i=1}^ms_{\lambda_1}\dots d(s_{\lambda_i})
\dots s_{\lambda_m}\right) \\
&=& D\left(\sum_{i=1}^md(s_{\lambda_i})s_{\lambda_1}\dots \Hat{s_{\lambda_i}} \dots
s_{\lambda_m}\right) \\
&=& \sum_{i=1}^mD(d(s_{\lambda_i}))s_{\lambda_1}\dots \Hat{s_{\lambda_i}} \dots
s_{\lambda_m}.
\end{eqnarray*}
So we see that to show $D(d(s_{\bar{\lambda}}))=0$ one needs only show
$D(d(s_i))=0$ for all $1 \leq i \leq k$. Now we work with the first
condition above. Suppose one has shown $d(D(x_i))=-D(d(x_i))$ for $1 \leq i
\leq n$. Then let us use induction on $t=l(\bar{\alpha})$ to show
$d(D(x_{\bar{\alpha}}))=-D(d(x_{\bar{\alpha}}))$. So by hypothesis we
have this for $t=1$ so assume we have shown it for $t < l$ for some
$l >1$. Then one has:
\begin{eqnarray*}
-D(d(x_{\alpha_1}\dots x_{\alpha_l})) &=&
-D(d(x_{\alpha_1})x_{\alpha_2}\dots x_{\alpha_l} - x_{\alpha_1}d(x_{\alpha_2}
\dots x_{\alpha_l})) \\
&=& -D(d(x_{\alpha_1}))x_{\alpha_2}\dots x_{\alpha_l} - d(x_{\alpha_1})D(
x_{\alpha_2}\dots x_{\alpha_l}) \\
&& + D(x_{\alpha_1})d(x_{\alpha_2}\dots x_{\alpha_l})
- x_{\alpha_1}D(d(x_{\alpha_2}\dots x_{\alpha_l})) \\
&=& d(D(x_{\alpha_1}))x_{\alpha_2}\dots x_{\alpha_l} - d(x_{\alpha_1})D(
x_{\alpha_2}\dots x_{\alpha_l}) \\
&& + D(x_{\alpha_1})d(x_{\alpha_2}\dots x_{\alpha_l})
+ x_{\alpha_1}d(D(x_{\alpha_2}\dots x_{\alpha_l})) \\
&=& d(D(x_{\alpha_1})x_{\alpha_2}\dots x_{\alpha_l} - x_{\alpha_1}D(
x_{\alpha_2}\dots x_{\alpha_l})) \\
&=& d(D(x_{\alpha_1}\dots x_{\alpha_l})).
\end{eqnarray*}
Thus by induction we have the second condition. We conclude from all
the above that if we have $d(D(x_i))=-D(d(x_i))$ for all $1 \leq i \leq n$
and $D(d(s_i))=0$ for all $1 \leq i \leq k$ then we have $d \circ D =
- D \circ d$ on $H$ as desired. Let us translate the condition
$d(D(x_i))=-D(d(x_i))$ noting that
$D(x_i)=\psi(x_i) \in \wedge^0(\mathfrak{L},S^1(V))=S^1(V)$.
Thus $d(D(x_i))(e_j)=e_j \cdot \psi(x_i)$. On the other hand, 
\begin{eqnarray*}
-D(d(x_i))(e_j) &=& D\left(\sum_{l < m}c_{lm}^ix_lx_m\right)(e_j) \\
&=& \sum_{l<m}c_{lm}^i[D(x_l)x_m - x_lD(x_m)](e_j) \\
&=& \sum_{l<j}c_{lj}^i\psi(x_l) -\sum_{j<m}c_{jm}^i\psi(x_m) \\
&=& \sum_{l<j}c_{lj}^i\psi(x_l) +\sum_{j<m}c_{mj}^i\psi(x_m) \\
&=& \sum_{l=1}^nc_{lj}^i\psi(x_l).
\end{eqnarray*}
So $d(D(x_i))=-D(d(x_i))$ for all $1 \leq i \leq n$ if and only if
$$
e_j \cdot \psi(x_i) = \sum_{l=1}^nc_{lj}^i\psi(x_l)
$$
for all $1 \leq i,j \leq n$. Now let us translate the condition
$D(d(s_i))=0$. Note $d(s_i) \in \wedge^1(\mathfrak{L},S^1(V))$ is a 1-form
on $\mathfrak{L}$ with values in $S^1(V)=V$ so as $d(s_i)(u)=u \cdot s_i$ for
all $u \in \mathfrak{L}$ we have $d(s_i)=\sum_{j=1}^nd(s_i)(e_j)x_j=
\sum_{j=1}^n(e_j \cdot s_i)x_j$. Thus
$$
D(d(s_i)) = D\left(\sum_{j=1}^n(e_j \cdot s_i)x_j\right) 
= \sum_{j=1}^n(e_j \cdot s_i)\psi(x_j)
$$
where we have used that $D(e_j \cdot s_i)=0$ as $e_j \cdot s_i \in \wedge^0(
\mathfrak{L},S^1(V))$. Thus $D(d(s_i))=0$ for all $1 \leq i \leq k$ if and
only if $\sum_{j=1}^n(e_j \cdot s_i)\psi(x_j)=0$ for all $1 \leq i \leq k$.
This ends the proof of the lemma.
\end{proof}
\subsection{The spectral sequence}
Now we apply the results of the last section to obtain a spectral sequence
for Lie algebra cohomology. Let $\mathfrak{L}$ be a Lie algebra over the
PID $\mathbf{k}$, and $V=ad^*$ be the $\mathfrak{L}$-module which is
the dual space $\mathfrak{L}^*$ of $\mathfrak{L}$ with action
$$
(u \cdot \alpha)(v)=\alpha([v,u])
$$
for $u,v \in \mathfrak{L}$ and $\alpha \in \mathfrak{L}^*$. It is easy to check
using the Jacobi identity that this is indeed a $\mathfrak{L}$-module.
Fix a basis $\{e_1,\dots,e_n\}$ of $\mathfrak{L}$. Let $\{x_1,\dots,x_n\}$ be
the dual basis viewed in $\wedge^1(\mathfrak{L},S^0(V))=\wedge^1(\mathfrak{L},\mathbf{k})$
and $\{s_1,\dots,s_n\}$ the dual basis viewed in $\wedge^0(\mathfrak{L},S^1(V))=V$.
Then let $\psi : \wedge^1(\mathfrak{L},S^0(V)) \rightarrow
\wedge^0(\mathfrak{L},S^1(V))$ be the linear map given by taking $x_i$ to $s_i$
for all $1 \leq i \leq n$. Recall $H$ and $d$ from the last sections.
By Lemma~\ref{lem: diff} one has $\psi$ extends to a differential $D$ on
$H$. Note we have
$$
e_j \cdot s_i = \sum_{l=1}^n[e_l,e_j]_is_l = \sum_{l=1}^nc_{lj}^is_l,
$$
where the first equality is verified by noting both sides evaluate to the
same thing when evaluated on $e_k$ for $1 \leq k \leq n$ and the second
equality uses the definition of the structure constants $c_{lj}^i$.
Thus as $\psi(x_i)=s_i$ we have
$$
e_j \cdot \psi(x_i)=\sum_{l=1}^nc_{lj}^i\psi(x_l).
$$
We also have
$$
\sum_{j=1}^n(e_j \cdot s_i)\psi(x_j) = \sum_{j=1}^n(e_j\cdot s_i)s_j 
= \sum_{j,l=1}^nc_{lj}^is_ls_j 
= 0.
$$
Thus we have verified the conditions of Lemma~\ref{lem: comdiff} so we conclude
$d \circ D = - D \circ d$.
Thus we have two commuting differentials $D,d$ on $H=\oplus_{i=0}^{\infty}
\wedge^*(\mathfrak{L},S^*(ad^*))$ which we note is bigraded.
Note $S^i(ad^*)$ is isomorphic to the
vector space of symmetric $i$-forms on $\mathfrak{L}$. We will use $S^i=S^i(ad^*)$
for short. Recall we can view
$$
H=\wedge^*(x_1,\dots,x_n) \otimes \mathbf{k}[s_1,\dots,s_n]
$$
as graded algebras where $\degr(x_i)=1,\degr(s_i)=2$ and when
we do this $D$ is a derivation. But $D(x_i)=s_i$ for all $1 \leq i \leq n$
by construction so we see that the cohomology of the complex $(H,D)$ is
concentrated in the zero grading where it is $\mathbf{k}$. (This follows from
K\"unneth's theorem for example). Now let
$$
F_0^{s,q}=\wedge^{q-s}(\mathfrak{L},S^s)
$$
for $s,q \geq 0$. Then we see $d$ gives a vertical differential
$F_0^{s,q} \rightarrow F_0^{s,q+1}$ and $D$ gives a horizontal
differential $F_0^{s,q} \rightarrow F_0^{s+1,q}$ and $D,d$ commute
(up to sign). Such a complex is called a double complex. If we let $M$
be the graded complex with $M^{\ell} = \oplus_{s+q=\ell}F_0^{s,q}$ and differential
$T=D+d$ then one gets by a standard construction two first quadrant $F_0$
spectral sequences
with $F_0^{*,*}$ given by the formula above converging
to the cohomology of $(M^*,T)$. (See \cite{Brown}).
One of these spectral sequences has $F_1^{*,*}$ equal to the cohomology of
$F_0$ with respect to $D$. However we have argued before that this is acyclic
so $F_1^{s,q}=0$ for $s+q>0$ and $F_1^{0,0}=\mathbf{k}$. Thus
$F_{\infty}^{s,q}=0$ for $s+q>0$ and $F_{\infty}^{0,0}=\mathbf{k}$ for this
spectral sequence. Since this abuts to the cohomology of $M$ we conclude
that $(M^*,T)$ is an acyclic complex. (This means that the only nonzero
cohomology is in the zero grading and $H^0(M^*;\mathbf{k})=\mathbf{k}$).
The other spectral
sequence has $F_1^{*,*}$ equal to the cohomology of $F_0$ with respect
to $d$. Thus it is easy to see $F_1^{s,q}=H^{q-s}(\mathfrak{L},S^s)$ for
$s,q \geq 0$. However this spectral sequence also abuts to the cohomology
of $M^*$ so we conclude it also has $F_{\infty}^{s,q}=0$ for $s+q>0$ and
$F_{\infty}^{0,0}=\mathbf{k}$. Also as $D$ and $d$ were derivations one
concludes as usual that all the differentials $d_r$ in this spectral sequence
are derivations with respect to the induced graded algebra structures on the
$F_r$. We summarize in the following theorem:
\begin{thm}
\label{thm: specseq}
Fix an arbitrary PID $\mathbf{k}$. Let $\mathfrak{L}$ be a 
$\mathbf{k}$-Lie algebra. Let $S^s$ be the $\mathfrak{L}$-module of
symmetric $s$-forms on $\mathfrak{L}$ with the dual adjoint action discussed before.
Then there is a first quadrant $F_0$-spectral sequence with
$F_0^{s,q}=\Lambda^{q-s}(\mathfrak{L},S^s)$ and $F_1^{s,q}=H^{q-s}(\mathfrak{L},S^s)$ for $s,q \geq 0$ which has
$F_{\infty}^{s,q}=0$ for $s+q>0$ and $F_{\infty}^{0,0}=\mathbf{k}$.
Furthermore all the differentials $d_r: F_r^{s,q} \to F_r^{s+r,q-r+1}$ are derivations with respect
to the induced graded algebra structure on $F_r$.
\end{thm}

In the above spectral sequence, $s$ denotes the polynomial degree, $q-s$ represents the exterior degree
and $(q-s)+2s=q+s$ represents the total degree. This unusual grading is used so that in the derivation,
the initial differentials would be horizontal and vertical. However, this bigrading is not so easy to
work with in applications so finally we will tweak the definition of the bigrading of the above spectral
to:

$$
E_0^{s,t}=F_0^{s,s+t}.
$$
This yields the final form of the main spectral sequence that we use in this paper:

\begin{thm}
Fix an arbitrary PID {\bf k} and ${\bf k}$-Lie algebra $\mathfrak{L}$. Let $S^s$
be the $\mathfrak{L}$-module of symmetric degree $s$-polynomials on $\mathfrak{L}$,
equipped with the dual adjoint action.

Then there is a first quadrant spectral sequence
$$
E_0^{s,t}=\Lambda^t(\mathfrak{L},S^s)
$$
with differentials $d_r: E_r^{s,t} \to E_r^{s+r,t-(2r-1)}$ which are derivations
with respect to the algebra structure induced from
$$
E_0^{*,*} = \Lambda^*(\mathfrak{L}^*) \otimes S^*(\mathfrak{L}^*).
$$
The number $s$ is called the polynomial (or Hodge) degree, $t$ the exterior degree and $2s+t$ the total degree.
Thus note the differential $d_r$ raises polynomial degree by $r$ while raising total degree by one
(and hence lowers exterior degree by $(2r-1)$).

The spectral sequence converges to the cohomology of a point. Explicit formulas for the differentials
are stated within the paper itself and are easily derived from the construction in the appendix.
\end{thm}
\label{thm: SpectralSequence}

\bigskip

\noindent
Dept. of Mathematics \\
University of Rochester, \\
Rochester, NY 14627 U.S.A. \\
E-mail address: jonpak@math.rochester.edu \\

\bigskip

\noindent
Dept. of Mathematics \\
University of Rochester, \\
Rochester, NY 14627 U.S.A. \\
E-mail address: rogers@math.rochester.edu \\

\end{document}